\documentclass[11pt]{article}
\usepackage[normalem]{ulem}

\usepackage[title, titletoc]{appendix}

\usepackage[margin=1in]{geometry}
\usepackage{enumerate}
\usepackage{amsfonts}
\usepackage{amsmath}
\usepackage{amssymb}
\usepackage{mathrsfs}
\usepackage{bbm}
\usepackage{amsthm}
\usepackage{xcolor}
\usepackage{multirow}
\usepackage{array}
\usepackage{indentfirst}
\usepackage{cite}
\usepackage{verbatim}

\theoremstyle{definition}
\newtheorem{theorem}{Theorem}[section]
\newtheorem{example}[theorem]{Example}
\newtheorem{definition}[theorem]{Definition}
\newtheorem{corollary}[theorem]{Corollary}
\newtheorem{proposition}[theorem]{Proposition}
\newtheorem{lemma}[theorem]{Lemma}
\newtheorem{remark}[theorem]{Remark}

\numberwithin{equation}{section}
\newtheorem{assumption}[theorem]{Assumption}

\newcounter{cnstcnt}
\newcommand{\newconstant}{
    \refstepcounter{cnstcnt}
    \ensuremath{c_{\thecnstcnt}}}
\newcommand{\oldconstant}[1]{\ensuremath{c_{\ref{#1}}}}

\newcommand{\bpar}[1]{\big( #1 \big)}
\newcommand{\Bpar}[1]{\Big( #1 \Big)}
\newcommand{\bsq}[1]{\big[ #1 \big]}
\newcommand{\Bsq}[1]{\Big[ #1 \Big]}
\newcommand{\lan}[1]{\langle #1 \rangle}
\usepackage{hyperref}
\hypersetup{
    colorlinks = true,
    linkcolor = blue,
    filecolor = blue,      
    urlcolor = blue,
    citecolor = blue
}
\newcommand{\R}{\mathbb{R}}

\usepackage{enumerate}
\usepackage{enumitem}

\makeatletter
\newcommand{\opnorm}[1]{\vert\kern-0.25ex\vert\kern-0.25ex\vert #1 
    \vert\kern-0.25ex\vert\kern-0.25ex\vert}
\makeatother

\makeatletter
\def\@tocline#1#2#3#4#5#6#7{\relax
  \ifnum #1>\c@tocdepth 
  \else
    \par \addpenalty\@secpenalty\addvspace{#2}%
    \begingroup \hyphenpenalty\@M
    \@ifempty{#4}{%
      \@tempdima\csname r@tocindent\number#1\endcsname\relax
    }{%
      \@tempdima#4\relax
    }%
    \parindent\z@ \leftskip#3\relax \advance\leftskip\@tempdima\relax
    \rightskip\@pnumwidth plus4em \parfillskip-\@pnumwidth
    #5\leavevmode\hskip-\@tempdima
      \ifcase #1
       \or\or \hskip 1em \or \hskip 2em \else \hskip 3em \fi%
      #6\nobreak\relax
    \hfill\hbox to\@pnumwidth{\@tocpagenum{#7}}\par
    \nobreak
    \endgroup
  \fi}
\makeatother

\allowdisplaybreaks

\usepackage{makecell}

\title{Complexity of the $p$-spin Hamiltonian \\with a Non-Rotationally Invariant Potential}
\author{Wei-Kuo Chen\thanks{Email: wkchen@umn.edu. Partly supported by NSF grants DMS-1752184 and DMS-2246715 and Simons Foundation grant 1027727 } \and  Te-Lun Lu \thanks{Email: lu000646@umn.edu.}  \and Arnab Sen \thanks{Email: arnab@umn.edu. Partly supported by Simons Foundation MP-TSM-00002716}} 
\date{\today}

\begin{document}

\maketitle

\begin{abstract}
    We investigate the complexity of the Hamiltonian in the pure $p$-spin spin glass model accompanied with a polynomial-type potential on $\mathbb{R}^N$. In this Hamiltonian,  the Gaussian field is anisotropic, and the potential lacks rotational invariance. Our main result derives the logarithmic limit for the expected number of critical points in terms of a variational formula. As a consequence, by identifying the critical location of the phase transition from our representation, we provide an upper bound for the ground state energy of the model.
\end{abstract}

\tableofcontents

\section{Introduction and Main Results}\label{sec:introduction}

\noindent The study of complexity for the Gaussian Hamiltonians is one of the central topics in spin glass theory. By analyzing the growth rate of the number of critical or saddle points, it provides a quantitative framework to describe the energy landscape of the systems. This line of research was initialized from the breakthrough papers \cite{Fyodorov07,Fyodorov04} that utilized the Kac-Rice formula to investigate the complexity of isotropic Gaussian fields on $\R^N$ with rotationally invariant potentials. Since then, the Kac-Rice formula has become a standard tool in the study of the complexity of spin glasses and has led to some major advances in recent years, see, e.g., \cite{Auffinger13,auffinger2013random,subag2017complexity}.

In this work, we aim to investigate the complexity of a anisotropic Gaussian field defined on $\mathbb{R}^N$ accompanied with a non-rotationally invariant potential. More precisely, fixing an integer $p\geq 2$, we are interested in the following objective, for any $N\geq 1,$
\begin{equation}\label{Equation:Hamiltonian}
    H_N(\sigma) = \frac{1}{N^{(p-1)/2}} \sum_{i_1,\dots,i_p=1}^N g_{i_1,\dots,i_p} \sigma_{i_1}\cdots\sigma_{i_p} - \sum_{i = 1}^N V(\sigma_i)
\end{equation}
for $\sigma = (\sigma_1,\dots,\sigma_N) \in \mathbb{R}^N$, where $g_{i_1,\ldots,i_p}$'s are i.i.d.\ standard normal and $V$ is a real-valued function defined on $\mathbb{R}.$
One of the main reasons for considering this Gaussian field was motivated by the study of $\ell^q$-Grothendieck problem \cite{grothendieck1956resume}, i.e.,
\begin{align}\label{deterministicGrothendieck_Origin}
    \frac{1}{N}\max_{\sigma \in \mathbb{R}^N, \|\sigma\|_q^q \leq N} \langle A_N \sigma, \sigma \rangle,
\end{align}
where $\|\cdot\|_q$ denotes the standard $\ell^q$-norm on $\mathbb{R}^N$ for $2<q<\infty$ and $A_N$ is a deterministic $N \times N$ matrix. This is a fundamental optimization problem that has drawn significant interest from analysis, combinatorics, computer science, and probability, see, e.g., \cite{lindenstrauss1968absolutely, khot2011grothendieck}. While several existing literature (see \cite{charikar2004maximizing, fiedler1973algebraic}) have attempted to approximate \eqref{deterministicGrothendieck_Origin} up to some constant orders from algorithmic perspectives, it has been shown in a recent work \cite{chen2023ℓ} that after normalization, the Gaussian $\ell^q$-Grothendieck problem, i.e., the entries of $A_N$ are i.i.d.\ centered normal with variance $1/N$, admits a Parisi-type variational representation as $N$ tends to infinity. In the analysis of \cite{chen2023ℓ} (see also \cite{dominguez2022ℓ}), one of the crucial steps relied on considering a relaxation of \eqref{deterministicGrothendieck_Origin}, namely, for any $\lambda>0,$
\begin{align}\label{deterministicGrothendieck}
 \mbox{\rm GP}_{N,q}(\lambda):= \frac{1}{N}\max_{\sigma \in \mathbb{R}^N} \Bigl(\langle A_N \sigma, \sigma \rangle-\lambda\sum_{i=1}^N|\sigma_i|^q\Bigr).
\end{align}
and the key observation that for any $N\geq 2$ and $\lambda > 0,$ \eqref{deterministicGrothendieck_Origin} is indeed equal to $C(q,\lambda)\mbox{\rm GP}_{N,q}(\lambda)^{1-2/q}$ for some universal constant $C(q,\lambda)$. Based on this relation, the work \cite{chen2023ℓ} achieved the exact limit of \eqref{deterministicGrothendieck} by deriving the Parisi formula for \eqref{deterministicGrothendieck}. In light of the latter, it gives rise to the generalized Gaussian field \eqref{Equation:Hamiltonian}.

To prepare for the statement of our main results, we introduce some notation and assumptions. Set the number of critical points of $H_N$ with energy inside a certain range by
\[\mathrm{Crt}_N(B) = \bigl|\{\sigma \in \mathbb{R}^N \mid \nabla H_N(\sigma) = \underline{0},\, H_N(\sigma) \in B\}\bigr|\]
for any Borel set $B\subseteq \mathbb{R}.$
Denote by $\mathcal{P}(\mathbb{R})$ the space of all probability measures on $\mathbb{R}$ and by $\mathcal{P}_s(\mathbb{R})$ for $0 < s < \infty$ the collection of $\mu\in \mathcal{P}(\mathbb{R})$ with $m_s(\mu):=\int_{\mathbb{R}}|x|^s\mu(dx)<\infty$. Also, denote by $\mu_{\text{Norm}}$ and $\mu_{\text{sc}}$ the standard Gaussian measure and the semicircle measure on $[-2,2]$, respectively. The following is our assumption on $V.$

\begin{assumption}\label{ass1} $V \colon \mathbb{R} \to [0, \infty)$ is an even $C^2$ function and the following conditions hold:
\begin{enumerate}[label = ($\dagger.$\arabic*)]
    \item\label{Condition:V} There exist reals $q_1,q_2$ with $p < q_1 \leq q_2$ and $\newconstant\label{Constant:Bound} > 0$ such that for all $x \in \mathbb{R}$,
    \[\oldconstant{Constant:Bound}^{-1}(|x|^{q_1} + |x|^{q_2}) \leq V(x), xV'(x), x^2V''(x) \leq \oldconstant{Constant:Bound}(|x|^{q_1} + |x|^{q_2}).\]
    \item\label{Condition:V'} There exist some real $q > p$ such that 
    \begin{equation}\label{Equation:V}
        xV'(x) \geq qV(x) \quad \text{for all} \quad x > 0.
    \end{equation}
    \item\label{Condition:V''} There exist some real $q > p$ such that
    \begin{equation}\label{Equation:V'}
        xV''(x) \geq (q - 1)V'(x) \quad \text{for all} \quad x > 0.
    \end{equation}
\end{enumerate}
The two $q$'s in the last two conditions do not need to be the same.
\end{assumption}

\begin{remark}
    \rm  The lower bound $xV'(x)$ in \ref{Condition:V} forces the critical points to be trapped inside a finite ball with respect to the averaged $\ell_{q_2}$-distance. To see this, let  $\opnorm{x}_s:=N^{-1/s}\|x\|_s$ for any $s>0.$  If $\sigma$ is a critical point, then  $\langle \sigma,\nabla H_N(\sigma)\rangle=0,$ hence 
    \[p^{-1}\lan{\sigma, V'(\sigma)}=N^{-(p-1)/2}\sum_{i_1,\ldots,i_p=1}^Ng_{i_1,\ldots,i_p}\sigma_{i_1}\cdots\sigma_{i_p}\leq N C_{N,p} \opnorm{\sigma}_2^p\leq N C_{N,p} \opnorm{\sigma}_{q_2}^p\]
 where
     \begin{align*}
        C_{N,p}:=\frac{1}{N} \max_{\opnorm{x}_2\leq 1} \frac{1}{N^{(p-1)/2}}\sum_{i_1,\ldots,i_p=1}^Ng_{i_1,\ldots,i_p}x_{i_1}\cdots x_{i_p},
    \end{align*}
which is the ground state energy density of the spherical pure $p$-spin model. It is almost surely bounded (in fact convergent) as $N\to\infty$, see \cite{chen2017parisi}.

    On the other hand, from \ref{Condition:V}, it follows  that $ \lan{\sigma, V'(\sigma)}  \ge  \sum_{i=1}^N c_1^{-1} |\sigma_i|^{q_2} = c_1^{-1} N    \opnorm{\sigma}_{q_2}^{q_2}$. Combining the two bounds yields
 $\opnorm{\sigma}_{q_2}^{q_2}\leq pc_1 C_{N,p} \opnorm{\sigma}_{q_2}^p
    $ and hence equivalently, 
    \[ \opnorm{\sigma}_{q_2}\leq (pc_1C_{N,p})^{1/(q_2-p)}.\]
    \end{remark}

For all $u \geq 0$, we set
\begin{equation}\label{Equation:D(u)}
    \mathfrak{D}(u) = \Big\{\mu \in \mathcal{P}(\mathbb{R}) \,\Big\vert\, \mathbb{E}_\mu\big[p^{-1}XV'(X) - V(X)\big] \geq u\Big\},
\end{equation}
where $X$ is distributed according to $\mu$ and $\mathbb{E}_\mu$ is the expectation with respect to the same measure. For $0\leq t<\infty$, let $g_t:\R\to \R$ be defined as
\begin{equation}
    g_t(x) = \begin{cases}
       \frac{ t^{2 - p}V''(x)}{\sqrt{p(p - 1)}}, & \mbox{$x\in \R$ when $t>0$},\\
        0, & \mbox{$x\in \R$ when $t=0$.}
    \end{cases}
\end{equation}
Consider the functional $\varphi$ on $[0, \infty) \times \mathcal{P}_{2q_2 - 2}(\mathbb{R})$ defined by \begin{align*}
    \varphi(t, \mu) &= \frac{t^{-2p}}{2p^2}\Big((p - 1)\mathbb{E}_\mu\big[XV'(X)\big]^2 - pt^2\mathbb{E}_\mu\big[V'(X)^2\big]\Big) +\int_{\mathbb{R}} \log|\lambda| \big((g_t)_\ast\mu \boxplus \mu_{\text{sc}}\big)(d\lambda)
\end{align*}
for $(t,\mu)\in (0,\infty)\times \mathcal{P}_{2q_2-2}(\R)$ and $\varphi(0,\cdot)\equiv-1/2$ on $\mathcal{P}_{2q_2-2}(\mathbb{R}),$ where $(g_t)_*\mu:=\mu\circ g_t^{-1}$ is the push-forward measure  of $\mu$ respect to $g_t$ and $\boxplus$ is the additive free convolution (see Appendix \ref{Section:Free_Convolution}). Finally, let KL be the Kullback-Leibler divergence. Set the functional $\mathcal{I}$ on $[0, \infty) \times \mathcal{P}_{2q_2 - 2}(\mathbb{R})$ by
\begin{align}
    \mathcal{I}(t, \mu) &= \frac{1}{2}\log(p - 1) + \frac{1}{2} + \varphi(t, \mu) - \text{KL}(\mu \,\| \,\mu_{\text{Norm}}) - \frac{1}{2}(1 - t^2 + 2\log t). \label{Equation:Funtional_with_t}
    \end{align}
    When $t=\sqrt{m_2(\mu)},$ we adapt the abbreviation, 
\begin{equation}\label{Equation:Funtional_no_t}
    \mathcal{I}(\mu) = \mathcal{I}\big(\sqrt{m_2(\mu)}, \mu\big).
\end{equation}
Our main result below establishes the limit for the logarithmic expected number of critical points when the energy level is at least $u\geq 0.$

\begin{theorem}\label{Theorem:Main_Theorem}
   Let $p\geq 2$ be an integer. Assume that $V$ satisfies Assumption \ref{ass1}. For all $u \geq 0$, we have the variational formula,
    \begin{equation}\label{Equation:Main}
        \Sigma(u) = \lim_{N \to \infty}\dfrac{1}{N}\log \mathbb{E}[\text{Crt}_N([Nu, \infty))] = \sup\{\mathcal{I}(\mu) \mid \mu \in \mathcal{P}_{2q_2 - 2}(\mathbb{R}) \cap \mathfrak{D}(u)\}.
    \end{equation}
    Moreover, 
    \[\lim_{N \to \infty}\dfrac{1}{N}\log \mathbb{E}[\text{Crt}_N(\mathbb{R})] = \lim_{N \to \infty}\dfrac{1}{N}\log \mathbb{E}[\text{Crt}_N([0, \infty))] = \sup\{\mathcal{I}(\mu) \mid \mu \in \mathcal{P}_{2q_2 - 2}(\mathbb{R})\}.\]
\end{theorem}

The following proposition ensures the finiteness of \eqref{Equation:Main}.

\begin{proposition}\label{Proposition:Finite_of_I}
    If $V$ satisfies Assumption \ref{ass1}, then the supremum in \eqref{Equation:Main} is finite for all $u\geq 0.$
\end{proposition}

Similar to previous literature on the analysis of complexity, the expression \eqref{Equation:Main} is established through the Kac-Rice formula \eqref{Equation:Main} that states that the expected complexity can be written as a high-dimensional integral of the Hessian matrix of $H_N$ against the density of the critical points. In view of this, the measures $\mu$ in \eqref{Equation:Main} can be heuristically interpreted as the empirical distributions induced by the unit vectors $\sigma/\|\sigma\|_2$ associated to the critical points $\sigma$, while the functional $\mathcal{I}$ is resulted from a number of highly nontrivial large deviation argument mainly for treating the Hessian determinant term. 

As a consequence of Theorem \ref{Theorem:Main_Theorem},  we can asymptotically bound the ground state energy of $H_N$, $$
u_N = \frac{1}{N}\sup_{\sigma \in \mathbb{R}^N} H_N(\sigma),
$$
 from above by the smallest energy level so that the complexity formula \eqref{Equation:Main} is negative.

\begin{proposition}\label{Proposition:Critical_Level}
   Let $p \geq 2$ be an integer. If $V$ satisfies Assumption \ref{ass1}, then the constant
   \begin{align*}
       u_c := \inf\{u \geq 0 \mid \Sigma(u) < 0\}
   \end{align*}
   satisfies
   \begin{equation}\label{Equation:Critical_Level_1}
       0 < u_c < \infty
   \end{equation}
    and in probability,
    \begin{equation}\label{Equation:Critical_Level_2}
        \limsup_{N \to \infty} u_N \leq u_c.
    \end{equation}
\end{proposition}

\begin{example} Let $p\geq 2$ be a fixed integer and $q_1,q_2$ be two reals with $p<q_1<q_2.$ Consider distinct real numbers $q_1\leq r_1,\ldots,r_k\leq q_2$ with $r_1=q_1$ and $r_k=q_2$ and an arbitrary sequence of positive constants $c_1,\ldots,c_k$. Then
    \[V(x) = \sum_{i=1}^k c_i|x|^{q_i} \]
    satisfies Assumption \ref{ass1}.
    We can also allow some negative coefficients in the middle, such as, $V(x) = x^4 - x^5 + x^6$. Then Assumption \ref{ass1} is satisfied with the choice, $p = 2$, $q = 3$, $q_1 = 4$, $q_2 = 6$, and the constant $\oldconstant{Constant:Bound} = 2$.
\end{example}

\begin{table}[h]
    \centering
    \renewcommand{\arraystretch}{1.3}
    \begin{tabular}{|c|c|c|c|}
        \hline
        Paper & Space & Gaussian Field $X_N(\sigma)$ & $U_N(\sigma)$ \\ \hline
        \cite{Fyodorov04} & $\mathbb{R}^N$ & Isotropic & $\mu\|\sigma\|_2^2$, $\mu \in \mathbb{R}$ \\ \hline
        \cite{Fyodorov07} & $\mathbb{R}^N$ & Isotropic & $V(\opnorm{\sigma}_2^2)$, $V$ is strictly increasing, convex \\ \hline
        \cite{auffinger2013random} & $\sqrt{N}\mathbb{S}^{N - 1}$ & Pure $p$-spin (Isotropic) & $0$ \\ \hline
        \cite{Auffinger13} & $\sqrt{N}\mathbb{S}^{N - 1}$ & Mixed $p$-spin (Isotropic) & $0$ \\ \hline
        \cite{auffingerZeng2023complexity} & $B_N \subseteq \mathbb{R}^N$ & Isotropic Increments & $\mu\|\sigma\|_2^2$, $\mu \in \mathbb{R}$ \\ \hline
        \cite{ben2024landscape} & $\mathbb{R}^N$ & Isotropic & $\lan{D_N\sigma, \sigma}$, $D_N$ is positive semidefinite \\ \hline
        \cite{belius2024determinant} & $[-1, 1]^N$ & $2$-spin (Anisotropic) & $\sum_{i = 1}^N \sigma_i + \text{Ent}(\sigma) + N(1 - \opnorm{\sigma}_2^2)^2$ \\ \hline  
        Ours & $\mathbb{R}^N$ & Pure $p$-spin (Anisotropic) & $-\sum_{i = 1}^N V(\sigma_i)$, $V \geq 0$, even, polynomial-like  \\ \hline
    \end{tabular}
    \caption{Hamitonian $H_N(\sigma) = X_N(\sigma) + U_N(\sigma)$, where $X_N(\sigma) = H_N(\sigma) - \mathbb{E}[H_N(\sigma)]$ is a centered Gaussian field. $\mathbb{S}^{N - 1} = \{\sigma \in \mathbb{R}^N \mid \|\sigma\|_2 = 1\}$ stands for the sphere. The sequence of domain $B_N \subseteq \mathbb{R}^N$ in \cite{auffingerZeng2023complexity} is required to satisfy certain exponential grwoth condition (see \cite[Eq (1.3), (1.4)]{auffingerZeng2023complexity}) and $\text{Ent}(\sigma)$ is sum of the coin tossing entropy computed in the coordinates of $\sigma$. In the case where $X_N(\sigma)$ is isotropic, both $H_N(\sigma)$ and $\nabla^2 H_N(\sigma)$ are independent of $\nabla H_N(\sigma)$. Furthermore, their distribution will not depend on the spatial location $\sigma$.}
    \label{Table:Paper_Review}
\end{table}

We close this section with a discussion on the novelty of our work. As mentioned before, the study of spin glass complexity began from the works \cite{Fyodorov04, Fyodorov07} that analyzed the complexity of isotropic Gaussian fields on $\mathbb{R}^N$ with rotationally invariant potentials. Subsequently, \cite{auffinger2013random, Auffinger13} studied the complexity of spherical pure and mixed $p$-spin models in the absence of an external field. Notably, following these results, the concentration of the complexity was established in the spherical pure $p$-spin model in \cite{subag2017complexity}. More recently, the study of complexity was pushed forward to the Gaussian fields with isotropic increments together with a rotationally invariant quadratic potential, see \cite{auffingerZeng2023complexity}. Beyond rotationally invariant settings, \cite{ben2024landscape} studies Hamiltonians with soft spins under an anisotropic quadratic potential. See Table \ref{Table:Paper_Review} for the comparison. In spirit, our consideration of \eqref{Equation:Hamiltonian} under Assumption \ref{ass1} may be viewed as a generalization of \cite{Fyodorov04, Fyodorov07, ben2024landscape}, in which the quadratic confinement is replaced by polynomial-type potentials. Consequently, the associated Gaussian fields no longer possess rotational invariance. This change breaks rotational invariance and leads to qualitatively new features in the Kac-Rice analysis. In particular, the conditional law of the Hessian is no longer a simple shifted GOE ensemble that is independent of the spatial location, but rather a GOE matrix coupled to a configuration-dependent diagonal component (together with additional perturbative terms). As a result, one must control the determinant contribution together with the $\sigma$-dependent exponential weights in the Kac-Rice integrand uniformly over $\sigma\in\Omega(u)$, instead of reducing the first moment to a volume computation as in the rotationally invariant setting. These difficulties and the strategy to overcome them are discussed in detail in Section~\ref{sec:sketchofproof}.

Finally, while we anticipate that aspects of our approach extend to more general classes of potentials, including multi-well landscapes, Assumption~\ref{ass1} appears necessary for our arguments. At a technical level, it provides the monotonicity and tail control needed for truncation and for the uniform estimates underlying both the determinant asymptotics and the large deviation analysis.

\section{Notation and Structure}

\subsection{Notations}

\noindent This section collects the notation used throughout the paper.
\begin{itemize}
    \item For $x, y \in \mathbb{R}$, we denote $x \wedge y = \min\{x, y\}$ and $x \vee y = \max\{x, y\}$.
    \item For any function $f \colon \mathbb{R} \to \mathbb{R}$, define the vector-valued function by
    $f(\sigma) = (f(\sigma_1), \dots, f(\sigma_N)) \in \mathbb{R}^N$ for any $\sigma\in \R^N.$
    We also denote by $\underline{0}, \underline{1} \in \mathbb{R}^N$ the vector $\underline{0} = (0, \dots, 0)$ and $\underline{1} = (1, \dots, 1)$. Let $\lan{\cdot, \cdot}$ be the standard inner product on $\mathbb{R}^N$. For $1\leq s<\infty,$ denote by $\|\sigma\|_s$ the usual $\ell_s$-norm and by $\opnorm{\sigma}_s$ the averaged $\ell_s$-norm, namely $\opnorm{\sigma}_s=\|\sigma\|_s/N^{1/s}.$ Set also $\min(\sigma) = \min\{|\sigma_i| \mid 1 \leq i \leq N\}.$
    \item For a $N \times N$ matrix $A$, denote by $\text{Spec}(A) = \{\lambda_i(A)\}_{i = 1}^N$ the spectrum of $A$ and set $\|A\|_{\text{op}}$ the operator norm. When $A$ is symmetric ensuring $\text{Spec}(A) \subseteq \mathbb{R}$, we order the eigenvalues by $\lambda_1(A) \leq \cdots \leq \lambda_N(A)$ and set the empirical measure
    $\mu_A = N^{-1}\sum_{i = 1}^N \delta_{\lambda_i(A)} \in \mathcal{P}(\mathbb{R}).$
    We also denote by $G_N$ the $N \times N$ Gaussian Orthogonal Ensemble (GOE).
    \item For $\mu \in \mathcal{P}(\mathbb{R})$, define $m_\infty(\mu) = \sup\{|x| \mid x \in \text{supp}(\mu)\}$. Denote by $F_\mu(x) = \mu((-\infty, x])$ the cumulative distribution function of $\mu$. Moreover, if $\mu$ is absolutely continuous with respect to Lebesgue measure, set $f_\mu(x) = F_\mu'(x)$ to be the probability density function. We also define the dilation $S_\alpha(\mu) = \mu \circ (\alpha^{-1} \times \cdot )$.
    \item For probability measures $\mu, \nu \in \mathcal{P}(\mathbb{R})$, we define the L\'{e}vy distance, the Kolmogorov-Smirnov distance, the bounded Lipschitz distance, and the Kullback-Leibler divergence by
    \begin{align*}
        d_{\text{L}}(\mu, \nu) &= \inf\{\epsilon > 0 \mid F_\mu(x - \epsilon) - \epsilon \leq F_\nu(x) \leq F_\mu(x + \epsilon) + \epsilon\}, \\
        d_{\text{KS}}(\mu, \nu) &= \sup\{|F_\mu(x) - F_\nu(x)| \mid x \in \mathbb{R}\}, \\
        d_{\text{BL}}(\mu, \nu) &= \sup\Big\{\int_{\mathbb{R}} f(x)(\mu - \nu)(dx) \,\Big\vert\, \|f\|_{\text{Lip}} \leq 1, \|f\|_\infty \leq 1 \Big\}, \\
        \text{KL}(\mu \,\|\, \nu) &= \int_{\mathbb{R}} \log\Bpar{\frac{d\mu}{d\nu}(x)} \mu(dx).
    \end{align*}
    We also define the $s$-Wasserstein distance ($s \geq 1$) to be
    \[W_s(\mu, \nu) = \Big(\inf_{\gamma \in \Pi(\mu, \nu)} \int_{\mathbb{R}^2} |x - y|^s \,\gamma(d(x, y))\Big)^{1/s},\]
    where the infimum is taken over all couplings $\Pi(\mu, \nu)$ (measures on $\mathbb{R}^2$ with marginal $\mu$ and $\nu$). For $s = 1$, we have the well-known Kantorovich-Rubinstein duality:
    \[W_1(\mu, \nu) = \sup\Big\{\int_{\mathbb{R}} f(x)(\mu - \nu)(dx) \,\Big\vert\, \|f\|_{\text{Lip}} \leq 1\Big\}.\]
    See \cite[Chapter 6]{villani2008optimal} for a review on some basic properties of the Wasserstein distance.
    \item Denote the set of Borel measures on $\mathbb{R}$ by $\mathcal{M}(\mathbb{R})$.
    \item Denote by $\mathscr{T}$ the standard topology on $\mathbb{R}$ (and its subspaces). Denote by $\mathscr{T}_{\text{weak}}$ the weak topology on $\mathcal{M}(\mathbb{R})$ (and its subspaces); $\mathscr{W}_s$ the $s$-Wasserstein-topology ($s \geq 1$), the topology induced from the metric $W_s$, on $\mathcal{P}_s(\mathbb{R})$.
    \item Denote the surface area of $\mathbb{S}^{N - 1}$ by
    $S_{N - 1} = 2\pi^{N/2}/\Gamma(N/2).$
\end{itemize}

\subsection{Sketch of Proof: New Challenges}\label{sec:sketchofproof}

\noindent
A general technique for studying the complexity of critical points is the Kac-Rice formula (see \cite[Theorem 11.2.1]{adler2009random} for details). Roughly speaking, if $H_N$ is a sufficiently regular Gaussian field defined on a compact set $C \subseteq \mathbb{R}^N$, then for any measurable set $B \subseteq \mathbb{R}$,
\begin{equation}\label{Equation:Kac_Rice}
    \begin{split}
        &\mathbb{E}\Bigl[\bigl|\{\sigma \in C \mid \nabla H_N(\sigma) = \underline{0},\; H_N(\sigma)\in B\}\bigr|\Bigr] \\
        &= \int_C 
        \mathbb{E}\Bigl[
            \bigl|\det \nabla^2 H_N(\sigma)\bigr|
            \,\mathbbm{1}_{\{H_N(\sigma)\in B\}}
            \,\Big|\, \nabla H_N(\sigma) = \underline{0}
        \Bigr]
        f_{\nabla H_N(\sigma)}(\underline{0})\, d\sigma .
    \end{split}
\end{equation}
This approach is standard and serves as the starting point for the derivation of complexity formulas in
\cite{Fyodorov04, Fyodorov07, auffinger2013random, Auffinger13, auffingerZeng2023complexity, ben2024landscape}.
In our setting, however, the presence of a anisotropic Gaussian field together with a non-rotationally invariant potential introduces substantial additional technical challenges.

Recall that in the isotropic case, the centered Gaussian process
$X_N(\sigma)=H_N(\sigma)-\mathbb{E}[H_N(\sigma)]$ has covariance profile
\[\mathbb{E}[X_N(\sigma)X_N(\tau)]
= N \varrho \Bigl(\frac{\|\sigma-\tau\|_2^2}{2N}\Bigr),
\qquad \forall\sigma,\tau\in C\]
for some function $\varrho$. In this situation, $H_N(\sigma)$ and $\nabla^2 H_N(\sigma)$ are independent of $\nabla H_N(\sigma)$, and moreover their joint distribution does not depend on the spatial location
$\sigma$. As a consequence, the high-dimensional integral in \eqref{Equation:Kac_Rice} reduces to the
computation of the volume of $C$. This is the case in
\cite{Fyodorov04, Fyodorov07, auffinger2013random, Auffinger13, ben2024landscape}.

In contrast, this reduction no longer holds in our model, and one must analyze an additional layer of integration before obtaining the asymptotic complexity formula. A further difficulty arises from the fact that our potential $-V_N(\sigma)=\mathbb{E}[H_N(\sigma)]$ is not rotationally invariant. Consequently, unlike in
\cite{auffinger2013random, Auffinger13, auffingerZeng2023complexity}, where the Hessian behaves as a GOE matrix with an identity shift and one may apply the determinant reduction trick of \cite[Lemma~3.3]{auffinger2013random}, such a representation is unavailable in our setting.

In our model, conditioning on $\nabla H_N(\sigma)=0$ renders the Hamiltonian $H_N(\sigma)$ deterministic, as a consequence of the homogeneity of the underlying Gaussian field. Moreover, under this conditioning, the Hessian $\nabla^2 H_N(\sigma)$ is distributed as a GOE matrix plus a non-rotationally invariant diagonal component, up to a small perturbation. More precisely, with $B = [N u, \infty)$, the expression obtained from \eqref{Equation:Kac_Rice} takes the following form (up to multiplicative constants that are inconsequential for exponential asymptotics):
\begin{equation}\label{Equation:Rough_First_Moment}
    \int_{\Omega(u)} \frac{1}{\opnorm{\sigma}_2^{N+p}}
    e^{-N f_N(\sigma)}
    \mathbb{E}\bigl[|\det M_{N-1}(\sigma)|\bigr]
    \, d\sigma,
\end{equation}
where $\Omega(u)=\big\{\sigma\in\mathbb{R}^N: N^{-1}\sum_{i=1}^N\big(p^{-1}\sigma_i V'(\sigma_i)-V(\sigma_i)\big)\ge u\big\}$  and
\begin{align*}
    f_N(\sigma)
    &= (1-p)\frac{\langle \sigma, V'(\sigma)\rangle^2}{N^2\opnorm{\sigma}_2^{2p}}
       + p\frac{\opnorm{V'(\sigma)}_2^2}{\opnorm{\sigma}_2^{2p-2}}, \\
    M_{N-1}(\sigma)
    &= B(\sigma)^\top\bigl(D_N(\sigma)+y(\sigma)y(\sigma)^\top\bigr)B(\sigma)
       + \mathrm{GOE}_{N-1}, \\
D_N(\sigma) &= (p(p - 1))^{-1/2}\opnorm{\sigma}_2^{2 - p} \text{diag}\big(V''(\sigma)\big).
\end{align*}
Here, $B(\sigma) = B_N(\sigma)$ is a semi-orthogonal matrix that projects $\mathbb{R}^N$ onto $\mathbb{R}^{N-1}$, and $y(\sigma)=y_N(\sigma)$ is a vector in $\mathbb{R}^N$ that depends on $\sigma$. Their precise definitions are given immediately before Proposition~\ref{Kac-Rice-Final-Formula}.

The analysis of the logarithmic asymptotics of \eqref{Equation:Rough_First_Moment} presents two main technical difficulties. The first concerns the determinant term appearing in the integrand, while the second is the identification of the correct large deviation principle governing the remaining factors. We discuss these issues in detail below.

\medskip

\noindent \textbf{Challenge 1: Exponential asymptotics of the determinant.}
The first step of our analysis is to establish the exponential asymptotics of the determinant term appearing in
\eqref{Equation:Rough_First_Moment}. Let $\mu_{A_N}$ denote the empirical spectral measure of a symmetric matrix $A_N$. For clarity of exposition, we first ignore the additive and multiplicative perturbations $y(\sigma)$ and $B(\sigma)$. It is well known in free probability that if $\mu_{D_N(\sigma)}$ converges to a probability measure $\nu$, then the empirical spectral measure of $\mathrm{GOE}_N + D_N(\sigma)$ converges to the free convolution $\nu \boxplus \mu_{\mathrm{sc}}$. In our setting, however, there is no single limiting ``target'' measure, since $\sigma\in\Omega(u)$ is arbitrary. Instead, we compare the empirical spectral measure of $M_{N-1}(\sigma)$ to the finite-$N$ approximation $\mu_{D_N(\sigma)} \boxplus \mu_{\mathrm{sc}}$, and approximate the determinant via
\begin{equation}\label{Equation:Exp_Asy_1}
    \frac{1}{N} \log \mathbb{E}\bigl[|\det M_{N - 1}(\sigma)|\bigr]
    \approx s_N(\sigma) :=
    \int_{\mathbb{R}} \log|\lambda|\,
    \bigl(\mu_{D_N(\sigma)} \boxplus \mu_{\mathrm{sc}}\bigr)(d\lambda),
\end{equation}
which introduces additional technical difficulties due to the singularity of the logarithm at the origin. The logarithmic asymptotics of determinants of Gaussian matrices with a covariance profile were studied systematically in \cite{ben2023exponential, ben2024landscape}. However, these results do not apply directly in our setting, since they require, as an input, a quantitative finite-$N$ approximation $\mathbb{E} \mu_{M_{N - 1}(\sigma)} \approx \mu_{D_N(\sigma)} \boxplus \mu_{\mathrm{sc}}$. Such estimates are available in the random matrix literature \cite{erdHos2019random, ajanki2019stability} via delicate Green function bounds, but only under the assumption that the mean matrix $\mathbb{E}[M_{N - 1}(\sigma)]$ has uniformly bounded spectrum. In our case, $D_N(\sigma)$ has unbounded spectrum, and we additionally need the approximation to hold uniformly over configurations $\sigma\in\Omega(u)$.

To overcome this difficulty, we proceed through the following sequence of approximations:
\begin{align}
    \frac{1}{N}\log |\det M_{N - 1}(\sigma)| = \int_{\mathbb{R}} \log |\lambda| \,\mu_{M_{N - 1}(\sigma)}(d\lambda)
    &\approx
    \int_{\mathbb{R}} \log |\lambda| \,\mathbb{E}\bigl[\mu_{M_{N - 1}(\sigma)}\bigr](d\lambda)
    \label{Equation:Exp_Asy_2} \\
    &\approx
    \int_{\mathbb{R}} \log |\lambda|\,
    \mathbb{E}\bigl[\mu_{D_N^K(\sigma)+\mathrm{GOE}_N}\bigr](d\lambda)
    \label{Equation:Exp_Asy_3} \\
    &\approx
    \int_{\mathbb{R}} \log |\lambda|\,
    \bigl(\mu_{D_N^K(\sigma)} \boxplus \mu_{\mathrm{sc}}\bigr)(d\lambda)
    \label{Equation:Exp_Asy_4} \\
    &\approx
    \int_{\mathbb{R}} \log |\lambda|\,
    \bigl(\mu_{D_N(\sigma)} \boxplus \mu_{\mathrm{sc}}\bigr)(d\lambda).
    \label{Equation:Exp_Asy_5}
\end{align}
Here $K$ is a large positive constant and $D_N^K(\sigma)$ denotes the truncated diagonal matrix
\[
    D_N^K(\sigma)
    =
    (p(p-1))^{-1/2}\,\opnorm{\sigma}_2^{2-p}
    \operatorname{diag}\bigl(V''(\sigma)\wedge K\bigr),
\]
which has bounded spectrum. We now briefly outline the justification of each approximation step, emphasizing that \eqref{Equation:Exp_Asy_3} is where the main truncation argument occurs.

\begin{enumerate}
    \item 
    The concentration step \eqref{Equation:Exp_Asy_2} is established via an analogue of \cite[Theorem 1.2]{ben2023exponential} (see Theorem~\ref{Theorem:Main_Asymptote_Unbounded}). In \cite{ben2023exponential}, this approximation relies on concentration of the empirical spectral measure of $M_{N - 1}(\sigma)$ together with control of the logarithmic singularity near the origin, obtained by assuming closeness to a compactly supported reference measure whose density has sufficient regularity at zero. While concentration follows from Theorem~\ref{Theorem:GZ_Concentration} via a standard application of the logarithmic Sobolev inequality, the latter assumption does not apply to our setting. We therefore reformulate the required input as a critical Wegner estimate (see Definition~\ref{Definition:Wegner}), available from \cite{aizenman2017matrix}, which controls the small eigenvalues of matrices of the form $\mathrm{GOE}_N + A_N$ for {\em arbitrary} deterministic symmetric $A_N$.
    
    \item
    The truncation step \eqref{Equation:Exp_Asy_3} is crucial in order to justify the key step
    \eqref{Equation:Exp_Asy_4}. We first show that the domain of integration in \eqref{Equation:Rough_First_Moment}
    can be restricted to configurations with bounded normalized $(2q_2-2)$-norm
    without affecting the leading exponential order (see Proposition~\ref{Proposition:q_Truncation}) by utilizing the rapid decay of the prefactor $e^{-Nf_N(\sigma)}$ when $\sigma$ is away from the origin. Under this restriction, we may truncate the diagonal matrix $D_N(\sigma)$ and simultaneously removing negligible terms such as the semi-orthogonal matrix and the rank-one projection from $M_{N - 1}(\sigma)$ (see Proposition \ref{Proposition:Main_Truncation}). We point out that this again relies on a Wegner estimate to control the logarithmic singularity near the origin, highlighting its essential role in this analysis.

    \item 
    After truncation, $\|D_N^K(\sigma)\|_{\text{op}}$ becomes uniformly bounded, so we can invoke existing, though nontrivial, results from the literature to deduce that
    \[\mu_{D_N^K(\sigma)+\mathrm{GOE}_N} \approx \mu_{D_N^K(\sigma)} \boxplus \mu_{\mathrm{sc}},\]
    via stability of the matrix Dyson equation (see Proposition \ref{Proposition:Stability_GOE}). Finally, \eqref{Equation:Exp_Asy_4} follows once we control the logarithmic singularity at $0$ using the Wegner estimate.
    
    \item 
    The final approximation \eqref{Equation:Exp_Asy_5}, in which the truncation parameter $K$ is sent to infinity, is carried out after deriving a truncated version of the variational formula \eqref{Equation:Main}. By proving that the convergence as $K\to\infty$ is uniform over all admissible configurations in the variational problem, we obtain the full asymptotic formula \eqref{Equation:Main} (see Proposition \ref{Proposition:Approximation_of_phi}).
\end{enumerate}

We remark that the recent approach of \cite[Theorem 1.2]{belius2024determinant} also yields the approximation in \eqref{Equation:Exp_Asy_3}, since it only requires sublinear growth of the operator norm of the mean matrix, which is $D_N(\sigma)$ in our case. In our setting, however, the spectrum of $D_N(\sigma)$ is unbounded prior to truncation, and a truncation-based approach therefore seems unavoidable. Moreover, we believe that the general framework developed in \eqref{Equation:Exp_Asy_2}--\eqref{Equation:Exp_Asy_5} may be useful in future investigations involving more general types of domains of $H_N.$

\medskip

\noindent\textbf{Challenge 2: Large Deviation Principle.} We decompose the integral in \eqref{Equation:Rough_First_Moment} into radial and spherical components. After substituting the approximation \eqref{Equation:Exp_Asy_1}, we obtain
\begin{align}
    &\int_{\Omega(u)} \frac{1}{\opnorm{\sigma}_2^{N + p}} 
    e^{-Nf_N(\sigma)}\mathbb{E}\bigl[|\det M_{N-1}(
    \sigma)|\bigr] d\sigma \notag \\
    &\approx 
    \int_0^\infty \int_{\sqrt{N}\mathbb{S}^{N - 1}} \mathbbm{1}_{\{t\omega \in \Omega(u)\}}
    \exp\bigl(N[-f_N(t\omega) + s_N(t\omega)]\bigr) 
    \,d\omega \,dt \notag \\
    &= \int_0^\infty \int_{\sqrt{N}\mathbb{S}^{N - 1}} \mathbbm{1}_{\{(T, L_{t\omega, N}))\in\mathfrak{F}(u)\}}
    \exp\bigl(N[-f(t,L_{t\omega,N}) + s(t,L_{t\omega,N})]\bigr)d\omega dt,
    \label{Equation:LDP_1}
\end{align}
where the set $\mathfrak{F}(u)$ encodes the constraint $\sigma=t\omega\in\Omega(u)$ (see \eqref{Equation:F(u)}), 
\begin{align*}
    f(t,\mu) &:= - (p-1) t^{-2p} \Big(\int_{\mathbb{R}} x V'(x) \mu(dx)\Big)^2 + pt^{-(2p-2)} \int_{\mathbb{R}} (V'(x))^2\,\mu(dx), \\
    s(t,\mu) &:= \int_{\mathbb{R}} \log|\lambda| \big((g_t)_*\mu \boxplus \mu_{\rm sc}\big)(d\lambda), \quad g_t(x) := \oldconstant{Constant:A_1}t^{2-p}\,V''(x)
\end{align*}
and 
\[L_{t\omega, N} = \dfrac{1}{N}\sum_{i = 1}^N \delta_{t\omega_i} \in \mathcal{P}(\mathbb{R})\]
denotes the empirical measure associated to any $t\omega.$ 
One can show that very large and very small values of $t$ do not contribute at the exponential scale.  Thus, after restricting the $t$-integral to $[\delta,M]$ with $M$ large and $\delta>0$ small, we may write
\begin{equation}\label{Equation:LDP_2}
    \eqref{Equation:LDP_1} \approx  (M - \delta)\cdot \mathbb{E}\Bigl[
    \mathbbm{1}_{\{(T, L_{T\omega, N})\in\mathfrak{F}(u)\}}
    \exp\bigl(N[-f(T,L_{T\omega,N}) + s(T,L_{T\omega,N})]\bigr)
    \Bigr],
\end{equation}
where $T$ is uniformly distributed on $[\delta,M]$, $\omega$ is uniformly distributed on $\sqrt{N}\mathbb{S}^{N-1}$, and $T$ and $\omega$ are independent.

Formally, if the pair $(T, L_{T\omega, N})$ satisfies a large deviation principle with rate function $J(t,\mu)$, then Laplace's method suggests that the logarithmic asymptotics of \eqref{Equation:LDP_2} are given by
\begin{equation}\label{Equation:Intro_4}
    \sup_{(t,\mu)\in\mathfrak{F}(u)}
    \bigl\{-f(t,\mu) + s(t,\mu)-J(t,\mu)\bigr\}.
\end{equation}
However, implementing this strategy is considerably more delicate. Ignoring the randomness of $T$ for the moment, it is known that the empirical measure of $\omega$ satisfies an LDP with an explicit rate function in the weak topology, and that this can be upgraded to the $s$-Wasserstein topology for $s<2,$ \cite{kim2018conditional}. In our case, however, the term $\big(\int_{\mathbb{R}} xV'(x)\,L_{t\omega,N}(dx)\big)^2$ appearing in the exponential factor $-f(t,L_{t\omega,N})$ depends on high moments of the empirical measure $L_{t\omega,N}$, and such functionals are not continuous in the $s$-Wasserstein topology when $s<2$. A second difficulty is that the components of the pair $(T,L_{T\omega,N})$ are not independent, so one needs a new joint large deviation principle, which does not appear to be available in the literature.

To address the first issue, we represent $\omega=g/\opnorm{g}_2$, for $g$ a standard Gaussian vector in $\mathbb{R}^N$ and fix the radial parameter $t$. The idea is to absorb the factor 
\begin{equation}\label{eq:tilting_factor}
    \exp\Bigl(-N t^{2-2p}\int V'(x)^2 L_{t \omega, N}(dx)\Bigr)
    = \exp\Bigl(-t^{2-2p}\sum_{i=1}^N V'\Bigl(\frac{tg_i}{\opnorm{g}_2}\Bigr)^2\Bigr)
\end{equation}
appearing in $e^{-Nf(t,L_{t\omega,N})}$ into the product Gaussian measure. This tilting enforces tails that decay faster than Gaussian.  However, it destroys the product structure of the base measure because of the global normalization term $\opnorm{g}_2$. To recover a product structure, we bound \eqref{eq:tilting_factor} by replacing $\opnorm{g}_2$ with a non-random proxy $a>0$. For instance, on the event $a_1 \le \opnorm{g}_2 \le a_2$, we have
\begin{equation*}\label{eq:tilting_factor_actual}
 \prod_{i=1}^N \exp\Bigl(-t^{2-2p} V'\Bigl(\frac{tg_i}{a_1}\Bigr)^2\Bigr) \le \exp\Bigl(-t^{2-2p}\sum_{i=1}^N V'\Bigl(\frac{tg_i}{\opnorm{g}_2}\Bigr)^2\Bigr)
    \le \prod_{i=1}^N \exp\Bigl(-t^{2-2p} V'\Bigl(\frac{tg_i}{a_2}\Bigr)^2\Bigr),
\end{equation*}
where the inequalities follows from the monotonicity property of $V'$ given in Assumption \ref{ass1}. We may therefore combine the product term $\prod_{i=1}^N \exp\bigl(-t^{2-2p} V'(tg_i/a)^2\bigr)$ with the Gaussian density to obtain a tilted product measure $\mu_{a,t}^{\otimes N}$, whose marginal has tails decaying like $\exp(-c_{a, t}|x|^{2q_2-2})$. For $x\sim \mu_{a,t}^{\otimes N}$,
the restriction on $\opnorm{x}_2$ into intervals then necessitates an analysis of the joint large deviation principle for
\[\bigl(m_2(L_{x,N}),\, L_{x/m_2(L_{x,N}),N}\bigr)\]
that follows from Cram\'er’s theorem for $\bigl(m_2(L_{x,N}),L_{x,N}\bigr)$ together with the contraction principle (see Lemma~\ref{Lemma:LDP_Tilted_Pre}).

To incorporate the random variable $T$, we again exploit certain  monotonicity property in Assumption~\ref{ass1} to show that the LDP obtained above is uniform in $t$. This yields a joint large deviation principle for the triple
\begin{equation}\label{LDP_joint_triple}
  \bigl(m_2(L_{x,N}),\, L_{x/m_2(L_{x,N}),N},\, T\bigr),  
\end{equation}
now with respect to the $s$-Wasserstein topology for any $s<2q_2-2$ (see Theorem~\ref{Theorem:LDP_Tilted}),  a threshold dictated by the tail decay of the tilted measure. Under this finer topology, the remaining term in $f(t,L_{t\omega,N})$, as well as $s(t,L_{t\omega,N})$, becomes continuous. A careful application of Varadhan's lemma to the joint large deviation principle in \eqref{LDP_joint_triple} (see Theorem~\ref{Theorem:Varadhan_Tilted}) then yields the desired asymptotic formula \eqref{Equation:Intro_4}.  We remark that \cite{belius2024determinant} does not address this Laplace transform step.

The resolution of the above two challenges constitutes the main body of this paper.

\subsection{Structure of the Paper}

\noindent In Section~\ref{Section:First_Moment}, we apply the Kac-Rice formula together with tools from random matrix theory to derive the main integral representation for the complexity. Section~\ref{Section:Truncation} is of a more technical nature, where we establish truncation lemmas that allow us to identify the exponential behavior of the determinant appearing in this integral, thereby resolving Challenge~1. Section~\ref{Section:Variational_Formula} addresses  Challenge~2. It begins by deriving a large deviation principle for the triple \eqref{LDP_joint_triple} followed by establishing the logarithmic asymptotics of the complexity as a variational formula stated in Theorem~ \ref{Theorem:Main_Theorem}. Finally, Section~\ref{Section:Analysis} analyzes our formula and concludes the proofs of Propositions~\ref{Proposition:Finite_of_I} and~\ref{Proposition:Critical_Level}.

\section{First Moment of the Complexity}\label{Section:First_Moment}

\noindent In this section, we establish the mean Kac-Rice-type formula for the mean number of critical points of $H_N$. Let $p\geq 2$ be an integer and $q>p$ be an arbitrary real number. Set constants 
$$\newconstant\label{Constant:A_1} = \frac{1}{\sqrt{p(p - 1)}}, \quad \newconstant\label{Constant:A_2} = \frac{p - 1}{\sqrt{p(p - 1)}}, \quad \mbox{and} \quad \newconstant\label{Constant:Exp} = \frac{1}{2p^2}.$$
For any fixed nonzero $\sigma \in \mathbb{R}^N$, let $\{ {v}_i\}_{i = 2}^N\subset \mathbb{R}^N$ satisfy that $\{{\sigma}/{\|\sigma\|_2},  {v}_2, \dots,  {v}_N\}$ forms an orthonormal basis of $\mathbb{R}^N$. Set the $N \times (N - 1)$ semi-orthogonal matrix $B(\sigma) = \begin{pmatrix}
    v_2 & \cdots & v_N
\end{pmatrix}$ and we define for all $\sigma \in \mathbb{R}^N$
\begin{equation}\label{def:M}
    M_{N - 1}(\sigma) =  B(\sigma)^\top\Bpar{\oldconstant{Constant:A_1}\,\text{diag}\Bpar{\dfrac{V''(\sigma)}{\opnorm{\sigma}_2^{p - 2}}} - \dfrac{v(\sigma) v(\sigma)^\top}{\lan{\sigma, v(\sigma)}\opnorm{\sigma}_2^{p - 2}}} B(\sigma) +  G_{N - 1},
\end{equation}
where $G_{N - 1}$ is a $(N - 1) \times (N - 1)$ GOE and the vector
\begin{equation}\label{Equation:Vector}
    v(\sigma) = \oldconstant{Constant:A_1}\sigma V''(\sigma) - \oldconstant{Constant:A_2}V'(\sigma).
\end{equation}
We adapt the convention for $\sigma = 0$ by $M_{N - 1}(0) = G_{N - 1}$. Note that since $G_{N - 1}$ is rotational invariant, the eigenvalue distribution of $M_{N - 1}$ is invariant under different choice of $B(\sigma)$.

\begin{proposition}\label{Kac-Rice-Final-Formula}
	For any Borel set $B \subseteq \mathbb{R}$, define
    \begin{equation}\label{Equation:Omega}
        \Omega(B) = \big\{\sigma \in \mathbb{R}^N \,\big\vert\, p^{-1}\lan{\sigma, V'(\sigma)} - \lan{\underline{1}, V(\sigma)} \in B\big\}.
    \end{equation}
    If $V$ satisfies \ref{Condition:V''}, we have
    \begin{equation}
        \mathbb{E}\bigl[\text{Crt}_{N}(B)\bigr] = \frac{1}{\sqrt{p}}\Bigl(\frac{p - 1}{2\pi}\Bigr)^{N/2}\int_{\Omega(B)} \frac{|\lan{\sigma, v(\sigma)}|}{N\opnorm{\sigma}_2^{N + p}} e^{-\oldconstant{Constant:Exp}Nf_N(\sigma)}\mathbb{E}\bigl[|\det {M}_{N - 1}(\sigma)|\bigr] \,d\sigma, \label{Equation:Kac-Rice-Final-Formula}
    \end{equation}
	where the function
	\[f_N(\sigma) = (1 - p)\frac{\lan{\sigma, V'(\sigma)}^2}{N^2\opnorm{\sigma}_2^{2p}} + p\frac{\opnorm{V'(\sigma)}_2^2}{\opnorm{\sigma}_2^{2p - 2}}.\]
\end{proposition}

For the rest of this section, we establish this proposition based on the Kac-Rice formula. We begin by computing the covariance of $(H_N,\nabla H_N,\nabla^2 H_N)$ in the following subsection. 

\subsection{Covariance Structure}

\begin{proposition}\label{Proposition:Cov_Structure}
	For any fixed $\sigma \in \mathbb{R}^N$, $(H_N(\sigma),\nabla H_N(\sigma),\nabla^2 H_N(\sigma))$ is a joint Gaussian field with the following means
    \begin{align}
        \mathbb{E}[H_N(\sigma)] &= -\lan{\underline{1}, V(\sigma)}, \label{Equation:Mean_H} \\
        \mathbb{E}[\nabla H_N(\sigma)] &= -V'(\sigma), \label{Equation:Mean_Grad} \\
        \mathbb{E}\bigl[\nabla^2H_N(\sigma)\bigr] &= -\,\text{diag}(V''(\sigma)) \label{Equation:Mean_Hess}
    \end{align}
    and covariances (for all $1 \leq i, j, k, \ell \leq N$)
    \begin{align}
        \text{Var}(H_N(\sigma)) &= N^{1 - p}\|\sigma\|_2^{2p}, \label{Equation:Cov_H} \\
		\text{Cov}(\nabla H_N(\sigma), \nabla H_N(\sigma)) &= N^{-1}p\opnorm{\sigma}_2^{2(p - 2)}\bigl((p - 1)\sigma\sigma^\top + \|\sigma\|_2^2\, I_N\bigr), \label{Equation:Cov_Grad} \\
        \text{Cov}(\partial_{\sigma_i\sigma_j}H_N(\sigma), \partial_{\sigma_k\sigma_\ell}H_N(\sigma)) &= N^{1 - p}p(p - 1)\|\sigma\|_2^{2(p - 4)}\big[(p - 2)(p - 3)\sigma_i\sigma_j\sigma_k\sigma_\ell \notag \\
		&+ \|\sigma\|_2^2(p - 2)\big(\delta_{ik}\sigma_j\sigma_\ell + \delta_{i\ell}\sigma_j\sigma_k + \delta_{jk}\sigma_i\sigma_\ell + \delta_{j\ell}\sigma_i\sigma_k\big) \notag \\
        &+ \|\sigma\|_2^4(\delta_{ik}\delta_{j\ell} + \delta_{i\ell}\delta_{jk})\big], \label{Equation:Cov_Hess} \\
        \text{Cov}\bpar{H_N(\sigma), \nabla H_N(\sigma)} &= N^{1 - p}p\|\sigma\|_2^{2(p - 1)}\sigma, \label{Equation:Cov_H_Grad} \\
		\text{Cov}\bpar{H_N(\sigma), \nabla^2H_N(\sigma)} &= N^{1 - p}p(p - 1)\|\sigma\|_2^{2(p - 2)}\sigma\sigma^\top, \label{Equation:Cov_H_Hess} \\
		\text{Cov}\bigl(\partial_{\sigma_i} H_N(\sigma), \partial_{\sigma_j\sigma_k}^2H_N(\sigma)\bigr) &= N^{1 - p}p(p - 1)\|\sigma\|_2^{2(p - 3)}\big[(p - 2)\sigma_i\sigma_j\sigma_k \notag \\
        & + \|\sigma\|_2^2\big(\delta_{ik}\sigma_j + \delta_{ij}\sigma_k\big)\big]. \label{Equation:Cov_Grad_Hess}
	\end{align}
\end{proposition}

\begin{proof}
    Since the proofs are pure algebraic, we refer the readers to Appendix \ref{Section:Covariance}.
\end{proof}

For any $\sigma \in \mathbb{R}^N$, we define the following $N \times N$ matrices
\begin{align*}
    P(\sigma) &= I_N - \frac{\sigma\sigma^\top}{\|\sigma\|_2^2}, \\
    A_N(\sigma) &= \opnorm{\sigma}_2^{2 - p}\Big[\oldconstant{Constant:A_1}\,\text{diag}(V''(\sigma)) + \oldconstant{Constant:A_2}\Bigl(\dfrac{\lan{\sigma, V'(\sigma)}}{\|\sigma\|_2^4}\sigma\sigma^\top - \dfrac{\sigma V'(\sigma)^\top + V'(\sigma)\sigma^\top}{\|\sigma\|_2^2}\Big)\Big],
\end{align*}
where we adapt the convention for $\sigma = \underline{0}$ by $P(\underline{0}) = I_N$ and $A_N(\underline{0}) = O_N$.

\begin{proposition}[Conditional Distribution of $H_N$ and $\nabla^2H_N$ on $\nabla H_N$]\label{Proposition:Conditional_Distribution}
    Under the measure $\mathbb{P}(\cdot \mid \nabla H_N(\sigma) = \underline{0})$, we have
	\begin{enumerate}[label = (\roman*)]
		\item the Hamiltonian is deterministic with
        \begin{equation}\label{Equation:CD_H}
            H_N(\sigma) = p^{-1}\lan{\sigma, V'(\sigma)} - \lan{\underline{1}, V(\sigma)};
        \end{equation}
		\item $\nabla^2 H_N(\sigma)$ is a Gaussian matrix distributed as
		\begin{equation}\label{Equation:CD_Hess}
			\sqrt{p(p - 1)}\opnorm{\sigma}_2^{p - 2}\big(-A_N(\sigma) + P(\sigma)G_NP(\sigma)\bigr),
		\end{equation}
        where $G_N$ is a $N \times N$ GOE.
	\end{enumerate}
\end{proposition}

\begin{proof}
	Recall the well-known rule on how Gaussian distributions transform under conditioning, see, e.g., \cite[Eqs (1.2.7) and (1.2.8)]{adler2009random}: If $x = (x_1, x_2)^\top \in \mathbb{R}^{n + m}$ is a Gaussian vector with mean and covariance matrix
	\[\mathbb{E}[x] = \begin{pmatrix}
		 {m}_1 \\
		 {m}_2
	\end{pmatrix} \in \mathbb{R}^{n + m} \quad \text{and} \quad \text{Cov}(x) = \begin{pmatrix}
		 {C}_{11} &  {C}_{12} \\
		 {C}_{21} &  {C}_{22}
	\end{pmatrix} \in \mathbb{R}^{(n + m) \times (n + m)},\]
	then the conditional distribution of $x_1$ given $x_2$ is a Gaussian vector with mean and covariance matrix
	\begin{equation}\label{Equation:Gaussian_Condition}
		\mathbb{E}[x_1 \mid x_2] = m_1 + C_{12}C_{22}^{-1}(x_2 - m_2) \quad \text{and} \quad \text{Cov}(x_1 \mid x_2) = C_{11} - C_{12}C_{22}^{-1}C_{21}.
	\end{equation}
	To prove \eqref{Equation:CD_H}, let $x_1 = H_N(\sigma)$ and $x_2 = \nabla H_N(\sigma)$. We see from \eqref{Equation:Cov_Grad} that
	\begin{equation}\label{Equation:Cov_Grad_Inverse}
		 {C}_{22}^{-1} = N^{-1}p^{-2}\opnorm{\sigma}_2^{-2p}\bigl((1 - p)\sigma\sigma^\top + p\|\sigma\|_2^2\, {I}_N\bigr).
	\end{equation}
	Then, by \eqref{Equation:Mean_H}, \eqref{Equation:Cov_H}, \eqref{Equation:Cov_H_Grad} and \eqref{Equation:Cov_Grad_Inverse}, the conditional distribution of $H_N(\sigma)$ on $\nabla H_N(\sigma) = \underline{0}$ has mean
	\[-\lan{\underline{1}, V(\sigma)} + \bpar{p\opnorm{\sigma}_2^{2(p - 1)}\sigma}^\top C_{22}^{-1} \bpar{\underline{0} + V'(\sigma)} = p^{-1}\lan{\sigma, V'(\sigma)} - \lan{\underline{1}, V(\sigma)}\]
    and variance
	\[N\opnorm{\sigma}_2^{2p} - \bigl(p\opnorm{\sigma}_2^{2(p - 1)}\sigma\bigr)^\top {C}_{22}^{-1}\bigl(p\opnorm{\sigma}_2^{2(p - 1)}\sigma\bigr) = 0.\]
	To prove \eqref{Equation:CD_Hess}, we set $x_1 = \nabla^2H_N(\sigma)$ and $x_2 = \nabla H_N(\sigma)$. By \eqref{Equation:Mean_Hess}, \eqref{Equation:Cov_Hess}, \eqref{Equation:Cov_Grad_Hess}, and \eqref{Equation:Cov_Grad_Inverse}, we see that the conditional distribution of $\nabla^2H_N(\sigma)$ on $\nabla H_N(\sigma) = \underline{0}$ has mean 
	\begin{align*}
		&\mathbb{E}\bsq{\partial_{\sigma_i\sigma_j}^2H_N(\sigma) \,\big\vert\, \nabla H_N(\sigma) =  \underline{0}} \\
		&= -\delta_{ij}V''(\sigma_i) + \sum_{\alpha, \beta = 1}^N \text{Cov}\bigl(\partial_{\sigma_i\sigma_j}^2H_N(\sigma), \partial_{\sigma_\alpha}H_N(\sigma)\big) (C_{22}^{-1})_{\alpha \beta}\big(\underline{0} + V'(\sigma)\big)_\beta \\
		&= -\Big[\delta_{ij}V''(\sigma_i) + (p - 1)\Big(\dfrac{\sum_{k = 1}^N \sigma_kV'(\sigma_k)}{\|\sigma\|_2^4}\sigma_i\sigma_j - \dfrac{\sigma_iV'(\sigma_j) + \sigma_jV'(\sigma_i)}{\|\sigma\|_2^2}\Big)\Big] \\
		&= -\sqrt{p(p - 1)}\opnorm{\sigma}_2^{p - 2}  A_N(\sigma)_{ij},
	\end{align*}
	and  the covariance is given by
	\begin{align}
		&\text{Cov}\bpar{\partial_{\sigma_i\sigma_j}^2H_N(\sigma), \partial_{\sigma_k\sigma_\ell}^2H_N(\sigma) \,\big\vert\, \nabla H_N(\sigma) = \underline{0}} \notag \\
        &= p(p - 1)\opnorm{\sigma}_2^{2(p - 2)} \cdot \dfrac{P(\sigma)_{ik} P(\sigma)_{j\ell} +  P(\sigma)_{i\ell} P(\sigma)_{jk}}{N} \label{Equation:Condition_Cov_Hess}.
	\end{align}
	We defer the justification \eqref{Equation:Condition_Cov_Hess} to Appendix \ref{Section:Covariance}. Set $ {M} =  {P}(\sigma) {G}_N {P}(\sigma)$. We have
	\begin{align}
		\text{Cov}(M_{ij}, M_{k\ell}) &= \mathbb{E}[M_{ij}M_{k\ell}] \notag \\
		&= \sum_{\alpha, \beta, \gamma, \delta = 1}^N \mathbb{E}[P(\sigma)_{i\alpha}( {G}_N)_{\alpha\beta} {P}(\sigma)_{\beta j} {P}(\sigma)_{k\gamma}( {G}_N)_{\gamma\delta} {P}(\sigma)_{\delta\ell}] \notag \\
		&= \sum_{\alpha, \beta, \gamma, \delta = 1}^N \frac{\delta_{\alpha \gamma}\delta_{\beta \delta} + \delta_{\alpha\delta}\delta_{\beta\gamma}}{N} \cdot  {P}(\sigma)_{i\alpha} {P}(\sigma)_{\beta j} {P}(\sigma)_{k\gamma} {P}(\sigma)_{\delta\ell} \notag \\
		&= \frac{ {P}(\sigma)_{ik} {P}(\sigma)_{j\ell} +  {P}(\sigma)_{i\ell} {P}(\sigma)_{jk}}{N}, \label{Equation:Condition_Cov_Hess_Guess}
	\end{align}
	where we used the fact that $ {P}(\sigma) =  {P}(\sigma)^\top$ and $ {P}(\sigma)^2 =  {P}(\sigma)$ in the last line. Comparing \eqref{Equation:Condition_Cov_Hess} and \eqref{Equation:Condition_Cov_Hess_Guess} yields \eqref{Equation:CD_Hess}.
\end{proof}

\subsection{Proof of Proposition \ref{Kac-Rice-Final-Formula}}

\begin{lemma}\label{ExpectedCrt}
	For any Borel set $B \subseteq \mathbb{R}$, we have
    \begin{equation}
        \mathbb{E}[\text{Crt}_N(B)] = \frac{1}{\sqrt{p}}\Big(\frac{p - 1}{2\pi}\Big)^{N/2}\int_{\Omega(B)} \dfrac{e^{-\oldconstant{Constant:Exp}Nf_N(\sigma)}}{\opnorm{\sigma}_2^N} \mathbb{E}[|\det (-A_N(\sigma) + P(\sigma)G_NP(\sigma))|]\,d\sigma. \label{Equation:Key_Formula_1}
    \end{equation}
\end{lemma}

\begin{proof}
    Denote for $t \geq 0$ the critical numbers
	\[\text{Crt}_N(B, t) = |\{\|\sigma\|_2 \leq t \mid \nabla H_N(\sigma) = 0, H_N(\sigma) \in B\}|.\] 
    To apply the Kac-Rice formula \eqref{Equation:Kac_Rice}, we check the requirements given in \cite[Theorem 11.2.1]{adler2009random}. Condition (a) is staightforward since the mappings $\sigma \mapsto H_N(\sigma), \nabla H_N(\sigma), \nabla^2H_N(\sigma)$ are continuous and has finite variances. Condition (b) - (f) are guaranteed since $(H_N(\sigma), \nabla H_N(\sigma), \nabla^2H_N(\sigma))$ are Gaussian. Lastly, for condition (g), it suffice to check that the covariance function of the Hessian $\nabla^2H_N(\sigma)$ satisfy \cite[Eq (11.2.5)]{adler2009random} and that the potential term $\sum_{i = 1}^N V(\sigma_i)$ is locally Lipshitz (Lipshitz on compact set). The first assertion can by directly verified by \eqref{Equation:covariance_Formula} and the second assertion follows from $V \in C^2(\mathbb{R})$. Recall the definition of $\Omega(B)$ in \eqref{Equation:Omega}, we may now apply \eqref{Equation:Kac_Rice} to show
	\begin{align}
		\mathbb{E}[\text{Crt}_{N}(B, t)] &= \int_{\{\|\sigma\|_2 \leq t\}} \mathbb{E}\big[|\det \nabla^2H_N(\sigma)| \cdot \mathbbm{1}_{B}(H_N(\sigma)) \mid \nabla H_{N, p}(\sigma) = \underline{0}\big] f_{\nabla H_N(\sigma)} (\underline{0})\,d\sigma \notag \\
		&= \int_{\{\|\sigma\|_2 \leq t\}} \mathbbm{1}_{\Omega(B)}(\sigma) \cdot \mathbb{E}\big[|\det \nabla^2H_N(\sigma)| \,\big\vert\, \nabla H_N(\sigma) = \underline{0}\big] f_{\nabla H_N(\sigma)}(\underline{0}) \,d\sigma, \label{Equation:KR_in_R^N_1}
	\end{align}
	where the second equality holds from \eqref{Equation:CD_H}. To evaluate the density, we see from \eqref{Equation:Cov_Grad_Inverse} that
	\begin{align*}
		f_{\nabla H_N(\sigma)}(\underline{0}) &= \dfrac{\exp\bigl(-\frac{1}{2}(x - \mathbb{E}[\nabla H_N(\sigma)])^\top\text{Cov}(\nabla H_N(\sigma))^{-1}(x - \mathbb{E}[\nabla H_N(\sigma)])\bigr)}{(2\pi)^{N/2}|\det \text{Cov}(\nabla H_N(\sigma)|^{1/2}}\Big|_{x = \underline{0}} \\
		&= \dfrac{\exp\bigl[-\frac{1}{2}\big(V'(\sigma)\big)^\top N^{-1}p^{-2}\opnorm{\sigma}_2^{-2p}\big((1 - p)\sigma\sigma^\top + p\|\sigma\|_2^2 {I}_N\big)\big(V'(\sigma)\big)\big]}{(2\pi)^{N/2}\bigl|\det\bigl[N^{-1}p\opnorm{\sigma}_2^{2(p - 2)}\bigl((p - 1)\sigma\sigma^\top + \|\sigma\|_2^2 I_N\bigr)\bigr]\bigr|^{1/2}} \\
		&= \dfrac{1}{p^{1/2}
			\bigl(2\pi p\opnorm{\sigma}_2^{2(p - 1)}\bigr)^{N/2}}\exp\Bigl[-\frac{1}{2p^2}\Bigl((1 - p)\frac{\lan{\sigma, V'(\sigma)}^2}{\opnorm{\sigma}_2^{2p}} + p\frac{\|V'(\sigma)\|_2^2}{\opnorm{\sigma}_2^{2p - 2}}\Bigr)\Bigr].
	\end{align*}
	Moreover, by \eqref{Equation:CD_Hess}, 
	\begin{align*}
		\mathbb{E}\big[\big|\det \nabla^2H_N(\sigma)\bigr|\big| \nabla H_N(\sigma) = \underline{0}\bigr] = \big(\sqrt{p(p - 1)}\opnorm{\sigma}_2^{p - 2}\big)^N \mathbb{E}[|\det (-A_N(\sigma) +  P(\sigma) G_N P(\sigma))|].
	\end{align*}
	Finally, by plugging these back in the integral \eqref{Equation:KR_in_R^N_1}, we get
    \[\mathbb{E}[\text{Crt}_N(B, t)] = \frac{1}{\sqrt{p}}\Big(\frac{p - 1}{2\pi}\Big)^{N/2}\int_{\{\|\sigma\|_2 \leq t\} \cap \Omega(B)} \dfrac{e^{-\oldconstant{Constant:Exp}Nf_N(\sigma)}}{\opnorm{\sigma}_2^N} \cdot \mathbb{E}\left[\left|\det \left(-A_N(\sigma) + P(\sigma)G_NP(\sigma)\right)\right|\right]\,d\sigma.\]
    Since both sides of the equation above are monotonic in $t$, we may deduce \eqref{Equation:Key_Formula_1} by tuning the radius $t \to \infty$.
\end{proof}

Next, we continue to perform a calculation that reduces the dimension of the matrix in the determinant term of \eqref{Equation:Key_Formula_1}  in the meantime ensuring that the Gaussian part is a lower dimensional GOE. For a fixed nonzero $\sigma\in \mathbb{R}^N$, recall the definition of $ {B}(\sigma)$ defined before Proposition \ref{Kac-Rice-Final-Formula} and the vector $v(\sigma)$ defined in \eqref{Equation:Vector}.

\begin{lemma}[Alternative Expression of Determinant]\label{Proposition:Expression_of_Det}
	For any fixed nonzero $\sigma\in \mathbb{R}^N$, if the inner product $\lan{\sigma, v(\sigma)} \neq 0$, we have
	\begin{align}
        \begin{split}
            &\mathbb{E}[|\det (- A_N(\sigma) +  P(\sigma) G_N P(\sigma))|] \\
		      &= \frac{|\lan{\sigma, v(\sigma)}|}{N\opnorm{\sigma}_2^p}\mathbb{E}\Bigl|\det\Bigl[B(\sigma)^\top\Bigl(\oldconstant{Constant:A_1}\text{diag}\Bigl(\dfrac{V''(\sigma)}{\opnorm{\sigma}_2^{p - 2}}\Bigr) - \dfrac{v(\sigma) v(\sigma)^\top}{\lan{\sigma, v(\sigma)}\opnorm{\sigma}_2^{p - 2}}\Bigr) {B}(\sigma) +  {G}_{N - 1}\Bigr]\Bigr|.
	           \label{Equation:Key_Formula_2}
        \end{split}
	\end{align}	
\end{lemma}

\begin{proof}
	Set the matrices
    \begin{equation}\label{Equation:B_bar}
        \overline{B}(\sigma) = \begin{pmatrix}
		    \frac{\sigma}{\|\sigma\|_2} &  B(\sigma)
	    \end{pmatrix} \in \mathbb{R}^{N \times N} \quad \text{and} \quad C(\sigma) = \oldconstant{Constant:A_1}B(\sigma)^\top\text{diag}\Bpar{\dfrac{V''(\sigma)}{\opnorm{\sigma}_2^{p - 2}}}B(\sigma) + G_{N - 1},
    \end{equation}
    where $G_{N - 1} = P(e_1)G_NP(e_1)$ is a $(N - 1) \times (N - 1)$ GOE. Note that $\overline{B}(\sigma)$ is an orthogonal matrix and a direct computation gives
	\[\overline B(\sigma)^\top P(\sigma) = P(e_1)\overline B(\sigma)^\top.\]
	It follows that
	\begin{align}
		&\det\bigl(-A_N(\sigma) + P(\sigma) G_N P(\sigma)\bigr) \notag \\
        &\overset{d}{=} \det\bigl(-\overline{B}(\sigma)^\top{A}_N(\sigma)\overline{B}(\sigma) +  \overline{B}(\sigma)^\top P(\sigma) G_N {P}(\sigma)\overline{B}(\sigma)\bigr) \label{Equation:D_Matrix_1} \\
		&= \det\bigl(-\overline{B}(\sigma)^\top A_N(\sigma)\overline{B}(\sigma) +  {P}( {e}_1)\bigl(\overline{ {B}}(\sigma)^\top {G}_N\overline{B}(\sigma)\bigr) {P}(e_1)\bigr) \notag \\
		&\overset{d}{=}  \det\bigl(-\overline{B}(\sigma)^\top A_N(\sigma)\overline{B}(\sigma) -  P(e_1) G_N {P}(e_1)\bigr) \label{Equation:D_Matrix_2} \\
		&= -\begin{pmatrix}
			N^{-1}\opnorm{\sigma}_2^{-p}\lan{\sigma, v(\sigma)} & \bpar{N^{-1/2}\opnorm{\sigma}_2^{1 - p} v(\sigma)}^\top B(\sigma) \\[0.1cm]
			B(\sigma)^\top\bpar{N^{-1/2}\opnorm{\sigma}_2^{1 - p} v(\sigma)} & C(\sigma)
		\end{pmatrix}, \label{Equation:Distribution_Matrix}
	\end{align}
	where \eqref{Equation:D_Matrix_1} holds by rotational invariance of $G_N$ and \eqref{Equation:D_Matrix_2} holds since $-G_N \overset{d}{=} G_N$. Note that if $\lan{\sigma, v(\sigma)} \neq 0$,
	\begin{align*}
		&|\det[-A_N(\sigma) + P(\sigma) G_N P(\sigma)]| \\
		& \overset{d}{=} \Bigl|\det(C(\sigma)) \cdot \Bigl(\frac{\lan{\sigma, v(\sigma)}}{N\opnorm{\sigma}_2^p} - \frac{v(\sigma)^\top B(\sigma)C(\sigma)^{-1}B(\sigma)^\top v(\sigma)}{N\opnorm{\sigma}_2^{2p - 2}}\Bigr)\Bigr| \\
		& = \frac{|\lan{\sigma, v(\sigma)}|}{N\opnorm{\sigma}_2^p}\Bigl|\det\Bigl(C(\sigma) - \dfrac{B(\sigma)^\top v(\sigma) v(\sigma)^\top B(\sigma)}{\lan{\sigma, v(\sigma)}\opnorm{\sigma}_2^{p - 2}}\Bigr)\Bigr|,
	\end{align*}
	where we apply \eqref{Equation:Det_Block_Matrix} in the first equality and \eqref{Equation:Det_Rank_1} in the second. We then obtain the desired formula.
\end{proof}

\begin{proof}[\bf Proof of Proposition \ref{Kac-Rice-Final-Formula}]
	\eqref{Equation:Kac-Rice-Final-Formula} is simply a combination of \eqref{Equation:Key_Formula_1}, \eqref{Equation:Key_Formula_2}, and \eqref{Equation:V'} (which justifies the inner product $\langle \sigma, v(\sigma)\rangle > 0$).
\end{proof}

\section{Truncation Arguments}\label{Section:Truncation}

\noindent This is a preparatory section for our proof of Theorem \ref{Theorem:Main_Theorem}. Let $u \geq 0$ be fixed and plug the Borel set $B = [Nu, \infty)$ into Proposition \ref{Kac-Rice-Final-Formula}, where we abuse notion by $\Omega(u) = \Omega([Nu, \infty))$. The idea is to take advantage of Theorem \ref{Theorem:Main_Asymptote_Bounded} and extract the log asymptotic behavior of the expected value of the determinant term in \eqref{Equation:Kac-Rice-Final-Formula}. However, to do so, we need to truncate the mean of $M_{N - 1}(\sigma)$ (recall from \eqref{def:M}) so that it has bounded operator norm. Define, for any $K > 0$, the matrix $Q_N^K(\sigma)$ by
\[Q_N^K(\sigma) = \oldconstant{Constant:A_1}\,\text{diag}\Bpar{\frac{V''(\sigma)}{\opnorm{\sigma}_2^{ p -2}} \wedge K} + G_N \in \mathbb{R}^{N \times N}.\]
Our main goal for this section will be to show the following proposition.
\begin{proposition}\label{Proposition:Main_Truncation}
    If $V$ satisfies \ref{Condition:V}, then for all $M > 0$,
    \[\lim_{N, K \to \infty} \sup_{\opnorm{\sigma}_2 \leq M} \Big|\dfrac{1}{N}\log\mathbb{E}\big|\det Q_N^K(\sigma)\bigr| - \dfrac{1}{N}\log\mathbb{E}\big|\det M_{N - 1}(\sigma)\bigr|\Bigr| = 0.\]
\end{proposition}
As a preparation, we need to show that the logarithmic asymptotics of the integral \eqref{Equation:Key_Formula_2} is essentially the same as the one where the domain of the integral is restricted to be within a bounded $\opnorm{\cdot}_{2q_2 - 2}$-ball. To be precise, given a function $w \colon \mathbb{R}^N \to \mathbb{R}$, we set for all Borel measurable set $B \subseteq \mathbb{R}^N$ the notation
\begin{equation}\label{Equation:Integral_Notation}
    I_w(B) = \int_{\Omega(u) \cap B} \frac{|\lan{\sigma, v(\sigma)}|}{N\opnorm{\sigma}_2^{N + p}} e^{-\oldconstant{Constant:Exp}Nf_N(\sigma)} \cdot w(\sigma) d\sigma.
\end{equation}
Then, we will show that
\begin{proposition}\label{Proposition:q_Truncation}
    If $V$ satisfies Assumption \ref{ass1}, we have for any $u \geq 0$ and $\epsilon > 0$, there exist $M = M(\epsilon, u)$ such that,
    \[\limsup_{N \to \infty} \dfrac{1}{N}\log I_{\mathbb{E}|\det M_{N - 1}(\cdot)|}(\mathbb{R}) \leq \epsilon + \limsup_{N \to \infty}\frac{1}{N}\log I_{\mathbb{E}|\det M_{N - 1}(\cdot)|}(\{\opnorm{\sigma}_{2q_2 - 2} \leq M\}).\]
\end{proposition}

The proofs of these truncation propositions are provided in the following two subsections.

\subsection{Proof of Proposition \ref{Proposition:Main_Truncation}}

\noindent We introduce the following two matrices:
\begin{equation}
    D_N^K(\sigma) = \oldconstant{Constant:A_1}\,\text{diag}\Bpar{\frac{V''(\sigma)}{\opnorm{\sigma}_2^{p-2}} \wedge K} \quad \text{and} \quad M_{N - 1}^K(\sigma) = B(\sigma)^\top D_N^K(\sigma)B(\sigma) + G_{N - 1}.
\end{equation}
We also recall the definition of $Q_N^K(\sigma) = D_N^K(\sigma) + G_N$ from the start of the section.

\begin{lemma}
    If $V$ satisfies \ref{Condition:V}, then for all $M \geq 1$ and $K \geq 2\oldconstant{Constant:Bound}$,
    \begin{equation}\label{Equation:KS_Distance}
        \sup_{\opnorm{\sigma}_2 \leq M} d_{\text{KS}}\bigl(\mathbb{E}[\mu_{M_{N - 1}(\sigma)}], \mathbb{E}[\mu_{Q_N^K(\sigma)}]\bigr) \leq 4\Bigl(\frac{2\oldconstant{Constant:Bound}M^{q_2 - p}}{K}\Bigr)^{2/(q_2 - 2)} + \frac{3}{N}.
    \end{equation}
\end{lemma}

\begin{proof}
    We will show that
    \begin{align}
        \sup_{\opnorm{\sigma}_2 \leq M} d_{\text{KS}}\bigl(\mathbb{E}[\mu_{M_{N - 1}(\sigma)}], \mathbb{E}[\mu_{M_{N - 1}^K(\sigma)}]\bigr) &\leq \frac{2N}{N - 1}\Big(\frac{2\oldconstant{Constant:Bound}M^{q_2 - p}}{K}\Bigr)^{2/(q_2 - 2)} + \frac{1}{N - 1}, \label{Equation:KS_Distance_1} \\
        \sup_{\sigma \in \mathbb{R}^N} d_{\text{KS}}\bigl(\mathbb{E}[\mu_{M_{N - 1}^K(\sigma)}], \mathbb{E}[\mu_{Q_N^K(\sigma)}]\bigr) &\leq \frac{1}{N}. \label{Equation:KS_Distance_2}
    \end{align}
    For inequality \eqref{Equation:KS_Distance_1}, observe that 
    \begin{align*}
        &\text{rank}(M_{N - 1}(\sigma) - M_{N - 1}^K(\sigma))\\
        & = \text{rank}\Big[B(\sigma)^\top\Big(\oldconstant{Constant:A_1}\,\text{diag}\Big[\Big(\frac{V''(\sigma)}{\opnorm{\sigma}_2^{p - 2}} - K\Big) \mathbbm{1}_{\{\opnorm{\sigma}_2^{2 - p}V''(\sigma) \geq K\}}\Big] + \dfrac{v(\sigma) v(\sigma)^\top}{\lan{\sigma, v(\sigma)}\opnorm{\sigma}_2^{p - 2}}\Bigr)B(\sigma)\Big] \\
        & \leq \text{rank}\Bigl(\oldconstant{Constant:A_1}\,\text{diag}\Big[\Big(\frac{V''(\sigma)}{\opnorm{\sigma}_2^{p - 2}} - K\Big) \mathbbm{1}_{\{\opnorm{\sigma}_2^{2 - p}V''(\sigma) \geq K\}}\Big] + \dfrac{v(\sigma) v(\sigma)^\top}{\lan{\sigma, v(\sigma)}\opnorm{\sigma}_2^{p - 2}}\Bigr) \\
        & \leq \sum_{i = 1}^N \mathbbm{1}_{\{\opnorm{\sigma}_2^{2 - p}V''(\sigma_i) \geq K\}} + 1,
    \end{align*}
    where the second inequality holds by the Cauchy interlacing inequality \eqref{Equation:Cauchy_Interlacing}. By condition \ref{Condition:V}, we have for $K \geq 2\oldconstant{Constant:Bound}$ that
    \begin{align*}
        \sum_{i = 1}^N \mathbbm{1}_{\{\opnorm{\sigma}_2^{2 - p}V''(\sigma_i) \geq K\}} &\leq \sum_{i = 1}^N \bigl(\mathbbm{1}_{\{\oldconstant{Constant:Bound}\opnorm{\sigma}_2^{2 - p}|\sigma_i|^{q_1 - 2} \geq K/2\}} + \mathbbm{1}_{\{\oldconstant{Constant:Bound}\opnorm{\sigma}_2^{2 - p}|\sigma_i|^{q_2 - 2} \geq K/2\}}\big) \\
        &\leq \sum_{i = 1}^N \Bigl(\frac{|\sigma_i|^2}{(\opnorm{\sigma}_2^{p - 2}K/2\oldconstant{Constant:Bound})^{2/(q_1 - 2)}} + \frac{|\sigma_i|^2}{(\opnorm{\sigma}_2^{p - 2}K/2\oldconstant{Constant:Bound})^{2/(q_2 - 2)}}\Bigr) \\
        &\leq N(K/2\oldconstant{Constant:Bound})^{-2/(q_2 - 2)}\bpar{\opnorm{\sigma}_2^{2(q_1 - p)/(q_1 - 2)} + \opnorm{\sigma}_2^{2(q_2 - p)/(q_2 - 2)}}.
    \end{align*}
    It follows from \eqref{Equation:E[KS]} and the rank inequality (Lemma \ref{Lemma:Rank_Inequality}) that if $\opnorm{\sigma}_2 \leq M$ and $M \geq 1$,
    \begin{align*}
        d_{\text{KS}}\bigl(\mathbb{E}[\mu_{M_{N - 1}(\sigma)}], \mathbb{E}[\mu_{M_{N - 1}^K(\sigma)}]\bigr) &\leq \mathbb{E}[d_{\text{KS}}\bigl(\mu_{M_{N - 1}(\sigma)}, \mu_{M_{N - 1}^K(\sigma)}\bigr)] \\
        &\leq \dfrac{\mathbb{E}[\text{rank}\,(M_{N - 1}(\sigma) - M_{N - 1}^K(\sigma))]}{N - 1} \\
        &\leq \frac{2N}{N - 1}\Big(\frac{2\oldconstant{Constant:Bound}M^{q_2 - p}}{K}\Bigr)^{2/(q_2 - 2)} + \frac{1}{N - 1}.
    \end{align*}
    This proves \eqref{Equation:KS_Distance_1}. For \eqref{Equation:KS_Distance_2}, recall from \eqref{Equation:B_bar} the definition of the orthogonal matrix $\overline{B}(\sigma)$, 
    \begin{align*}
        \text{Spec}(Q_N^K(\sigma)) &= \text{Spec}\big[\overline{B}(\sigma) Q_N^K(\sigma) \overline{B}(\sigma)\big] \notag \\
        &\overset{d}{=} \text{Spec}\Big[\begin{pmatrix}
            \opnorm{\sigma}_2^{-2}\sigma^\top D_N^K(\sigma)\sigma & \opnorm{\sigma}_2^{-1}\sigma^\top D_N^K(\sigma)B(\sigma) \\[0.1cm]
            \opnorm{\sigma}_2^{-1}B(\sigma)^\top D_N^K(\sigma)\sigma & B(\sigma)^\top D_N^K(\sigma)B(\sigma)
        \end{pmatrix} + G_N\Big] = \text{Spec}(\overline{Q}_N^K(\sigma)), \label{Equation:Truncation_KS_2}
    \end{align*}
    where we apply the rotational invariance of $G_N$ in the second equality and $\overline{Q}_N^K(\sigma)$ is defined as 
    \[\overline{Q}_N^K(\sigma) = \begin{pmatrix}
        \opnorm{\sigma}_2^{-2}\sigma^\top D_N^K(\sigma)\sigma + g_{11} & \opnorm{\sigma}_2^{-1}\sigma^\top D_N^K(\sigma)B(\sigma) + g^\top \\[0.2cm]
        \opnorm{\sigma}_2^{-1}B(\sigma)^\top D_N^K(\sigma)\sigma + g & B(\sigma)^\top D_N^K(\sigma)B(\sigma) + G_{N - 1}
    \end{pmatrix},\]
    where $G_N = (g_{ij})_{1 \leq i, j \leq N}$, $g = (g_{12}, \dots, g_{1N})^\top$. Therefore,
    \begin{align*}
	   d_{\text{KS}}\bigl(\mathbb{E}[\mu_{M_{N - 1}^K(\sigma)}], \mathbb{E}[\mu_{Q_N^K(\sigma)}]\bigr) &= 
	   d_{\text{KS}}\bigl(\mathbb{E}[\mu_{M_{N - 1}^K(\sigma)}], \mathbb{E}[\mu_{\overline{Q}_N^K(\sigma)}]\bigr) \\
	   &\leq \mathbb{E}\bigl[d_{\text{KS}}\bigl(\mu_{M_{N - 1}^K(\sigma)}, \mu_{\overline{Q}_N^K(\sigma)}\bigr)\bigr] \leq \frac{1}{N},
    \end{align*}
    where the first inequality again follows from \eqref{Equation:E[KS]} and the second inequality holds by the Cauchy interlacing inequality \eqref{Equation:Cauchy_Interlacing} since $M_{N - 1}^K(\sigma)$ is a submatrix of $\overline{Q}_N^K(\sigma)$. Combining \eqref{Equation:KS_Distance_1} and \eqref{Equation:KS_Distance_2}, we may obtain \eqref{Equation:KS_Distance} by employing triangular inequality.
\end{proof}

For $\eta > 0$ and $\lambda > 0$, set the function
\[\log_\eta(\lambda) = \log|\lambda + i\eta|.\]
Note that $\log \lambda<\log_\eta(\lambda)$ for all $\lambda>0$ and $\|\log_\eta(\cdot)\|_{\text{Lip}} \leq (2\eta)^{-1}.$

\begin{lemma}\label{Lemma:KSD}
    If $V$ satisfies \ref{Condition:V}, then for all $M > 0$,
    \begin{equation}\label{Equation:KSD}
        \lim_{N, K \to \infty} \sup_{\opnorm{\sigma}_2 \leq M} \Bigl|\int_{\mathbb{R}} \log|\lambda|\,\mathbb{E}\bigl[\mu_{Q_N^K(\sigma)} - \mu_{M_{N - 1}(\sigma)}\bigr](d\lambda)\Bigr| = 0.
    \end{equation}
\end{lemma}

\begin{proof}
    For convenience, denote $\mu = \mathbb{E}[\mu_{Q_N^K(\sigma)}]$ and $\nu = \mathbb{E}[\mu_{M_{N - 1}(\sigma)}]$.
    Observe that
    \begin{equation}\label{Equation:KSD_1}
        \Bigl|\int_{\mathbb{R}} \log|\lambda| (\mu - \nu)(d\lambda)\Bigr| \leq \int_{\mathbb{R}} \bigl(\log_{\eta}(\lambda) - \log|\lambda|\bigr)(\mu + \nu)(d\lambda) + \Bigl|\int_{\mathbb{R}} \log_\eta(\lambda) (\mu - \nu)(d\lambda)\Bigr|.
    \end{equation}
    For the first term in \eqref{Equation:KSD_1}, note that by Theorem \ref{Theorem:GOE_Wegner_Estimate}, we have Wegner estimates (uniform over $\sigma \in \mathbb{R}^N$) for $Q_N^K(\sigma)$ and $M_{N - 1}(\sigma)$.  Fix $\epsilon > 0$, we may choose $\eta = \eta(\epsilon) > 0$ so that \eqref{Equation:Condition_Wegner} holds. Then, by Corollary \ref{corollary:C.8}, we have the estimation
    \[\int_{\mathbb{R}} \left(\log_{\eta}(\lambda) - \log|\lambda|\right)(\mu + \nu)(d\lambda) \leq 2\oldconstant{Constant:Truncation_Log}\bpar{\log(1 + \eta) + \eta^{(1 - \epsilon)/2}} < \epsilon\]
    by choosing small $\eta > 0$. For the second term in \eqref{Equation:KSD_1}, we bound
    \begin{equation}\label{Equation:KSD_2}
        \Bigl|\int_{\mathbb{R}} \log_\eta(\lambda)(\mu - \nu)(d\lambda)\Bigr| \leq \int_{\{|\lambda| \geq A\}} \log_\eta(\lambda) (\mu + \nu)(d\lambda) + \Bigl|\int_{\{|\lambda| < A\}} \log_\eta(\lambda)(\mu - \nu)(d\lambda)\Bigr|
    \end{equation}
    for some constant $A > 1$ to be chosen later. For the tail term in \eqref{Equation:KSD_2}, observe that for any $0<\eta<1$ and any random variable $X$ with the law $\mu$, then by setting $s = 2/(q_2 - 2)$, there exist $\newconstant\label{A} > 0$ such that
    \begin{equation}\label{Equation:KSD_3}
        \int_{\{|\lambda| \geq A\}} \log_\eta(\lambda)\mu(d\lambda) = \mathbb{E}\bsq{\log_\eta(X) \cdot \mathbbm{1}_{\{|X| \geq A\}}} \leq \oldconstant{A}\bsq{|X|^{s / 2} \cdot \mathbbm{1}_{\{|X| \geq A\}}} \leq \oldconstant{A} A^{-s/2}\mathbb{E}[|X|^s].
    \end{equation}
    Here, the $s$-moment
    \begin{align}
        \mathbb{E}[|X|^s] &= \dfrac{1}{N}\mathbb{E}\bsq{\text{Tr}\bigl(\left|Q_N^K(\sigma)\right|^s\bigr)} \notag \\
        &\leq \dfrac{2^{s}}{N}\mathbb{E}\left[\text{Tr}\bigl[(B(\sigma)^\top D_N^K(\sigma)B(\sigma))^s + |G_N|^s\bigr]\right] \label{Equation:KSD_4} \\
        &\leq 2^s\Bigl(\frac{1}{N}\text{Tr}\left(\oldconstant{Constant:A_1}^s\text{diag}\left(\oldconstant{Constant:Bound}\opnorm{\sigma}_2^{2 - p}(|\sigma|^{q_1 - 2} + |\sigma|^{q_2 - 2})\right)^s\right) + \mathbb{E}\left[\||G_N|\|_{\text{op}}^s\right]\Bigr) \label{Equation:KSD_5} \\
        &\leq 2^s\left(2^s(\oldconstant{Constant:Bound}\oldconstant{Constant:A_1})^s\opnorm{\sigma}_2^{s(2 - p)}\bigl(\opnorm{\sigma}_{2(q_1 - 2)/(q_2 - 2)}^{2(q_1 - 2)/(q_2 - 2)} + \opnorm{\sigma}_2^2\bigr) + \mathbb{E}\left[\|G_N\|_{\text{op}}^s\right]\right) \label{Equation:KSD_6}  \\
        &\leq 2^{s}\left(2^s(\oldconstant{Constant:Bound}\oldconstant{Constant:A_1})^s\bigl(\opnorm{\sigma}_{2}^{2(q_1 - p)/(q_2 - 2)} + \opnorm{\sigma}_2^{2(q_2 - p) / (q_2 - 2)}\bigr) + \mathbb{E}\left[\|G_N\|_{\text{op}}^s\right]\right), \label{Equation:KSD_7}
    \end{align}
    where \eqref{Equation:KSD_4} holds by \eqref{Equation:Subadditive_Trace} with $m = 2$ and the fact that $|A| = A$ if $A$ is positive semi-definite; \eqref{Equation:KSD_5} holds by the Cauchy interlacing \eqref{Equation:Cauchy_Interlacing} (Using $\mbox{Tr}f(B^\top AB)\leq \mbox{Tr}(f(A))$ if $B^\top B=I_{k}$), condition \ref{Condition:V}, and $\||G_N|\|_{\text{op}} = \|G_N\|_{\text{op}}$; and \eqref{Equation:KSD_6} holds by the elementary inequality $(a + b)^s \leq 2^s(a^s + b^s)$. Observe that \eqref{Equation:KSD_6} is a finite constant (independent of $N$) since $\opnorm{\sigma}_2 \leq M$. Hence, we may choose $A$ large enough so that \eqref{Equation:KSD_2} is less than $\epsilon$. The same argument works for $M_{N - 1}(\sigma)$. For the bulk term in \eqref{Equation:KSD_2}, we see by applying integration by part,
    \begin{align*}
        &\Big|\int_{\{|\lambda| < A\}} \log_\eta(\lambda)(\mu - \nu)(d\lambda)\Big| \\
        &= \Big|\left.\log_\eta(\lambda)(F_\mu(\lambda) - F_\nu(\lambda))\right|_{\lambda = -A}^{\lambda = A} - \int_{\{|\lambda| < A\}} \left(\log_\eta(\lambda)\right)'(F_\mu(\lambda) - F_\nu(\lambda))\,d\lambda\Big| \\
        &\leq 2\log_\eta(A) \cdot d_{\text{KS}}(\mu, \nu) + \frac{2A}{2\eta} \cdot d_{\text{KS}}(\mu, \nu),
    \end{align*}
    where the inequality holds since $\|\log_\eta(\cdot)\|_{\text{Lip}} \leq (2\eta)^{-1}$. By \eqref{Equation:KS_Distance}, we see then the term vanishes as $N, K \to \infty$ and thus \eqref{Equation:KSD} holds.
\end{proof}

\begin{proof}[\bf Proof of Proposition \ref{Proposition:Main_Truncation}]
    Observe that we have
    \begin{align*}
        &\Big|\dfrac{1}{N}\log\mathbb{E}\left|\det Q_N^K(\sigma)\right| - \dfrac{1}{N}\log\mathbb{E}|\det M_{N - 1}(\sigma)|\Big| \\
        &\leq \Big|\dfrac{1}{N}\log\mathbb{E}\big|\det Q_N^K(\sigma)\big| - \int_{\mathbb{R}} \log|\lambda|\,\mathbb{E}\big[\mu_{Q_N^K(\sigma)}\big](d\lambda)\Big| + \Big|\int_{\mathbb{R}} \log|\lambda|\,\mathbb{E}\big[\mu_{Q_N^K(\sigma)} - \mu_{M_{N - 1}(\sigma)}\big](d\lambda)\Big| \\
        &\hspace{2cm} + \Big|\dfrac{1}{N}\log\mathbb{E}|\det M_{N - 1}(\sigma)| - \int_{\mathbb{R}} \log|\lambda|\,\mathbb{E}\left[\mu_{M_{N - 1}(\sigma)}\right](d\lambda)\Big|.
    \end{align*}
    The first and the third term vanishes as $N \to \infty$ uniformly over $\sigma \in \mathbb{R}^N$ and $K > 0$ by Theorem \ref{Theorem:Main_Asymptote_Unbounded_GOE}. The second term vanishes as $N, K \to \infty$ uniformly over $\opnorm{\sigma}_2 \leq M$ by Lemma \ref{Lemma:KSD}. We then deduce the desired assertion.
\end{proof}

\subsection{Proof of Proposition \ref{Proposition:q_Truncation}}

\noindent First of all, note that by the Cauchy-Schwarz inequality, we have for all $\sigma \in \mathbb{R}^N$, 
\begin{equation}\label{Equation:Rough_Estimate_Exp}
	\dfrac{\opnorm{V'(\sigma)}_2^2}{\opnorm{\sigma}_2^{2p - 2}} \leq f_N(\sigma) \leq p\dfrac{\opnorm{V'(\sigma)}_2^2}{\opnorm{\sigma}_2^{2p - 2}}.
\end{equation}
We begin with two technical lemmas.

\begin{lemma}
    If $V$ satisfies Assumption \ref{ass1}, then for all $C > 1$, and $B \subseteq \mathbb{R}^N$ we have
    \begin{align}
        I_1(B) &\geq \oldconstant{Constant:Tech_2}N^{N/2}S_{N - 1}\int_{(u/\oldconstant{Constant:Tech_1})^{1 - p/q_1}}^\infty e^{-4\oldconstant{Constant:Bound}^2\oldconstant{Constant:Exp}pN\big(t^2C^{2(2q_1 - 2)} + t^{2(q_2 - p)/(q_1 - p)}C^{2(2q_2 - 2)}\big)} \notag \\
        &\hspace{1cm} \times \mathbb{P}\Big(\{t^{1/(q_1 - p)}\sqrt{N}g/\|g\|_2 \in B\} \cap \bigcap_{i = 1}^N \{C^{-1} \leq |g_i| \leq C\}\Big) dt, \label{Equation:Tech}
    \end{align}
    where the constant $\newconstant\label{Constant:Tech_1} = \oldconstant{Constant:Bound}^{-1}(q/p - 1)$, $\newconstant\label{Constant:Tech_2} = \oldconstant{Constant:Bound}^{-1}\oldconstant{Constant:A_1}(q - p)$, and $g_1, \dots, g_N$ are i.i.d.\ standard normal random variables.
\end{lemma}

\begin{proof}
    We will start with some preliminary estimate. First, by \ref{Condition:V} and the elementary inequality $(a + b)^2 \leq 2(a^2 + b^2)$, we have
    \begin{equation}\label{Equation:Tech_1}
        \opnorm{V'(\sigma)}_2^2 \leq \frac{\oldconstant{Constant:Bound}^2}{N}\sum_{i = 1}^N (|\sigma_i|^{q_1 - 1} + |\sigma_i|^{q_2 - 1})^2 \leq 2\oldconstant{Constant:Bound}^2\big(\opnorm{\sigma}_{2q_1 - 2}^{2q_1 - 2} + \opnorm{\sigma}_{2q_2 - 2}^{2q_2 - 2}\big).
    \end{equation}
    By \eqref{Equation:V'} (since condition \ref{Condition:V''} is satisfied) and \ref{Condition:V}, we have
    \begin{equation}\label{Equation:Tech_2}
        N^{-1}\lan{\sigma, v(\sigma)} \geq \oldconstant{Constant:Bound}^{-1}\oldconstant{Constant:A_1}(q - p)\opnorm{\sigma}_{q_1}^{q_1} \geq \oldconstant{Constant:Tech_2}\opnorm{\sigma}_{2}^{q_1}
    \end{equation}
    Last, we have by \eqref{Equation:V} (since condition \ref{Condition:V'} is satisfied) and \ref{Condition:V},
    \begin{equation}\label{Equation:Tech_3}
        N^{-1}\big(p^{-1}\lan{\sigma, V'(\sigma)} - \lan{\underline{1}, V(\sigma)}\big) \geq \oldconstant{Constant:Bound}^{-1}(q/p - 1)\opnorm{\sigma}_{q_1}^{q_1} \geq \oldconstant{Constant:Tech_1}\opnorm{\sigma}_2^{q_1},
    \end{equation}
    which implies $\Omega(u) \supseteq \{\opnorm{\sigma}_2 \geq (u/\oldconstant{Constant:Tech_1})^{1/q_1}\}$. We see then
    \begin{align*}
        I_1(B) &= \int_{\Omega(u) \cap B}\frac{|\lan{\sigma, v(\sigma)}|}{N\opnorm{\sigma}_2^{N + p}} e^{-\oldconstant{Constant:Exp}Nf_N(\sigma)} d\sigma \\
        &\geq \oldconstant{Constant:Tech_2}\int_{\{\opnorm{\sigma}_2 \geq (u/\oldconstant{Constant:Tech_1})^{1/q_1}\} \cap B} \opnorm{\sigma}_2^{q_1 - p - N} e^{-2\oldconstant{Constant:Bound}^2\oldconstant{Constant:Exp}pN\big(\opnorm{\sigma}_{2q_1 - 2}^{2q_1 - 2} + \opnorm{\sigma}_{2q_2 - 2}^{2q_2 - 2}\big)/\opnorm{\sigma}_2^{2p - 2}} d\sigma \\
        &= \oldconstant{Constant:Tech_2} N^{N/2}S_{N - 1}\int_{(u/\oldconstant{Constant:Tech_1})^{1/q_1}}^\infty t^{q_1 - p -1} \\
        &\hspace{1cm} \times\mathbb{E}\Big[e^{-4\oldconstant{Constant:Bound}^2\oldconstant{Constant:Exp}pN\big(t^{2(q_1 - p)}\opnorm{\sqrt{N}g/\|g\|_2}_{2q_1 - 2}^{2q_1 - 2} + t^{2(q_2 - p)}\opnorm{\sqrt{N}g/\|g\|_2}_{2q_2 - 2}^{2q_2 - 2}\big)}\mathbbm{1}_{\{t\sqrt{N}g/\|g\|_2 \in B\}}\Big] dt,
    \end{align*}
    where the first inequality holds by the upper bound in \eqref{Equation:Rough_Estimate_Exp}, \eqref{Equation:Tech_1}, \eqref{Equation:Tech_2} and \eqref{Equation:Tech_3}; the second inequality holds by the change of variable $\sigma = \sqrt{N}t\omega$ where $\omega \in \mathbb{S}^{N - 1}$ and the identification that $\omega \overset{d}{=} g/\|g\|_2$. If we further restrict the expected value on the event $\bigcap_{i = 1}^N \{C^{-1} \leq |g_i| \leq C\}$, then the exponent
    \[t^{2(q_1 - p)}\opnorm{\sqrt{N}g/\|g\|_2}_{2q_1 - 2}^{2q_1 - 2} + t^{2(q_2 - p)}\opnorm{\sqrt{N}g/\|g\|_2}_{2q_2 - 2}^{2q_2 - 2} \leq t^{2(q_1 - p)}C^{2(2q_1 - 2)} + t^{2(q_2 - p)}C^{2(2q_2 - 2)}.\]
    By substituting $s = t^{q_1 - p}$, we see \eqref{Equation:Tech} holds.
\end{proof}

\begin{lemma}\label{Lemma:General_Truncation}
    Suppose $V$ satisfies Assumption \ref{ass1} and let $w_N \colon \mathbb{R}^N \to \mathbb{R}$ be a sequence of functions.
    \begin{enumerate}[label = (\roman*)]
        \item\label{Lemma_Trucation_1} If there exist a constant $\newconstant\label{Constant:Rough_Lower_Bound_1} > 0$ such that 
        \begin{equation}
            w_N(\sigma) \geq e^{-\oldconstant{Constant:Rough_Lower_Bound_1}N}\mathbbm{1}_{\big\{\opnorm{\sigma}_2^{2 - p}\min(|\sigma|^{q_1 - 2}) \geq \oldconstant{Constant:Rough_Lower_Bound_1}\big\}} \label{Equation:Rough_Lower_Bound_Condition}
        \end{equation}
        for all $N \in \mathbb{N}$. Then, there exist constant $\newconstant\label{Constant:Rough_Lower_Bound_2} = \oldconstant{Constant:Rough_Lower_Bound_2}(u) > 0$ such that,
        \begin{equation}\label{Equation:Rough_Lower_Bound}
            I_{w_N}(\mathbb{R}^N) \geq e^{-\oldconstant{Constant:Rough_Lower_Bound_2}N}.
        \end{equation}
        \item\label{Lemma_Trucation_2} If there exist a constant $\newconstant\label{Constant:Rough_Upper_Bound_1} > 0$ such that
        \begin{equation}\label{Equation:Rough_Upper_Bound_Condition}
            |w_N(\sigma)| \leq \exp\Big[\oldconstant{Constant:Rough_Upper_Bound_1}N\Big(1 + \frac{\opnorm{V'(\sigma)}_2}{\opnorm{\sigma}_2^{p - 1}}\Big)\Big]
        \end{equation}
        for all $N \in \mathbb{N}$. Then, there exist constant $\newconstant\label{Constant:Rough_Upper_Bound_2} > 0$ such that for all large $M$,
        \begin{equation}\label{Equation:Rough_Upper_Bound}
            I_{w_N}(\{\opnorm{\sigma}_{2q_2 - 2} \geq M\}) \leq e^{\oldconstant{Constant:Rough_Upper_Bound_2}N(1 - M^{2(q_2 - p)})}.
        \end{equation}
        \item\label{Lemma_Trucation_3} If $w_N$ satisfies both \eqref{Equation:Rough_Lower_Bound} and \eqref{Equation:Rough_Upper_Bound}, then for all $\epsilon > 0$, $u \geq 0$, there exist $M = M(\epsilon, u) > 0$ such that
        \[\limsup_{N \to \infty} \dfrac{1}{N}\log I_{w_N}(\mathbb{R}) \leq \epsilon + \limsup_{N \to \infty} \frac{1}{N}\log I_w(\{\opnorm{\sigma}_{2q_2 - 2} \leq M\}).\]
    \end{enumerate} 
\end{lemma}

\begin{proof}
    Let's start with the lower bound \ref{Lemma_Trucation_1}. Observe that by \eqref{Equation:Rough_Lower_Bound_Condition} and \eqref{Equation:Tech}, we have
    \begin{align}
        I_{w_N}(\mathbb{R}^N) &\geq e^{-\oldconstant{Constant:Rough_Lower_Bound_1}N}I_1(\{\opnorm{\sigma}_2^{2 - p}\min(|\sigma|^{q_1 - 2}) \geq \oldconstant{Constant:Rough_Lower_Bound_1}\}) \notag \\
        &\geq \oldconstant{Constant:Tech_2}e^{-\oldconstant{Constant:Rough_Lower_Bound_1}N}N^{N/2}S_{N - 1}\int_{(u/\oldconstant{Constant:Tech_1})^{1 - p/q_1}}^\infty e^{-4\oldconstant{Constant:Bound}^2\oldconstant{Constant:Exp}pN\big(t^2C^{2(2q_1 - 2)} + t^{2(q_2 - p)/(q_1 - p)}C^{2(2q_2 - 2)}\big)} \label{Equation:Rough_Lower_Bound_1} \\
        &\hspace{1cm} \times \mathbb{P}\Big(\big\{t\min((\sqrt{N}|g|/\|g\|_2)^{q_1 - 2}) \geq \oldconstant{Constant:Rough_Lower_Bound_1}\big\} \cap \bigcap_{i = 1}^N \{C^{-1} \leq |g_i| \leq C\}\Big) dt. \notag 
    \end{align}
    On the event $\bigcap_{i = 1}^N \{C^{-1} \leq |g_i| \leq C\}$, 
    \[\{t\min((\sqrt{N}|g|/\|g\|_2)^{q_1 - 2}) \geq \oldconstant{Constant:Rough_Lower_Bound_1}\} \supseteq \{t \geq \oldconstant{Constant:Rough_Lower_Bound_1}C^{2(q_1 - 2)}\}.\]
    Moreover, if $t \leq C^\alpha$ for some $\alpha > 0$, the exponent in \eqref{Equation:Rough_Lower_Bound_1} can be bounded by
    \begin{align*}
        t^2C^{2(2q_1 - 2)} + t^{2(q_2 - p)/(q_1 - p)}C^{2(2q_2 - 2)} &\leq C^{2\alpha + 2(2q_1 - 2)} + C^{2\alpha(q_2 - p)/(q_1 - p) + 2(2q_2 - 2)} \\
        &\leq 2C^{2\alpha(q_2 - p)/(q_1 - p) + 2(2q_2 - 2)}.
    \end{align*}
    Therefore, by choosing $\alpha > 2(q_1 - 2)$, we have 
    \begin{align*}
        I_{w_N}(\mathbb{R}^N) &\geq \oldconstant{Constant:Tech_2}e^{-\oldconstant{Constant:Rough_Lower_Bound_1}N}N^{N/2}S_{N - 1}\Big(C^\alpha - (u/\oldconstant{Constant:Tech_1})^{1 - p/q_1} \vee \oldconstant{Constant:Rough_Lower_Bound_1}C^{2(q_1 - 2)}\Big) \\
        &\hspace{1cm} \times e^{-4\oldconstant{Constant:Bound}^2\oldconstant{Constant:Exp}pNC^{2\alpha(q_2 - p)/(q_1 - p) + 2(2q_2 - 2)}}\mathbb{P}(C^{-1} \leq |g_1| \leq C)^N
    \end{align*}
    and we see \eqref{Equation:Rough_Lower_Bound} holds for some constant $\oldconstant{Constant:Rough_Lower_Bound_2} > 0$.

    We will start the proof of the upper bound \ref{Lemma_Trucation_2} with some preliminary estimation. Note that by condition \ref{Condition:V},
    \[\frac{\opnorm{V'(\sigma)}_2}{\opnorm{\sigma}_2^{p - 1}} \geq \frac{1}{\oldconstant{Constant:Bound}\opnorm{\sigma}_2^{p - 1}} \Big(\frac{1}{N}\sum_{i = 1}^N |\sigma_i|^{2q_2 - 2}\Big)^{1/2} = \frac{\opnorm{\sigma}_{2q_2 - 2}^{q_2 - 1}}{\oldconstant{Constant:Bound}^2\opnorm{\sigma}_2^{p - 1}} \geq \oldconstant{Constant:Bound}^{-1}\opnorm{\sigma}_{2q_2 - 2}^{q_2 - p}.\]
    Therefore, if $\opnorm{\sigma}_{2q_2 - 2}$ is sufficiently large (so that $\opnorm{V'(\sigma)}_2/\opnorm{\sigma}_2^{p - 1} \geq 2\oldconstant{Constant:Rough_Upper_Bound_1}/\oldconstant{Constant:Exp}$), 
    \begin{align}
        \oldconstant{Constant:Exp}\frac{\opnorm{V'(\sigma)}_2^2}{\opnorm{\sigma}_2^{2p - 2}} - \oldconstant{Constant:Rough_Upper_Bound_1}\frac{\opnorm{V'(\sigma)}_2}{\opnorm{\sigma}_2^{p - 1}} \geq \oldconstant{Constant:Exp}\frac{\opnorm{V'(\sigma)}_2^2}{\opnorm{\sigma}_2^{2p - 2}} - \oldconstant{Constant:Rough_Upper_Bound_1} \cdot \frac{\oldconstant{Constant:Exp}}{2\oldconstant{Constant:Rough_Upper_Bound_1}}\frac{\opnorm{V'(\sigma)}_2^2}{\opnorm{\sigma}_2^{2p - 2}} \geq \frac{\oldconstant{Constant:Exp}}{2\oldconstant{Constant:Bound}^2}\opnorm{\sigma}_{2q_2 - 2}^{2(q_2 - p)}. \label{Equation:Rough_Upper_Bound_1}
    \end{align}
    Moreover, we see that by condition \ref{Condition:V}, if $\opnorm{\sigma}_{2q_2 - 2}$ is sufficiently large,
    \begin{align}
        N^{-1}|\lan{\sigma, v(\sigma)}| &\leq \oldconstant{Constant:Bound}(\oldconstant{Constant:A_1} + \oldconstant{Constant:A_2})(\opnorm{\sigma}_{q_1}^{q_1} + \opnorm{\sigma}_{q_2}^{q_2}) \notag \\
        &\leq \oldconstant{Constant:Bound}(\oldconstant{Constant:A_1} + \oldconstant{Constant:A_2})(\opnorm{\sigma}_{2q_2 - 2}^{q_1} + \opnorm{\sigma}_{2q_2 - 2}^{q_2}) \leq \exp\Big( \frac{\oldconstant{Constant:Exp}}{4\oldconstant{Constant:Bound}^2}\opnorm{\sigma}_{2q_2 - 2}^{2(q_2 - p)}\Big). \label{Equation:Rough_Upper_Bound_2}
    \end{align}
    Combining \eqref{Equation:Rough_Estimate_Exp}, \eqref{Equation:Rough_Upper_Bound_1} and \eqref{Equation:Rough_Upper_Bound_2}, we see that for $M > 0$ large enough, the integral
    \begin{align}
        &I_{w_N}(\{\opnorm{\sigma}_{2q_2 - 2} \geq M\})  \notag \\
        &\leq e^{\oldconstant{Constant:Rough_Upper_Bound_1}N} \int_{\{\opnorm{\sigma}_{2q_2 - 2} \geq M\}} \opnorm{\sigma}_2^{-(N + p)} e^{-(\oldconstant{Constant:Exp}/(4\oldconstant{Constant:Bound}^2))\opnorm{\sigma}_{2q_2 - 2}^{2(q_2 - p)}N} d\sigma \notag \\
        &= \frac{e^{\oldconstant{Constant:Rough_Upper_Bound_1}N}N^{N/2}S_{N - 1}}{q_2 - p} \mathbb{E}\Big[\int_{(M/\opnorm{\sqrt{N}\omega}_{2q_2 - 2})^{q_2 - p}}^\infty t^{-(p + 2)/(q_2 - p)}e^{-(\oldconstant{Constant:Exp}/(4\oldconstant{Constant:Bound}^2))Nt^2\opnorm{\sqrt{N}\omega}_{2q_2 - 2}^{2(q_2 - p)}} dt\Big], \label{Equation:Rough_Upper_Bound_3}
    \end{align}
    where \eqref{Equation:Rough_Upper_Bound_3} holds by change of variables $\sigma = \sqrt{N}t^{1/(q_2 - p)}\omega$ for $\omega \in \mathbb{S}^{N - 1}$ (the expected value is taken over the uniform measure on $\mathbb{S}^{N - 1}$). Note that since $\opnorm{\sqrt{N}\omega}_q \leq N^{1/2 - 1/q}$ when $q \geq 2$,
    \[t^{-1} \leq M^{-1}\opnorm{\sqrt{N}\omega}_{2q_2 - 2} \leq M^{-1}N^{(q_2 - 2)/(2q_2 - 2)} \leq N^{(q_2 - 2)/(2q_2 - 2)}\]
    for $M \geq 1$. Therefore, up to a term of order $e^{O_{N \to \infty}(N)}$, \eqref{Equation:Rough_Upper_Bound_3} can be bound by
    \begin{align}
        &\mathbb{E}\Big[\int_{(M/\opnorm{\sqrt{N}\omega}_{2q_2 - 2})^{q_2 - p}}^\infty e^{-(\oldconstant{Constant:Exp}/(4\oldconstant{Constant:Bound}^2))Nt^2\opnorm{\sqrt{N}\omega}_{2q_2 - 2}^{2(q_2 - p)}} dt\Big] \notag \\
        &\leq \mathbb{E}\Bigl[\frac{e^{-(\oldconstant{Constant:Exp}/(4\oldconstant{Constant:Bound}^2))NM^{2(q_2 - p)}}}{2(\oldconstant{Constant:Exp}/(4\oldconstant{Constant:Bound}^2)N\opnorm{\sqrt{N}\omega}_{2q_2 - 2}^{2(q_2 - p)})(M/\opnorm{\sqrt{N}\omega}_{2q_2 - 2})^{q_2 - p}}\Bigr] \leq e^{-(\oldconstant{Constant:Exp}/(4\oldconstant{Constant:Bound}^2))NM^{2(q_2 - p)}}, \label{Equation:Rough_Upper_Bound_4}
    \end{align}
    where the first inequality holds by $\int_c^\infty e^{-ax^2}dx \leq e^{-ac^2}/(2ac)$ and the second holds for large $N$ since $\opnorm{\sqrt{N}\omega}_q \geq 1$ for $q \geq 2$. We now see \eqref{Equation:Rough_Upper_Bound} follows from combining \eqref{Equation:Rough_Upper_Bound_3} and \eqref{Equation:Rough_Upper_Bound_4}.

    Last, for \ref{Lemma_Trucation_3}, we see that
    \begin{align*}
        \frac{I_{w_N}(\{\opnorm{\sigma}_{2q_2 - 2} \geq M\})}{I_{w_N}(\{\opnorm{\sigma}_{2q_2 - 2} < M\})} &= \frac{I_{w_N}(\{\opnorm{\sigma}_{2q_2 - 2} \geq M\})}{I_{w_N}(\mathbb{R}^N) - I_{w_N}(\{\opnorm{\sigma}_{2q_2 - 2} \geq M\})} \leq \frac{e^{\oldconstant{Constant:Rough_Upper_Bound_2}N(1 - M^{2(q_2 - p)})}}{e^{-\oldconstant{Constant:Rough_Lower_Bound_2}N} - e^{\oldconstant{Constant:Rough_Upper_Bound_2}N(1 - M^{2(q_2 - p)})}} < \epsilon^N
    \end{align*}
    for large $M$. This directly implies our assertion.
\end{proof}

\begin{proof}[\bf Proof of Proposition \ref{Proposition:q_Truncation}]
    Recall $M_{N-1}$ from \eqref{def:M}. Let $w_N(\sigma) = \mathbb{E}|\det M_{N - 1}(\sigma)|$. By the virtue of Lemma \ref{Lemma:General_Truncation}, it suffices to show that conditions \eqref{Equation:Rough_Lower_Bound_Condition} and \eqref{Equation:Rough_Upper_Bound_Condition} hold. We will check \eqref{Equation:Rough_Lower_Bound_Condition} first. Note that for all $\sigma ,y \in \mathbb{R}^N$,
    \begin{align}
        &y^\top\Big(\oldconstant{Constant:A_1}\text{diag}\Big(\dfrac{V''(\sigma)}{\opnorm{\sigma}_2^{p - 2}}\Big) - \dfrac{v(\sigma) v(\sigma)^\top}{\lan{\sigma, v(\sigma)}\opnorm{\sigma}_2^{p - 2}}\Big) y \notag \\
        &= \opnorm{\sigma}_2^{2 - p} \Big(\oldconstant{Constant:A_1}\sum_{i = 1}^N V''(\sigma_i)|y_i|^2 - \frac{\lan{v(\sigma), y}^2}{\lan{\sigma, v(\sigma)}}\Big) \notag \\
        &\geq \opnorm{\sigma}_2^{2 - p} \Big(\oldconstant{Constant:A_1}\sum_{i = 1}^N V''(\sigma_i)|y_i|^2 - \sum_{i = 1}^N \sigma_i^{-1}v(\sigma_i)|y_i|^2\Big) \label{Equation:q_Truncation_1} \\
        &= \oldconstant{Constant:A_2}\opnorm{\sigma}_2^{2 - p} \sum_{i = 1}^N \sigma_i^{-1}V'(\sigma_i)|y_i|^2 = \oldconstant{Constant:A_2}\opnorm{\sigma}_2^{2 - p}\min\{\sigma^{-1}V'(\sigma)\}\|y\|_2^2, \label{Equation:q_Truncation_2}
    \end{align}
    where \eqref{Equation:q_Truncation_1} holds by the Cauchy-Schwarz inequality and noting that $\lan{\sigma, v(\sigma)} > 0$, $V'(t) > 0$ (we understand the notation $x^{-1}v(x) = 0$ if $x = 0$); \eqref{Equation:q_Truncation_2} holds by the definition of $v$ (see \eqref{Equation:Vector}). Recall the notation that $\lambda_1(A)$ stands for the smallest eigenvalue for the symmetric matrix $A$. From \eqref{Equation:q_Truncation_2}, the Cauchy interlacing inequality (see \eqref{Equation:Cauchy_Interlacing}) and condition \ref{Condition:V}, we have
    \begin{align}
        \lambda_1(M_{N - 1}(\sigma)) &= \lambda_1\Big[B(\sigma)^\top\Big(\dfrac{V''(\sigma)}{\opnorm{\sigma}_2^{p - 2}}\Big) - \dfrac{v(\sigma) v(\sigma)^\top}{\lan{\sigma, v(\sigma)}\opnorm{\sigma}_2^{p - 2}}\Big)B(\sigma)\Big] \notag \\
        &\geq \oldconstant{Constant:A_2}\opnorm{\sigma}_2^{2 - p}\min\{\sigma^{-1}V'(\sigma)\} \geq \oldconstant{Constant:Bound}^{-1}\oldconstant{Constant:A_2}\opnorm{\sigma}_2^{2 - p}\min(|\sigma|^{q_1 - 2}). \label{Equation:q_Truncation_3}
    \end{align}
    Therefore, we have the lower bound
    \begin{align*}
        \mathbb{E}|\det M_{N - 1}(\sigma)| &\geq \mathbbm{1}_{\{\oldconstant{Constant:Bound}^{-1}\oldconstant{Constant:A_2}\opnorm{\sigma}_2^{2 - p}\min(|\sigma|^{q_1 - 2}) \geq 10\}}\mathbb{E}\left[|\det M_{N - 1}(\sigma)|\mathbbm{1}_{\left\{\|G_{N - 1}\|_{\text{op}} \leq 9\right\}}\right] \notag \\
        &\geq (10 - 9)^N\mathbbm{1}_{\{\oldconstant{Constant:Bound}^{-1}\oldconstant{Constant:A_2}\opnorm{\sigma}_2^{2 - p}\min(|\sigma|^{q_1 - 2}) \geq 10\}}\mathbb{P}(\|G_{N - 1}\|_{\text{op}} \leq 9).
    \end{align*}
    By \eqref{Equation:Tail_of_GOE}, we see then $w_N$ satisfies \eqref{Equation:Rough_Lower_Bound_Condition} with $\oldconstant{Constant:Rough_Lower_Bound_1} = 10\oldconstant{Constant:Bound}\oldconstant{Constant:A_2}^{-1}$ for large $N$. As for \eqref{Equation:Rough_Upper_Bound_Condition}, observe that 
    \begin{align}
        &|\det M_{N - 1}(\sigma)| \notag \\
        &\leq \exp\Bsq{\sum_{i = 1}^{N - 1} \log \Bpar{\Bigl|\lambda_i\Bigl(B(\sigma)^\top\Bigl(\oldconstant{Constant:A_1}\text{diag}\Bigl(\dfrac{V''(\sigma)}{\opnorm{\sigma}_2^{p - 2}}\Bigr) - \dfrac{v(\sigma) v(\sigma)^\top}{\lan{\sigma, v(\sigma)}\opnorm{\sigma}_2^{p - 2}}\Bigr) {B}(\sigma)\Bigr)\Bigr| + \|G_{N - 1}\|_{\text{op}}}} \notag \\
        &\leq \exp\Bigl( N\log \mbox{\rm Tr}\Bigl(B(\sigma)^\top\Bigl(\oldconstant{Constant:A_1}\text{diag}\Bigl(\dfrac{V''(\sigma)}{\opnorm{\sigma}_2^{p - 2}}\Bigr) - \dfrac{v(\sigma) v(\sigma)^\top}{\lan{\sigma, v(\sigma)}\opnorm{\sigma}_2^{p - 2}}\Bigr) {B}(\sigma)\Bigr) + N\|G_{N - 1}\|_{\text{op}}\Bigr) \label{Equation:q_Truncation_4} \\
            &\leq \exp\Bigl(N\log\Bigl(\oldconstant{Constant:A_1}\frac{\opnorm{V''(\sigma)}_1}{\opnorm{\sigma}_2^{p - 2}} - \frac{\opnorm{v(\sigma)}_2^2}{\lan{\sigma, v(\sigma)}\opnorm{\sigma}_2^{p - 2}} + \|G_{N - 1}\|_{\text{op}}\Bigr)\Bigr)  \label{Equation:q_Truncation_5} \\
            &\leq \Big(\oldconstant{Constant:A_1}\opnorm{\sigma}_2^{2 - p}\opnorm{V''(\sigma)}_1 + \|G_{N - 1}\|_{\text{op}}\Big)^N, \label{Equation:q_Truncation_6}
    \end{align}
    where \eqref{Equation:q_Truncation_4} holds by Jensen's inequality and the fact that the eigenvalues are all nonnegative as was guaranteed by \eqref{Equation:q_Truncation_3}; \eqref{Equation:q_Truncation_5} holds by the Cauchy interlacing inequality \eqref{Equation:Cauchy_Interlacing}; and \eqref{Equation:q_Truncation_6} holds by removing the negative term. We see that
    \begin{align}
        \oldconstant{Constant:A_1}\opnorm{\sigma}_2^{2 - p}\opnorm{V''(\sigma)}_1 &\leq \oldconstant{Constant:Bound}\oldconstant{Constant:A_1}\opnorm{\sigma}_2^{2 - p}\big(\opnorm{\sigma}_{q_1 - 2}^{q_1 - 2} + \opnorm{\sigma}_{q_2 - 2}^{q_2 - 2}\big) \label{Equation:q_Truncation_7} \\
        &\leq \oldconstant{Constant:Bound}\oldconstant{Constant:A_1}\opnorm{\sigma}_2^{1 - p}\big(\opnorm{\sigma}_{2q_1 - 2}^{q_1 - 1} + \opnorm{\sigma}_{2q_2 - 2}^{q_2 - 1}\big) \label{Equation:q_Truncation_8} \\
        &\leq \exp\Big[\sqrt{2}\oldconstant{Constant:Bound}^2\oldconstant{Constant:A_1}\Bpar{1 + \oldconstant{Constant:Bound}^{-1}\opnorm{\sigma}_2^{1 - p}\big(\opnorm{\sigma}_{2q_1 - 2}^{2q_1 - 2} + \opnorm{\sigma}_{2q_2 - 2}^{2q_2 - 2}\big)^{1/2}}\Big] \label{Equation:q_Truncation_9} \\
        &\leq \exp\Big[\sqrt{2}\oldconstant{Constant:Bound}^2\oldconstant{Constant:A_1}\Big(1 + \frac{\opnorm{V'(\sigma)}_2}{\opnorm{\sigma}_2^{p - 1}}\Big)\Big], \label{Equation:q_Truncation_10}
    \end{align}
    where \eqref{Equation:q_Truncation_7} holds by condition \ref{Condition:V}; \eqref{Equation:q_Truncation_8} holds by Jensen's inequality (recall $q_2 \geq q_1 > p \geq 2$) $\opnorm{\sigma}_2 \leq \opnorm{\sigma}_{2q_1 - 2} \leq \opnorm{\sigma}_{2q_2 - 2}$; \eqref{Equation:q_Truncation_9} holds by the elementary inequality $(a + b)^2 \leq 2(a^2 + b^2)$ and $x \leq e^x$ for all $x \geq 0$; last, \eqref{Equation:q_Truncation_10} again holds by condition \ref{Condition:V}. Therefore, we see the first term in \eqref{Equation:q_Truncation_6} is dominated by the desired function in \eqref{Equation:Rough_Upper_Bound_Condition}. By taking expectation on \eqref{Equation:q_Truncation_6}, we may conclude that $w_N$ satisfies \eqref{Equation:Rough_Upper_Bound_Condition} by \eqref{Equation:Moment_of_GOE}. Since both \eqref{Equation:Rough_Lower_Bound_Condition} and \eqref{Equation:Rough_Upper_Bound_Condition} are satisfied, our proposition holds by Lemma \ref{Lemma:General_Truncation}.
\end{proof}

\section{Establishing the Variational Formula}\label{Section:Variational_Formula}

\noindent With the estimations in Proposition \ref{Proposition:Main_Truncation} and Proposition \ref{Proposition:q_Truncation}, we can now analyze the truncated version of \eqref{Equation:Key_Formula_2},
\begin{equation}
    \frac{1}{\sqrt{p}}\Bigl(\frac{p - 1}{2\pi}\Bigr)^{N/2}\int_{\Omega(u)} \mathbbm{1}_{\{\opnorm{\sigma}_{2q_2 - 2} \leq M\}} \frac{|\lan{\sigma, v(\sigma)}|}{N\opnorm{\sigma}_2^{N + p}} e^{-\oldconstant{Constant:Exp}Nf_N(\sigma)}\mathbb{E}\bigl[|\det Q_N^K(\sigma)|\bigr] \,d\sigma.\label{Equation:Kac-Rice-Final-Formula_2}
\end{equation}
Recall the space $\mathfrak{D}(u)$, the function $g$, and the functional $\varphi$ defined in Section \ref{sec:introduction}. We begin by introducing their modifications.

\begin{definition}\label{def:functions}
$\,$
    \begin{enumerate}[label = (\roman*)]
        \item For all $u \geq 0$, define the set
        \begin{equation}\label{Equation:F(u)}
            \mathfrak{F}(u) = \big\{(t, \mu) \in [0, \infty) \times \mathcal{P}(\mathbb{R}) \,\big\vert\, \mathbb{E}_\mu\big[p^{-1}tXV'(tX) - V(tX)\big] \geq u\big\}.
        \end{equation}
        
        \item Fix $K\in [0,\infty].$ For all $t \geq 0$, define the auxiliary function $g_{t, K} \colon \mathbb{R} \to \mathbb{R}$ by
        $g_{t, K}(x) = 
            \oldconstant{Constant:A_1}(t^{2 - p}V''(tx) \wedge K)$ when $t>0$ and $g_{t, K}(x)\equiv 0$ when $t=0.$
        Note that $g_{t, \infty} = g_t$ and that for any measure $\mu \in \mathcal{P}(\mathbb{R})$, we have
        \begin{equation}\label{Equation:Support_Control}
            \text{supp}\,((g_{t, K})_\ast\mu) \subseteq g_{t, K}(\mathbb{R}) = \left[0, \oldconstant{Constant:A_1}K\right].
        \end{equation}
        \item Set the mapping $\psi \colon [0, \infty) \times \mathbb{R} \to \mathbb{R}$ by
        \[\psi(t, x) = 
        \begin{cases}
            \oldconstant{Constant:Exp}pt^{2 - 2p}V'(tx)^2, & \mbox{if $t>0$ and $x\in \R$}, \\
            0, & \mbox{if $t=0$ and $x\in \R$}.
        \end{cases}\]
        Note that $\psi(t, \cdot)$ is an even function. Set the functions $\phi_1 \colon [0, \infty) \times \mathcal{P}_{q_2}(\mathbb{R}) \to [0 , \infty)$ and $\phi_2 \colon [0, \infty) \times \mathcal{P}(\mathbb{R}) \to [0, \infty]$  by
        \begin{align*}
            \phi_1(t, \mu) &= \oldconstant{Constant:Exp}(p - 1)t^{2 - 2p}\mathbb{E}_\mu[XV'(tX)]^2
        \end{align*}
        and
        \begin{align*}
            \phi_2(t, \mu) &= \begin{cases}
                \oldconstant{Constant:Exp}pt^{2 - 2p}\mathbb{E}_\mu[V'(tX)^2] = \mathbb{E}_\mu[\psi(t, X)], & \mbox{if $(t,\mu)\in [0,\infty)\times\mathcal{P}_{2q_2 - 2}(\mathbb{R})$}, \\
                \infty, & \text{ otherwise}.
            \end{cases}
        \end{align*}
        Also, when $0<K<\infty,$ we set $\phi_{3, K} \colon [0, \infty) \times \mathcal{P}(\mathbb{R}) \to \mathbb{R}$ by
        \begin{align*}
            \phi_{3, K}(t, \mu) &= \int_{\mathbb{R}} \log|\lambda| \big((g_{t, K})_\ast\mu \boxplus \mu_{\text{sc}}\big)(d\lambda)
        \end{align*} 
        and when $K=\infty$, we restrict to the domain $[0, \infty) \times \bigcup_{s > 0} \mathcal{P}_s(\mathbb{R})$ and set
        \[\phi_{3, \infty}(t, \mu) = \int_{\mathbb{R}} \log|\lambda| \big((g_t)_\ast\mu \boxplus \mu_{\text{sc}}\big)(d\lambda).\]
        Finally, set $\varphi_K = \phi_1 - \phi_2 + \phi_{3, K}$ and
        \[\mathcal{I}_K(t, \mu) = \frac{1}{2}\log(p - 1) + \frac{1}{2} + \varphi_K(t, \mu) - \text{KL}(\mu \,\| \,\mu_{\text{Norm}}).\]
    \end{enumerate}
\end{definition}

Recall the definitions of $\mathcal{I}(t,\mu)$ and $\mathcal{I}(\mu)$ from \eqref{Equation:Funtional_with_t} and \eqref{Equation:Funtional_no_t}. By performing change of variable $\nu = S_t(\mu)$, we can rewrite
\begin{align}
    &\sup\{\,\mathcal{I}_\infty(t, \mu) \mid m_2(\mu) = 1, \mu \in \mathcal{P}_{2q_2 - 2}(\mathbb{R}), (t, \mu) \in \mathfrak{F}(u)\} \notag \\
    &= \sup\{\,\mathcal{I}_\infty(t, S_{1/t}(\nu)) \mid m_2(\nu) = t^2, \nu \in \mathcal{P}_{2q_2 - 2}(\mathbb{R}) \cap \mathfrak{D}(u)\} \notag \\
    &= \sup\{\,\mathcal{I}(\nu) \mid \nu \in \mathcal{P}_{2q_2 - 2}(\mathbb{R}) \cap \mathfrak{D}(u)\}. \label{Equation:Reduction_t}
\end{align}

\subsection{Auxiliary Lemmas}

\noindent This subsection gathers some properties of the auxiliary functions defined above.

\begin{proposition}[Properties of $\phi_1$, $\phi_2$ and $\phi_{3, K}$]\label{Proposition:Continuity_of_Exp}\
    \begin{enumerate}[label = (\roman*)]
        \item\label{Item:Continuity_1} If $V$ satisfies condition \ref{Condition:V}, then the mapping $\phi_1$ is well-defined and is continuous on $([0, \infty) \times \mathcal{P}_s(\mathbb{R}), \mathscr{T} \otimes \mathscr{W}_s)$ for all $s \geq q_2$.
        \item\label{Item:Continuity_2} If $V$ satisfies conditions \ref{Condition:V}, then $x \mapsto \psi(t, x)$ is nonincreasing  on $(-\infty, 0]$ and nondecreasing on $[0, \infty)$ for all $t \in [0, \infty)$.
        \item\label{Item:Continuity_3} If $V$ satisfies conditions \ref{Condition:V''}, then $t \mapsto \psi(t, x)$ is increasing on $[0, \infty)$ for all $x \in \mathbb{R}$.
        \item\label{Item:Continuity_4} For $K \neq \infty$, the mapping $\phi_{3, K}$ is well-defined and is continuous on $([0, \infty) \times \mathcal{P}(\mathbb{R}), \mathscr{T} \otimes \mathscr{T}_{\text{weak}})$. Moreover, for all $(t, \mu) \in [0, \infty) \times \mathcal{P}(\mathbb{R})$,
        \begin{equation}\label{Equation:Bounded_phi_3}
            -4\oldconstant{Constant:Free_Convolution_Infty} \leq \phi_{3, K}(t, \mu) \leq 4\oldconstant{Constant:Free_Convolution_Infty} + 2\oldconstant{Constant:A_1}K + 4.
        \end{equation}
        \item\label{Item:Continuity_5} If $V$ satisfies condition \ref{Condition:V}, then $\phi_{3, \infty}$ is well-defined on $[0, \infty) \times \mathcal{P}_s(\mathbb{R})$ for all $s > 0$.
    \end{enumerate}
\end{proposition}

\begin{proof}
    For \ref{Item:Continuity_1}, observe that by \ref{Condition:V}, the integral
    \begin{equation}\label{Equation:Property_Phi_1}
        0 \leq \int_{\mathbb{R}} txV'(tx)\mu(dx) \leq \int_{\mathbb{R}} \oldconstant{Constant:Bound}(|tx|^{q_1} + |tx|^{q_2})\mu(dx) = \oldconstant{Constant:Bound}\big(t^{q_1}m_{q_1}(\mu) + t^{q_2}m_{q_2}(\mu)\big)
    \end{equation}
    is finite on $\mathcal{P}_s(\mathbb{R})$ for $s \geq q_2$. To show continuity, let $(t_n, \mu_n) \to (t_\infty, \mu_\infty)$ in $([0, \infty) \times \mathcal{P}_s(\mathbb{R}), \mathscr{T} \otimes \mathscr{W}_s)$. First we consider the case $t_\infty > 0$ and assume $t_n > 0$ of all $n \in \mathbb{N}$ without loss of generality. Set $f(x) = xV'(x)$. By the continuous mapping theorem (see \cite[Theorems 2.7 and 2.8]{billingsley2013convergence}), we have $f_\ast S_{t_n}(\mu_n) \to f_\ast S_{t_\infty}(\mu_\infty)$ in $(\mathcal{P}_s(\mathbb{R}), \mathscr{T}_{\text{weak}})$. Moreover, from the uniform integrability  \eqref{Equation:Property_Phi_1}, we conclude that 
    \[\phi_1(t_n, \mu_n) = \oldconstant{Constant:Exp}(p - 1)t_n^{-2p}m_1(f_\ast S_{t_n}(\mu_n))^2 \to \oldconstant{Constant:Exp}(p - 1)t_\infty^{-2p}m_1(f_\ast S_{t_\infty}(\mu_\infty))^2 = \phi_1(t_\infty, \mu_\infty).\]
Next, for the case $t_\infty = 0$, we see that by \eqref{Equation:Property_Phi_1},
    \[0 \leq \phi_1(t_n, \mu_n) \leq \oldconstant{Constant:Bound}^2\oldconstant{Constant:Exp}(p - 1)t^{2(q_1 - p)}\big(m_{q_1}(\mu_n) + t_n^{q_2 - q_1}m_{q_2}(\mu_n)\big)^2 \to 0=\phi_1(t_\infty, \mu_\infty)\]
   since the $q_2$-th moment of $\mu_n$ are uniformly bounded by the $\mathscr{W}_s$-convergence. Hence, we have the desired continuity.
    
    For \ref{Item:Continuity_2}, note the assumption \ref{Condition:V} implies that $V'$ and $V''$ are both nonnegative for $x \geq 0$. Consequently,
    \[\frac{\partial \psi(t, x)}{\partial x} = \oldconstant{Constant:Exp}pt^{2 - 2p}\frac{\partial }{\partial x} \bigl(V'(tx)^2\bigr) = 2\oldconstant{Constant:Exp}pt^{3 - 2p}V'(tx)V''(tx) \geq 0.\]
    Therefore, $x \mapsto \psi(t, x)$ is nondecreasing on $[0, \infty)$. Since $\psi(t,\cdot)$ is even, it follows that  $\psi(t,\cdot)$ is decreasing on $(-\infty, 0]$.
 As for \ref{Item:Continuity_3}, observe that from \eqref{Equation:V'}, 
    \[\frac{\partial \psi(t, x)}{\partial t} = 2\oldconstant{Constant:Exp}pt^{1 - 2p}\big((1 - p)V'(tx)^2 + txV'(tx)V''(tx)\big) \geq 2\oldconstant{Constant:Exp}(q - p)pt^{1 - 2p}V'(tx)^2 \geq 0.\]
 Therefore, $t \mapsto \psi(t, x)$ is nondecreasing.
    
    For the finiteness assertion in \ref{Item:Continuity_4}, note that from \eqref{Equation:Support_Control} and Theorem \ref{Theorem:Free_Convolution_Semi_Circle}, $(g_{t, K})_\ast \mu \boxplus \mu_{\text{sc}}$ for $K<\infty$ has a bounded density and its support is contained in the interval $[-\oldconstant{Constant:A_1}K - 2,\oldconstant{Constant:A_1}K + 2]$, which implies that $\phi_{3,K}$ is finite. To show continuity, let $(t_n, \mu_n) \to (t_\infty, \mu_\infty)$ in $([0, \infty) \times \mathcal{P}(\mathbb{R}), \mathscr{T} \otimes \mathscr{T}_{\text{weak}})$ and denote the measures $\nu_n = (g_{t_n, K})_\ast\mu_n \boxplus \mu_{\text{sc}}$ for $n \in \mathbb{N} \cup \{\infty\}$. Note that the difference
    \begin{align}
        \Big|\int_{\mathbb{R}} \log|\lambda| (\nu_n - \nu_\infty)(d\lambda)\Big| &\leq \int_{\mathbb{R}} \bpar{\log_\eta(\lambda) + \log|\lambda|} (\nu_n - \nu_\infty)(d\lambda) + \Big|\int_{\mathbb{R}} \log_\eta(\lambda) (\nu_n - \nu_\infty)(d\lambda)\Big| \notag \\
        &\leq \int_{\mathbb{R}} \dfrac{1}{2}\log\Bpar{1 + \frac{\eta^2}{\lambda^2}}(f_{\nu_n}(\lambda) + f_{\nu_\infty}(\lambda))\,d\lambda + \dfrac{1}{2\eta}W_1(\nu_n, \nu_\infty). \label{Equation:Proposition:Continuity_of_Exp_1}
    \end{align}
    Here, for the first term, we have
    \[\int_{\mathbb{R}} \log\Bpar{1 + \frac{\eta^2}{\lambda^2}}(f_{\nu_n}(\lambda) + f_{\nu_\infty}(\lambda))\,d\lambda \leq 2\oldconstant{Constant:Free_Convolution_Infty}\int_{\{|\lambda| \leq \oldconstant{Constant:A_1}K + 2\}} \log\Bpar{1 + \frac{\eta^2}{\lambda^2}}\,d\lambda,\]
    which converges to $0$ as $\eta \to 0$ by dominated convergence theorem. For the second term, we see by Theorem \ref{Theorem:Triangle_Equation_Free_Convolution} and \cite[Theorem 11.3.3]{dudley2018real} that
    \begin{align*}
        W_1(\nu_n, \nu_\infty) &\leq (\oldconstant{Constant:A_1}K + 2) \cdot d_{\text{BL}}(\nu_n, \nu_\infty) \\
        &\leq (\oldconstant{Constant:A_1}K + 2) \cdot 2d_{\text{L}}(\nu_n, \nu_\infty) \leq 2(\oldconstant{Constant:A_1}K + 2) \cdot d_{\text{L}}\left((g_{t_n, K})_\ast\mu_n, (g_{t_\infty, K})_\ast\mu_\infty\right),
    \end{align*}
    which converges to $0$ ad $n \to \infty$ since $g_{t, K}$ is continuous and bounded. Therefore,
    \[\Big|\int_{\mathbb{R}} \log|\lambda| (\nu_n - \nu_\infty)(d\lambda)\Big| \leq o_{\eta \to 0}(1) + \frac{o_{n \to \infty}(1)}{\eta},\]
    concluding the desired continuity. The second assertion follows by 
    \[\phi_{3, K}(t, \mu) \leq -2\oldconstant{Constant:Free_Convolution_Infty}\int_{\{|\lambda| \leq 1\}} \log|\lambda| \,d\lambda + \int_{\{|\lambda| \geq 1\}} |\lambda| \big((g_{t, K})_\ast\mu \boxplus \mu_{\text{sc}}\big)(d\lambda) \leq 2(2\oldconstant{Constant:Free_Convolution_Infty} + \oldconstant{Constant:A_1}K + 2)\]
    and
    \[\phi_{3, K}(t, \mu) \geq -\Big|\int_{\{|\lambda| \leq 1\}} \log|\lambda| \big((g_{t, K})_\ast\mu \boxplus \mu_{\text{sc}}\big)(d\lambda)\Big| \geq -4\oldconstant{Constant:Free_Convolution_Infty}.\]

    Finally, for \ref{Item:Continuity_5}, we write
    \[\phi_{3, \infty}(t, \mu) = \int_{\{|\lambda| \leq 1\}} \log |\lambda| \big((g_t)_\ast\mu \boxplus \mu_{\text{sc}}\big)(d\lambda) + \int_{\{|\lambda| > 1\}} \log |\lambda| \big((g_t)_\ast\mu \boxplus \mu_{\text{sc}}\big)(d\lambda).\]
    The first term is finite since $(g_{t, K})_\ast\mu \boxplus \mu_{\text{sc}}$ has bounded density with respect to Lebesgue measure by Theorem \ref{Theorem:Free_Convolution_Semi_Circle}. The second term can be bounded by the $m_{s/(q_2 - 2)}\bpar{(g_{t, K})_\ast\mu \boxplus \mu_{\text{sc}}}$ up to a constant. By \eqref{Equation:Free_Addition_m_2} and condition \ref{Condition:V}, we see that such $(s/(q_2 - 2))$ moment is finite and this completes our proof.
\end{proof}

The following proposition shows that $\varphi_{3, K}$ is a good approximation of $\varphi_{3, \infty}$ on appropriate sets:

\begin{proposition}\label{Proposition:Approximation_of_phi}
    If $V$ satisfies condition \ref{Condition:V}, then for all $d > c > 0$, we have
    \[\lim_{K \to \infty} \sup_{t \in [c, d], m_2(\mu) = 1} |\phi_{3, K}(t, \mu) - \phi_{3, \infty}(t, \mu)| = 0.\]
\end{proposition}

\begin{proof}
    Note that for all $t \in [c, d]$ and $m_2(\mu) = 1$, we have
    \begin{equation}\label{Equation:Approximation_of_phi_1}
        \phi_{3, K}(t, \mu) - \phi_{3, \infty}(t, \mu) = \int_{\mathbb{R}} \log|\lambda|\left((g_{t, K})_\ast\mu \boxplus \mu_{\text{sc}} - (g_{t, \infty})_\ast\mu \boxplus \mu_{\text{sc}}\right)(d\lambda).
    \end{equation}
    Denote $\nu_K = (g_{t, K})_\ast\mu \boxplus \mu_{\text{sc}}$ and $\nu_\infty = (g_{t, \infty})_\ast\mu \boxplus \mu_{\text{sc}}$. 
    For the portion near origin in \eqref{Equation:Approximation_of_phi_1}, first note that both $\nu_K$ and $\nu_\infty$ admits bounded density by Theorem \ref{Theorem:Free_Convolution_Semi_Circle} and hence
    \[\left|\int_0^\delta \log|\lambda|(\nu_K - \nu_\infty)(d\lambda)\right| \leq \int_0^\delta -\log(\lambda)(\nu_K + \nu_\infty)(d\lambda) \leq 2\oldconstant{Constant:Free_Convolution_Infty}\delta(1 - \log(\delta)),\]
    where we may take $\delta > 0$ small enough so this term is also small. For the tail in \eqref{Equation:Approximation_of_phi_1}, we observe first by Theorem \ref{Theorem:Triangle_Equation_Free_Convolution}, there exist $\newconstant\label{Constant:Main_Lower_Bound} > 0$ such that if $K \geq 1$,
    \begin{align}
        d_{\text{KS}}(\nu_K, \nu_\infty) &\leq d_{\text{KS}}\left((g_{t, K})_\ast\mu, (g_{t, \infty})_\ast\mu\right) \notag \\
        &= \sup\{\mathbb{P}_\mu\left(\oldconstant{Constant:A_1}\left(t^{2 - p}V''(tX) \wedge K\right) \leq x < \oldconstant{Constant:A_1}t^{2 - p}V''(tX)\right) \mid x \in \mathbb{R}\} \notag \\
        &\leq \mathbb{P}_\mu\big(t^{2 - p}V''(tX) \geq K\big) \notag \\
        &\leq \mathbb{P}_\mu\big(\oldconstant{Constant:Bound}t^{2 - p}[(t|X|)^{q_1 - 2} + (t|X|)^{q_2 - 2}] \geq K\big) \label{Equation:Limit_K_1} \\
        &\leq \mathbb{P}_\mu\big(\oldconstant{Constant:Bound}t^{q_1 - p}|X|^{q_1 - 2} \geq K/2\big) + \mathbb{P}\big(\oldconstant{Constant:Bound}t^{q_2 - p}|X|^{q_2 - 2} \geq K/2\big) \notag \\
        &\leq \frac{(\oldconstant{Constant:Bound}d^{q_1 - p})^{2/(q_1 - 2)}}{(K/2)^{2/(q_1 - 2)}} + \frac{(\oldconstant{Constant:Bound}d^{q_2 - p})^{2/(q_2 - 2)}}{(K/2)^{2/(q_2 - 2)}} \label{Equation:Limit_K_2} \\
        &\leq \oldconstant{Constant:Main_Lower_Bound}K^{-2/(q_2 - 2)}, \label{Equation:Limit_K_3}.
    \end{align}
    Here \eqref{Equation:Limit_K_1} holds by condition \ref{Condition:V} and \eqref{Equation:Limit_K_2} follows  by the Markov inequality, $m_2(\mu) = 1$, and $0 < t \leq d$. On the other hand, as before, by \eqref{Equation:Support_Control} and Theorem \ref{Theorem:Free_Convolution_Semi_Circle}, the support of $\nu_K$ is contained in the interval $[-\oldconstant{Constant:A_1}K - 2,\oldconstant{Constant:A_1}K + 2]$. Now, for any $\lambda > 4$, we choose $K >2\oldconstant{Constant:A_1}^{-1}$ so that $\lambda \in (\oldconstant{Constant:A_1}K + 2, 2\oldconstant{Constant:A_1}K)$ and thus, by \eqref{Equation:Limit_K_3},
        \begin{align*}
            1 - F_{\nu_\infty}(\lambda) = F_{\nu_K}(\lambda) - F_{\nu_\infty}(\lambda) \leq \oldconstant{Constant:Main_Lower_Bound}K^{-2/(q_2 - 2)} \leq \oldconstant{Constant:Main_Lower_Bound}\Big(\frac{2\oldconstant{Constant:A_1}}{\lambda}\Big)^{2/(q_2 - 2)}.
        \end{align*}
        Therefore, the tail in \eqref{Equation:Approximation_of_phi_1} satisfies
        \begin{align*}
            \int_\delta^\infty \log|\lambda|(\nu_K - \nu_\infty)(d\lambda) &= \left.\log(\lambda)\left[F_{\nu_K}(\lambda) - F_{\nu_\infty}(\lambda)\right]\right|_{ \delta}^{\infty} - \int_\delta^\infty \dfrac{1}{\lambda}\left[F_{\nu_K}(\lambda) - F_{\nu_\infty}(\lambda)\right]\,d\lambda \\
            &\leq \lim_{\lambda \to \infty} \log(\lambda) \oldconstant{Constant:Main_Lower_Bound}\Big(\frac{2\oldconstant{Constant:A_1}}{\lambda}\Big)^{2/(q_2 - 2)} - \log(\delta)d_{\text{KS}}(\nu_K, \nu_\infty) \\
            & + \int_\delta^{\oldconstant{Constant:A_1}K + 3} \dfrac{1}{\lambda} \cdot \oldconstant{Constant:Main_Lower_Bound}K^{-2/(q_2 - 2)} d\lambda + \int_{\oldconstant{Constant:A_1}K + 3}^\infty \dfrac{1}{\lambda} \cdot \oldconstant{Constant:Main_Lower_Bound}\Big(\frac{2\oldconstant{Constant:A_1}}{\lambda}\Big)^{2/(q_2 - 2)}\,d\lambda,
        \end{align*}
        which obviously converges to $0$ as $K \to \infty$.
\end{proof}

For all $K \in [0, \infty]$ and $\sigma \in \mathbb{R}^N$, define
    \begin{equation}
        s_N^K(\sigma) = \phi_{3, K}\Bpar{1, \frac{1}{N}\sum_{i = 1}^N \delta_{\sigma_i}} = \int_{\mathbb{R}} \log|\lambda| (\mu_{D_N^K(\sigma)} \boxplus \mu_{\text{sc}})(d\lambda).
    \end{equation}

\begin{proposition}\label{Proposition:Application_of_Ben}
    If $K <\infty$, we have
    \[\lim_{N \to \infty} \sup_{\sigma \in \mathbb{R}^N} \Bpar{\dfrac{1}{N}\log \mathbb{E}\left|\det Q_N^K(\sigma)\right| - s_N^K(\sigma)} = 0.\]
\end{proposition}

\begin{proof}
    Note that the convergence is uniform in $\sigma \in \mathbb{R}^N$ since the convergent rate only depends on 
    \[\sup_{N \in \mathbb{N}} m_\infty(\mu_{D_N^K(\sigma)}) = \sup_{N \in \mathbb{N}}\|D_N^K(\sigma)\|_{\text{op}} \leq \oldconstant{Constant:A_1}K.\]
    Our assertion holds by using Theorem \ref{Theorem:Asymptote_Bounded}.
\end{proof}

\subsection{A Varadhan-type Theorem}

    \noindent 
   Recall $\psi$ from Definition \ref{def:functions}(iii). For all $a, t > 0$, define the tilted measure 
    \begin{equation}\label{Equation:Tilted_Measure}
        \mu_{a, t}(dx) = \frac{e^{-\psi(t, x/a)}}{Z_{a, t}}\mu_{\text{Norm}}(dx),
    \end{equation}
    where
    \begin{align*}
        Z_{a, t} = \int_{\mathbb{R}} e^{-\psi(t, x/a)} \mu_{\text{Norm}}(dx).
    \end{align*}
Throughout the section, we will use $g = (g_i)_{i = 1}^N$ and $x^{a, t} = (x_i^{a, t})_{i = 1}^N$ to denote i.i.d.\ samples from $\mu_{\text{Norm}}$ and $\mu_{a, t}$, respectively. For any $x \in \mathbb{R}^N$, set the empirical measures
    \[L_{x, N} = \dfrac{1}{N}\sum_{i = 1}^N \delta_{x_i} \quad \text{and} \quad \nu_{x, N} = \dfrac{1}{N}\sum_{i = 1}^N \delta_{x_i/\sqrt{m_2(L_{x, N})}}.\]  Note that $m_2(\nu_{x,N})=1$ and $\nu_{g,N}$ is indeed a uniform measure on the unit sphere $\mathbb{S}^{N-1}.$ For convenience, denote $L_{a, t, N} = L_{x^{a, t}, N}$ and $\nu_{a, t, N} = \nu_{x^{a, t}, N}$. Our goal of this subsection is to prove the following Varadhan integral inequalities.

\begin{theorem}\label{Theorem:Varadhan_Tilted}
    Fix $0 < s < 2q_2 - 2$, $a > 0$, and $d > c > 0$. Let $\Phi \colon ([c, d] \times \mathcal{P}_s(\mathbb{R}), \mathscr{T} \otimes \mathscr{W}_s) \to \mathbb{R}$ be a continuous function and the random variable $T$ is uniformly sampled from $[c, d]$. Set the function $\Psi = \Phi - \phi_2$ and assume
    $\Psi$ satisfies the tail condition
    \begin{equation}\label{Equation:Varadhan_Tilted_Condition_1}
        \limsup_{M \to \infty} \sup_{t \in [c, d], m_2(\mu) = 1, m_{2q_2 - 2}(\mu) \geq M} \Psi(t, \mu) =  - \infty
    \end{equation}
    and that there exists some $\newconstant\label{Constant:Varadhan} > 0$ such that
    \begin{equation}\label{Equation:Varadhan_Tilted_Condition_2}
        \sup_{t \in [c, d], \mu \in \mathcal{P}_s(\mathbb{R}), m_2(\mu) = 1} \Psi(t, \mu) \leq \oldconstant{Constant:Varadhan}.
    \end{equation}
    If $V$ satisfies \ref{Condition:V}, we have 
    \begin{enumerate}[label = (\roman*)]
        \item for all open sets $\mathfrak{U} \subseteq ([c, d] \times \mathcal{P}_s(\mathbb{R}), \mathscr{T} \otimes \mathscr{W}_s)$,
        \[\liminf_{N \to \infty} \frac{1}{N}\log \mathbb{E}\Big[e^{N\Psi(T, \nu_{g, N})}\mathbbm{1}_{\{(T, \nu_{g, N}) \in \mathfrak{U}\}}\Big]  \geq \sup_{m_2(\mu) = 1, (t, \mu) \in \mathfrak{U}} \bpar{\Psi(t, \mu) - \text{KL}(\mu \,\| \,\mu_{\text{Norm}})},\]
        \item for all closed set $\mathfrak{C} \subseteq ([c, d] \times \mathcal{P}_s(\mathbb{R}), \mathscr{T} \otimes \mathscr{W}_s)$,
        \[\limsup_{N \to \infty} \frac{1}{N}\log \mathbb{E}\Big[e^{N\Psi(T, \nu_{g, N})}\mathbbm{1}_{\{(T, \nu_{g, N}) \in \mathfrak{C}\}}\Big] \leq \sup_{m_2(\mu) = 1, (t, \mu) \in \mathfrak{C}} \bpar{\Psi(t, \mu) - \text{KL}(\mu \,\| \,\mu_{\text{Norm}})}.\]
        \end{enumerate}
\end{theorem}

The proof of Theorem \ref{Theorem:Varadhan_Tilted} is based on the following large deviation principle, whose proof is deferred to Subsection \ref{subsec:LDP}.

    \begin{theorem}\label{Theorem:LDP_Tilted}
    Fix $0 < s < 2q_2 - 2$ and $a > 0$. Let $T$ be uniformly sampled from $[c, d]$ where $d > c > 0$. If $V$ satisfies \ref{Condition:V} and \ref{Condition:V''}, then the sequence 
    $(m_2(L_{a, T, N}), \nu_{a, T, N}, T)$ satisfies a large deviation principle in the space $([0, \infty) \times \mathcal{P}_s(\mathbb{R}) \times [c, d], \mathscr{T} \otimes \mathscr{W}_s \otimes \mathscr{T})$ with speed $N$ and a good rate function given by
    \[J_a(r, \mu, t) = \begin{cases}
        \log Z_{a, t} + \text{KL}(\mu \,\| \,\mu_{\text{Norm}}) + \frac{1}{2}(r - 1) - \frac{1}{2}\log r + \phi_2\big(t, S_{\sqrt{r}/a}(\mu)\big), & m_2(\mu) = 1, \\
        \infty, & m_2(\mu) \neq 1.
    \end{cases}\]
\end{theorem}

\begin{proof}[\bf Proof of Theorem \ref{Theorem:Varadhan_Tilted}]
    First, we deal with the lower bound. Note that for any $a >0$, 
    \begin{align}
        &\mathbb{E}\Big[e^{N\Psi(T, \nu_{g, N})}\mathbbm{1}_{\{(T, \nu_{g, N}) \in \mathfrak{U}\}}\Big] \notag \\
        &\geq \mathbb{E}\Big[e^{N\Psi(\nu_{g, N})} \mathbbm{1}_{\{m_2(L_{g, N}) > a^2\}}\mathbbm{1}_{\{(T, \nu_{g, N}) \in \mathfrak{U}\}}\Big] \notag \\
        &= \mathbb{E}\Big[\int_{\mathbb{R}^N} e^{N\Phi(T, \nu_{g, N}) - \sum_{i = 1}^N \psi(T, \sqrt{N}g_i/\|g\|_2)}\mathbbm{1}_{\{\|g\|_2/\sqrt{N} > a\}} \mathbbm{1}_{\{(T, \nu_{g, N}) \in \mathfrak{U}\}}\mu_{\text{Norm}}^{\otimes N}(\text{d}g)\Big] \notag \\
        &\geq \mathbb{E}\Big[\int_{\mathbb{R}^N} e^{N\Phi(T, \nu_{g, N}) - \sum_{i = 1}^N \psi(T, g_i/a)}\mathbbm{1}_{\{\|g\|_2/\sqrt{N} > a\}} \mathbbm{1}_{\{(T, \nu_{g, N}) \in \mathfrak{U}\}} \mu_{\text{Norm}}^{\otimes N}(\text{d}g)\Big] \label{Equation:Varadhan_Lower_1} \\
        &= \mathbb{E}\Big[Z_{a, T}^N \int_{\mathbb{R}^N} e^{N\Phi(T, \nu_{x, N})} \mathbbm{1}_{\{m_2(L_{x, N}) > a^2\}} \mathbbm{1}_{\{(T, \nu_{x, N}) \in \mathfrak{U}\}} \mu_{a, T}^{\otimes N}(\text{d}x)\Big] \label{Equation:Varadhan_Lower_2} \\
        &= \mathbb{E}\Big[e^{N(\log Z_{a, T} + \Phi(T, \nu_{a, T, N}))}\mathbbm{1}_{\{m_2(L_{a, T, N}) > a^2\}}\mathbbm{1}_{\{(T, \nu_{a, T, N}) \in \mathfrak{U}\}}\Big], \notag
    \end{align}
    where \eqref{Equation:Varadhan_Lower_1} holds since $\psi(t, x)$ is increasing in $|x|$ by Proposition \ref{Proposition:Continuity_of_Exp} and \eqref{Equation:Varadhan_Lower_2} holds by the definition \eqref{Equation:Tilted_Measure}. 
    
    Now, note that $(t, \mu) \mapsto \log Z_{a, t} + \Phi(t, \mu)$ is continuous on $([c, d] \times \mathcal{P}_s(\mathbb{R}), \mathscr{T} \otimes \mathscr{W}_s)$ (continuity of $t \mapsto Z_{a, t}$ follows from dominating convergence theorem). Therefore, in view of Theorem \ref{Theorem:LDP_Tilted}, we can apply the lower bound of Varadhan's lemma \eqref{Equation:Varadhan_Original_Lower} to the pair $(m_2(L_{a, T, N}), \nu_{a, T, N}, T)$ with respect to the open set $(a^2, \infty) \times \mathfrak{U}$ to see that
    \begin{align}
        &\liminf_{n \to \infty} \frac{1}{N}\log \mathbb{E}\Big[e^{N\Psi(T, \nu_{g, N})}\mathbbm{1}_{\{(T, \nu_{g, N}) \in \mathfrak{U}\}}\Big] \notag \\
        &\geq \sup_{m_2(\mu) = 1, (t, \mu) \in \mathfrak{U}, r > a^2} \Big[\log Z_{a, t} + \Phi(\mu) - J_a(r, \mu, t)\Big] \notag \\
        &= \sup_{m_2(\mu) = 1, (t, \mu) \in \mathfrak{U}, r > a^2} \Big[\Phi(\mu) - \text{KL}(\mu \,\| \,\mu_{\text{Norm}}) - \frac{1}{2}(r - 1) + \frac{1}{2}\log r - \int_{\mathbb{R}} \psi\Bpar{t, \dfrac{\sqrt{r}x}{a}} \mu(dx)\Big]. \label{Equation:Varadhan_Lower_3}
    \end{align}
    Taking the supremum over $a > 0$ on the right hand side yields
    \begin{align}
        &\sup_{a > 0} \sup_{r > a^2} \Big[- \frac{1}{2}(r - 1) + \frac{1}{2}\log r - \int_{\mathbb{R}} \psi\Bpar{t, \dfrac{\sqrt{r}x}{a}}\mu(dx)\Big] \notag \\
        &= \sup_{r > 0}\Big(- \frac{1}{2}(r - 1) + \frac{1}{2}\log r - \inf_{a < \sqrt{r}} \int_{\mathbb{R}} \psi\Bpar{t, \dfrac{\sqrt{r}x}{a}}\mu(dx)\Big)\Big] \notag \\
        &= -\int_{\mathbb{R}} \psi(t, x)\mu(dx) + \sup_{r > 0}\Bpar{-\dfrac{1}{2}(r - 1) + \dfrac{1}{2}\log r} \label{Equation:Varadhan_Lower_4} \\
        &= -\int_{\mathbb{R}} \psi(t, x)\mu(dx), \label{Equation:Varadhan_Lower_5}
    \end{align}
    where \eqref{Equation:Varadhan_Lower_4} holds since the optimal choice of $a$ is $a \nearrow \sqrt{r}$ (recall that $\psi(t, x/a)$ is monotinically decreasing in $a$ by Proposition \ref{Proposition:Continuity_of_Exp}) and \eqref{Equation:Varadhan_Lower_5} holds by simple calculus. Combining \eqref{Equation:Varadhan_Lower_3} and \eqref{Equation:Varadhan_Lower_5}, we obtain the asserted lower bound.

    Next, we treat the upper bound. Fix $M, K \geq 1$ and $0 < \epsilon, \delta < 1$. Since $[\delta, K]$ is compact, there exists for $a_1, a_2, \ldots, a_m  \in [\sqrt{\delta}, \sqrt{K}]$ such that $[\delta, K] \subseteq \bigcup_{i = 1}^m [a_j^2, (a_j + \epsilon)^2]$. We then have
    \begin{align}
        &\mathbb{E}\Big[e^{N\Psi(T, \nu_{g, N})}\mathbbm{1}_{\{(T, \nu_{g, N}) \in \mathfrak{C}\}}\Big] \notag \\
        &\leq \mathbb{E}\Big[e^{N\Psi(T, \nu_{g, N})}\mathbbm{1}_{\{m_{2q_2 - 2}(\nu_{g, N}) \geq M\}}\Big] + \mathbb{E}\Big[e^{N\Psi(T, \nu_{g, N})}\mathbbm{1}_{\{(T, \nu_{g, N}) \in \mathfrak{C}\} \cap \{m_{2q_2 - 2}(\nu_{g, N}) \leq M\}}\Big] \notag \\
        &\leq \mathbb{E}\Big[e^{N\Psi(T, \nu_{g, N})}\mathbbm{1}_{\{m_{2q_2 - 2}(\nu_{g, N}) \geq M\}}\Big] \label{Equation:Varadhan_Upper_1} \\
        &\hspace{1cm} + \sum_{i = 1}^m \mathbb{E}\Big[e^{N\Psi(T, \nu_{g, N})}\mathbbm{1}_{\{(T, \nu_{g, N}) \in \mathfrak{C}\} \cap \{m_{2q_2 - 2}(\nu_{g, N}) \leq M\} \cap \{a_j^2 \leq m_2(L_{g, N}) \leq (a_j + \epsilon)^2\}}\Big] \label{Equation:Varadhan_Upper_2} \\
        &\hspace{1cm} + \mathbb{E}\Big[e^{N\Psi(T, \nu_{g, N})}\mathbbm{1}_{\{m_2(L_{g, N}) < \delta\}}\Big] + \mathbb{E}\Big[e^{N\Psi(T, \nu_{g, N})}\mathbbm{1}_{\{m_2(L_{g, N}) \geq K\}}\Big] \label{Equation:Varadhan_Upper_3}.
    \end{align}
    Here, the $N^{-1}\log$ asymptote of \eqref{Equation:Varadhan_Upper_1} is obviously bounded above by
    \[h_1(M) := \sup_{t \in [c,d], m_2(\mu) = 1, m_{2q_2 - 2}(\mu) \geq M} \Psi(t, \mu) 
    \]
    and due to \eqref{Equation:Varadhan_Tilted_Condition_2}, the $N^{-1}\log$ of \eqref{Equation:Varadhan_Upper_3} is controlled by
    \[h_2(\delta, K) :=  \oldconstant{Constant:Varadhan} + \limsup_{N \to \infty} \frac{1}{N}\log((\mathbb{P}(m_2(L_{g, N}) < \delta)) + \mathbb{P}(m_2(L_{g, N}) \geq K)).\]
    As for \eqref{Equation:Varadhan_Upper_2}, denote for $a, M > 0$,
    \[\mathfrak{C}_{a, M} = \{(r, \mu, t) \in [0, \infty) \times \mathcal{P}_s(\mathbb{R}) \times [c, d] \mid (t, \mu) \in \mathfrak{C}, m_{2q_2 - 2}(\mu) \leq M, a^2 \leq r \leq (a + \epsilon)^2\},\]
    which is compact under $\mathscr{T} \otimes \mathscr{W}_s \otimes \mathscr{T}$ by \cite[Lemma 3.14]{kim2018conditional}. The $N^{-1}\log$ asymptote of \eqref{Equation:Varadhan_Upper_2} is bounded above by
    \begin{align}
       &\limsup_{N \to \infty}\frac{1}{N} \log \sum_{j = 1}^m \mathbb{E}\Big[\int_{\mathbb{R}^N} e^{N\Phi(T, \nu_{g, N}) - \sum_{i = 1}^N \psi(T, \sqrt{N}g_i/\|g\|_2)}\mathbbm{1}_{\{(m_2(L_{g, N}), \nu_{g, N}, T) \in \mathfrak{C}_{a_j, M}\}} \mu_{\text{Norm}}^{\otimes N}(dg)\Big] \notag \\
        &\leq \limsup_{N \to \infty}\frac{1}{N} \log \sum_{j = 1}^m \mathbb{E}\Big[e^{N(\log Z_{a_j + \epsilon, T} + \Phi(T, \nu_{a_j + \epsilon, T, N}))}\mathbbm{1}_{\{(m_2(L_{a_j + \epsilon, T, N}), \nu_{a_j + \epsilon, T, N}, T) \in \mathfrak{C}_{a_j, M}\}}\Big]  \label{Equation:Varadhan_Upper_4} \\
        &\leq \max_{1 \leq j \leq m} \sup_{m_2(\mu) = 1, (r, \mu, t) \in \mathfrak{C}_{a_j, M}} \Big[\log Z_{a_j + \epsilon, t} + \Phi(t, \mu) - J_{a_j + \epsilon}(r, \mu, t)\Big] \label{Equation:Varadhan_Upper_5} \\
        &\leq \sup_{\delta \leq a^2 \leq K, m_2(\mu) = 1, (r, \mu, t) \in \mathfrak{C}_{a, M}} \Bsq{\Phi(t, \mu) - \text{KL}(\mu \,\| \,\mu_{\text{Norm}}) - \frac{1}{2}(r - 1) + \frac{1}{2}\log r - \phi_2\bpar{t, S_{\sqrt{r}/(a + \epsilon)}(\mu)}}, \notag
    \end{align}
    where \eqref{Equation:Varadhan_Upper_4} holds since $\psi(t, x)$ is increasing in $|x|$ by Proposition \ref{Proposition:Continuity_of_Exp} and \eqref{Equation:Varadhan_Upper_5} is valid by the virtue of Theorem \ref{Theorem:LDP_Tilted}, in which we may apply the upper bound of Varadhan's lemma \eqref{Equation:Varadhan_Original_Upper} on the pair $(m_2(L_{a_j + \epsilon, T, N}), \nu_{a_j + \epsilon, T, N}, T)$ with respect to the closed set $\mathfrak{C}_{a_j, M}$ for $1 \leq j \leq m$. Note that \eqref{Equation:Varadhan_Original_Upper} is applicable since the mapping $(r, \mu, t) \mapsto \log Z_{a_j + \epsilon, t} + \Phi(t, \mu)$ is continuous (continuity of $t \mapsto Z_{a_j + \epsilon, t}$ follows from dominating convergence theorem) and it satisfies the condition \eqref{Equation:Varadhan_Original_Condition} by
    \begin{align*}
        \log Z_{a_j + \epsilon, t} + \Phi(t, \mu) &\leq 0 + \oldconstant{Constant:Varadhan} + \phi_2(t, \mu) \\
        &\leq \oldconstant{Constant:Varadhan} + \oldconstant{Constant:Exp}pt^{2 - 2p}\oldconstant{Constant:Bound}^2 \int_{\mathbb{R}}(t^{q_1 - 1}|x|^{q_1 - 1} + t^{q_2 - 1}|x|^{q_2 - 1})^2 \mu(dx) \\
        &\leq \oldconstant{Constant:Varadhan} + 2\oldconstant{Constant:Bound}^2\oldconstant{Constant:Exp}p\big(t^{2(q_1 - p)}m_{2q_1 - 2}(\mu) + t^{2(q_2 - p)}m_{2q_2 - 2}(\mu)\big) \\
        &\leq \oldconstant{Constant:Varadhan} + 2\oldconstant{Constant:Bound}^2\oldconstant{Constant:Exp}p\big(d^{q_1 - p}M^{(q_1 - 1)/(q_1 - 2)} + d^{q_2 - p}M\big)
    \end{align*}
    for all configurations $m_2(\mu) = 1$ (we may assume this since $m_2(\nu_{a_j + \epsilon, t, N}) = 1$ for all $t$ and $N$), $m_{2q_2 - 2}(\mu) \leq M$ and $t \in [c, d]$, where the first inequality comes from our assumption \eqref{Equation:Varadhan_Tilted_Condition_2}.
    
    To obtain the desired upper bound, it remains to analyze the last inequality as $\delta,\epsilon\downarrow 0$ and $K,M\uparrow\infty.$ First, observe that taking supremum over $a$ and $r$ gives
    \begin{align*}
        &\sup_{\delta \leq a^2 \leq K, a^2 \leq r \leq (a + \epsilon)^2}\Big[- \frac{1}{2}(r - 1) + \frac{1}{2}\log r - \int_{\mathbb{R}} \psi\Big(t, \frac{\sqrt{r}x}{a + \epsilon}\Big)\mu(dx)\Big] \\
        &\leq \sup_{\delta \leq r \leq 2K} \Big[-\frac{1}{2}(r - 1) + \frac{1}{2}\log r - \inf_{a \leq \sqrt{r} \wedge K} \int_{\mathbb{R}} \psi\Big(t, \frac{\sqrt{r}x}{a + \epsilon}\Big)\mu(dx)\Big] \\
        &= \sup_{\delta \leq r \leq 2K} \Big[-\frac{1}{2}(r - 1) + \frac{1}{2}\log r - \int_{\mathbb{R}} \psi\Big(t, \frac{\sqrt{r}x}{\sqrt{r} \wedge K + \epsilon}\Big)\mu(dx)\Big] \\
        &= \sup_{\delta \leq r \leq 1} \Big[-\frac{1}{2}(r - 1) + \frac{1}{2}\log r - \int_{\mathbb{R}} \psi\Big(t, \frac{\sqrt{r}x}{\sqrt{r} \wedge K + \epsilon}\Big)\mu(dx)\Big],
    \end{align*}
    where the last equality holds since the mapping 
    \[r \mapsto - \frac{1}{2}(r - 1) + \frac{1}{2}\log r - \int_{\mathbb{R}} \psi\Big(t, \frac{\sqrt{r}x}{\sqrt{r} + \epsilon}\Big)\mu(dx)\]
    is decreasing on $[1, \infty)$. To deal with the limit $\epsilon \downarrow 0$, note that the parameter set 
    \[\mathfrak{A}_{\delta, M} = \{(r, \mu, t) \in [0, \infty) \times \mathcal{P}_s(\mathbb{R}) \times [c, d] \mid (t, \mu) \in \mathfrak{C}, m_{2q_2 - 2}(\mu) \leq M, \delta \leq r \leq 1\}\]
    is again compact in $([c, d] \times [0, \infty) \times \mathcal{P}_s(\mathbb{R}), \mathscr{T} \otimes \mathscr{W}_s \otimes \mathscr{T})$ by \cite[Lemma 3.14]{kim2018conditional} and 
    \[F(\epsilon, (r, \mu, t)) = \Phi(t, \mu) - \text{KL}(\mu \,\| \,\mu_{\text{Norm}}) - \frac{1}{2}(r - 1) + \frac{1}{2}\log r - \int_{\mathbb{R}} \psi\Big(t, \frac{\sqrt{r}x}{\sqrt{r} + \epsilon}\Big)\mu(dx)\]
    is upper semi-continuous on $[0, \infty) \times ([\delta, \infty) \times \mathcal{P}_s(\mathbb{R}) \times [c, d])$ by Fatou's lemma, we then have
    \begin{align*}
        \limsup_{\epsilon \to 0^+} \sup_{(r, \mu, t) \in \mathfrak{A}_{\delta, M}} F(\epsilon, (t, r, \mu)) &\leq \sup_{(r, \mu, t) \in \mathfrak{A}_{\delta, M}} F(0, (t, r, \mu)) \notag \\
        &\leq \sup_{m_2(\mu) = 1, (t, \mu) \in \mathfrak{C}} \Big(\Phi(t, \mu) - \int_{\mathbb{R}} \psi(t, x) \mu(dx) - \text{KL}(\mu \,\|\, \mu_{\text{Norm}}) \Big).
    \end{align*}
    Since $h_1(M) \to -\infty$ as $M \uparrow \infty$ due to \eqref{Equation:Varadhan_Tilted_Condition_1} and $h_2(\delta, K) \to -\infty$ as $\delta \downarrow 0$ and $K \uparrow\infty$ by standard Gaussian estimates, the announced upper bound follows.
\end{proof}

\subsection{Proof of Theorem \ref{Theorem:LDP_Tilted}}\label{subsec:LDP}

\noindent This subsection is devoted to establishing the proof of Theorem \ref{Theorem:LDP_Tilted} that is based on the following three lemmas.

\begin{lemma}
    If $V$ satisfies \ref{Condition:V}, then the moment generating function satisfies
    \begin{equation}\label{Equation:MGF}
        \mathbb{E}\Big[\int_{\mathbb{R}}e^{\lambda |x|^s} \mu_{a, T}(dx)\Big] < \infty
    \end{equation}
    for all $\lambda > 0$ and $0 < s < 2q_2 - 2$.
\end{lemma}

\begin{proof}
   By condition \ref{Condition:V}, we have that for all $t \in [c, d]$,
    \[\psi(t, x/a) = \oldconstant{Constant:Exp}pt^{2 - 2p}V'(tx/a)^2 \geq \oldconstant{Constant:Bound}^{-2}\oldconstant{Constant:Exp}pt^{2(q_2 - p)}a^{2q_2 - 2}|x|^{2q_2 - 2} \geq \oldconstant{Constant:Bound}^{-2}\oldconstant{Constant:Exp}pc^{2(q_2 - p)}a^{2q_2 - 2}|x|^{2q_2 - 2},\]
which readily yields our assertion.
\end{proof}

\begin{lemma}\label{Lemma:Exp_Tight}
    If $V$ satisfies \ref{Condition:V}, then the random sequence $(m_2(L_{a, T, N}), \nu_{a, T, N}, T)$ is exponentially tight in the space $([0, \infty) \times \mathcal{P}_s(\mathbb{R}) \times [c, d], \mathscr{T} \otimes \mathscr{W}_s \otimes \mathscr{T})$ for all $0 < s < 2q_2 - 2$.
\end{lemma}

\begin{proof}
    It suffice to show that each of them are exponentially tight in their corresponding spaces. Fix $M, K > 0$. Observe that by the Markov inequality,
    \[\mathbb{P}(m_2(L_{a, T, N}) \geq M) = \mathbb{P}\Big(\frac{1}{N}\sum_{i = 1}^N (x_i^{a, t})^2 \geq M\Big) \leq \Big(e^{-M} \mathbb{E}\Big[\int_{\mathbb{R}} e^{x^2}\mu_{a, T}(dx)\Big]\Big)^N.\]
    By \eqref{Equation:MGF}, we choose $M > 0$ so that 
    \[\mathbb{P}(m_2(L_{a, T, N}) \not \in [0, M]) \leq e^{-NK}.\] 
    Hence, $m_2(L_{a, T, N})$ is exponentially tight in $([0, \infty), \mathscr{T})$. Next, consider the set
    \[\mathfrak{K}_{\alpha, M} = \{\mu \in \mathcal{P}(\mathbb{R}) \mid m_\alpha(\mu) \leq M\},\]
    which is compact in $(\mathcal{P}_s(\mathbb{R}), \mathscr{W}_s)$ as long as $s < \alpha < 2q_2 - 2$ by \cite[Lemma 3.14]{kim2018conditional}. Observe that
    \begin{align}
        \mathbb{P}(\nu_{a, T, N} \not \in \mathfrak{K}_{\alpha, M}) &= \mathbb{P}(m_\alpha(\nu_{a, T, N}) > M) \notag \\
        &= \mathbb{P}\Big(\frac{m_\alpha(L_{a, T, N})}{m_2(L_{a, T, N})^{\alpha/2}} > M \Big) \notag \\
        &\leq \mathbb{P}(m_2(L_{a, T, N}) < \delta) + \mathbb{P}(m_\alpha(L_{a, T, N}) > \delta^{\alpha/2}M). \label{Equation:Exp_Tight_1}
    \end{align}
    For the first term in \eqref{Equation:Exp_Tight_1}, by the Markov inequality, for all $\lambda > 0$,
    \begin{align}
        \mathbb{P}(m_2(L_{a, T, N}) < \delta) &\leq \Big(e^{\lambda \delta} \mathbb{E}\Big[\int_{\mathbb{R}} e^{-\lambda x^2}\mu_{a, T}(dx)\Big]\Big)^N \notag \\
        &= \Big(e^{\lambda \delta} \mathbb{E}\Big[\frac{1}{\sqrt{2\pi}Z_{a, T}}\int_{\mathbb{R}} e^{-\lambda x^2 - x^2/2 - \psi(T, x/a)} dx\Big]\Big)^N \leq \Big(\frac{e^{\lambda \delta}}{\sqrt{2\lambda}Z_{a, d}}\Big)^N, \label{Equation:Exp_Tight_2}
    \end{align}
    where the last inequality used Proposition \ref{Proposition:Continuity_of_Exp}(iii).
    For the second term in \eqref{Equation:Exp_Tight_1}, we have again by the Markov inequality,
    \begin{equation}\label{Equation:Exp_Tight_3}
        \mathbb{P}(m_\alpha(L_{a, t, N}) > \delta^{\alpha/2}M) \leq \Big(e^{-\delta^{\alpha/2}M} \mathbb{E}\Big[\int_{\mathbb{R}} e^{|x|^\alpha}\mu_{a, T}(dx)\Big]\Big)^N.
    \end{equation}
    Now, we may choose $\delta = 1/\lambda$ and $\lambda > 0$ large enough so that the term in \eqref{Equation:Exp_Tight_2} is less than $e^{-NK}$. Then, we may choose $M > 0$ large enough so that the term in \eqref{Equation:Exp_Tight_3} is also less than $e^{-NK}$, which is valid by \eqref{Equation:MGF}. This shows that $\nu_{a, T, N}$ is exponentially tight in $(\mathcal{P}_s(\mathbb{R}), \mathscr{W}_s)$. Finally, since $T$ is supported on a compact interval $[c, d]$, it is obviously exponentially tight and our proof is completed.
\end{proof}

\begin{lemma}\label{Lemma:LDP_Tilted_Pre}
    If $V$ satisfies \ref{Condition:V}, then for all $a, t > 0$, the sequence $(m_2(L_{a, t, N}), \nu_{a, t, N})$ satisfies a large deviation principle in the space $([0, \infty) \times \mathcal{P}_s(\mathbb{R}), \mathscr{T} \otimes \mathscr{W}_s)$ for all $0 < s < 2q_2 - 2$ with speed $N$ and a good rate function given by
    \[J_{a, t}(r, \mu) = \begin{cases}
        \log Z_{a, t} + \text{KL}(\mu \,\| \,\mu_{\text{Norm}}) + \frac{1}{2}(r - 1) - \frac{1}{2}\log r + \phi_2\big(t, S_{\sqrt{r}/a}(\mu)\big), & m_2(\mu) = 1, \\
        \infty, & m_2(\mu) \neq 1.
    \end{cases}\]
\end{lemma}

\begin{proof}
    First, by Cram\'{e}r's theorem (see \cite[Theorem 6.1.3]{dembo2010large}), the empirical mean $(m_2(L_{a, t, N}), L_{a, t, N})$ of the i.i.d.\ sequence $((x_i^{a, t})^2, \delta_{x_i^{a, t}})$ in the locally convex, Hausdorff vector space $(\mathbb{R} \times \mathcal{M}(\mathbb{R}), \mathscr{T} \otimes \mathscr{T}_{\text{weak}})$ satisfies an weak LDP with speed $N$ and a good rate function given by
    \begin{align*}
        \Lambda_{a, t}(r, \mu) &= \sup_{(\lambda, f) \in \mathbb{R} \times C_b(\mathbb{R})} \Big\{ \lan{(\lambda, f), (r, \mu)} - \log\mathbb{E}\Big[\exp\Big(\Big\langle (\lambda, f), ((x_1^{a, t})^2, \delta_{x_{1}^{a, t}})\Big\rangle\Big)\Big]\Big\} \\
        &= \sup_{(\lambda, f) \in \mathbb{R} \times C_b(\mathbb{R})} \Big[\Big(\lambda r + \int_{\mathbb{R}} f(x)\mu(dx) - \log\int_{\mathbb{R}} e^{f(x) + \lambda x^2} \mu_{a, t}(dx)\Big].
    \end{align*}
    Since the sequence $(m_2(L_{a, t, N}), L_{a, t, N})$ is exponentially tight in $(\mathbb{R} \times \mathcal{P}(\mathbb{R}), \mathscr{T} \otimes \mathscr{T}_{\text{weak}})$ (this fact can be shown using the same strategy as in Lemma \ref{Lemma:Exp_Tight}), we see then by \cite[Lemmas 1.2.18 and 4.1.5]{dembo2010large} that the empirical mean $(m_2(L_{a, t, N}), L_{a, t, N})$ satisfies an LDP on $([0, \infty) \times \mathcal{P}(\mathbb{R}), \mathscr{T} \otimes \mathscr{T}_{\text{weak}})$ with speed $N$ and good rate function $\Lambda_{a, t}(r, \mu)$.
    
    Next, we simplify the rate function as follows. Consider again the tilted measure
    \[\mu_{a, t, \lambda} = \frac{e^{\lambda x^2}}{Z_{a, t, \lambda}}\mu_{a, t}(dx), \quad Z_{a, t, \lambda} = \int_{\mathbb{R}} e^{\lambda x^2} \mu_{a, t}(dx),\]
    where $Z_{a, t, \lambda} < \infty$ for all $\lambda > 0$ by \eqref{Equation:MGF}. By the Donsker-Varadhan variational formula (see \cite[Lemma 6.2.13]{dembo2010large}), we may rewrite the rate function as
    \begin{align*}
        \Lambda_{a, t}(r, \mu) &= \sup_{(\lambda, f) \in \mathbb{R} \times C_b(\mathbb{R})} \Big[\Big(\lambda r + \int_{\mathbb{R}} f(x)\mu(dx) - \log Z_{a, t, \lambda} - \log\int_{\mathbb{R}} e^{f(x)}\mu_{a, t, \lambda}(dx)\Big] \\
        &= \sup_{\lambda \in \mathbb{R}} \,\big(\lambda r - \log Z_{a, t, \lambda} + \text{KL}(\mu \,\|\, \mu_{a, t, \lambda})\big).
    \end{align*}
    Note that 
    \begin{align}
        \text{KL}(\mu \,\|\, \mu_{a, t, \lambda}) &= \text{KL}(\mu \,\| \,\mu_{\text{Norm}}) - \int_{\mathbb{R}} \log\Big(\frac{d\mu_{a, t, \lambda}}{d\mu_{\text{Norm}}}(x)\Big) \mu(dx) \notag \\
        &= \text{KL}(\mu \,\| \,\mu_{\text{Norm}}) - \int_{\mathbb{R}} \log\Big(\frac{e^{\lambda x^2}}{Z_{a, t, \lambda}} \cdot \frac{e^{-\psi(t, x/a)}}{Z_{a, t}}\Big) \mu(dx) \label{Equation:LDP_Tilted_Pre_1} \\
        &= \text{KL}(\mu \,\| \,\mu_{\text{Norm}}) + \log Z_{a, t, \lambda} + \log Z_{a, t} - \lambda m_2(\mu) + \phi_2(t, S_{1/a}(\mu)). \label{Equation:LDP_Tilted_Pre_2}
    \end{align}
    Observe that if $m_{2q_2 - 2}(\mu) = \infty$, the quantity in \eqref{Equation:LDP_Tilted_Pre_1} does blow up according to the definition of $\psi$ and condition \ref{Condition:V}. This matches up with the definition of $\phi_2(\mu) = \infty$ in \eqref{Equation:LDP_Tilted_Pre_2}. Hence, the rate function simplifies to
    \begin{align*}
        \Lambda_{a, t}(r, \mu) &= \log Z_{a, t} + \text{KL}(\mu \,\| \,\mu_{\text{Norm}}) + \phi_2(t, S_{1/a}(\mu)) + \sup_{\lambda \in \mathbb{R}} \lambda(r - m_2(\mu)) \\
        &= \begin{cases}
            \log Z_{a, t} + \text{KL}(\mu \,\|\, \mu_{\text{Norm}}) + \phi_2(t, S_{1/a}(\mu)), & m_2(\mu) = r, \\
            \infty, & m_2(\mu) \neq r.
        \end{cases}
    \end{align*}

    Now, we consider the LDP for $(m_2(L_{a, t, N}), \nu_{a, t, N})$. Note that the mapping $G \colon [0, \infty) \times \mathcal{P}(\mathbb{R}) \to [0, \infty) \times \mathcal{P}(\mathbb{R})$ given by 
    \[G(r, \mu) = (r, S_{1/\sqrt{r}}(\mu)).\]
    is continuous on  $(\mathbb{R} \times \mathcal{P}(\mathbb{R}), \mathscr{T} \otimes \mathscr{T}_{\text{weak}})$ by Slutsky's theorem and we have 
    \[G(m_2(L_{a, t, N}), L_{a, t, N}) = (m_2(L_{a, t, N}), \nu_{a, t, N}).\]
    By contraction principle (see \cite[Theorem 4.2.1]{dembo2010large}), we see then $(m_2(L_{a, t, N}), \nu_{a, t, N})$ satisfies an LDP on $(\mathbb{R} \times \mathcal{P}(\mathbb{R}), \mathscr{T} \otimes \mathscr{T}_{\text{weak}})$ with speed $N$ and good rate function
    \begin{align*}
        J_{a, t}(r, \mu) &= \inf\{\Lambda_{a, t}(s, \mu) \mid (s, \mu) \in \mathbb{R} \times \mathcal{P}(\mathbb{R}), G(s, \nu) = (r, \mu)\} \\
        &= \Lambda_{a, t}(r, S_{\sqrt{r}}(\mu)) \\
        &= \begin{cases}
            \log Z_{a, t} + \text{KL}(S_{\sqrt{r}}(\mu) \,\|\, \mu_{\text{Norm}}) + \phi_2\big(t, S_{\sqrt{r}/a}(\mu)\big), & m_2(S_{\sqrt{r}}(\mu)) = r, \\
            \infty, & m_2(S_{\sqrt{r}}(\mu)) \neq r
        \end{cases} \\
        &= \begin{cases}
            \log Z_{a, t} + \text{KL}(\mu \,\| \,\mu_{\text{Norm}}) + \frac{1}{2}(r - 1) - \frac{1}{2}\log r + \phi_2\big(t, S_{\sqrt{r}/a}(\mu)\big) & m_2(\mu) = 1, \\
            \infty & m_2(\mu) \neq 1.
        \end{cases}
    \end{align*}

    Lastly, since the sequence $(m_2(L_{a, t, N}), \nu_{a, t, N})$ is exponentially tight in $([0, \infty) \times \mathcal{P}_s(\mathbb{R}), \mathscr{T} \otimes \mathscr{W}_s)$ by Lemma \ref{Lemma:Exp_Tight}, we can upgrade the topology for the LDP by \cite[Corollary 4.2.6]{dembo2010large} and the proof is concluded.
\end{proof}

Now we are ready to present the proof of Theorem \ref{Theorem:LDP_Tilted}.

\begin{proof}[\bf Proof of Theorem \ref{Theorem:LDP_Tilted}]
    First, we provide a simple fact about the measure $\mu_{a, t}$. Observe that for all $t \in [c', d'] \subseteq [c, d]$, the density of $\mu_{a, t}$ satisfies
    \begin{align*}
        f_{\mu_{a, t}}(x) = \dfrac{e^{-\psi(t, x/a)}}{Z_{a, t}}f_{\mu_{\text{Norm}}}(x) \geq \dfrac{e^{-\psi(d', x/a)}}{Z_{a, c'}}f_{\mu_{\text{Norm}}}(x) = \frac{Z_{a, d'}}{Z_{a, c'}}f_{\mu_{a, d'}}(x),
    \end{align*}
    where the inequality holds since $t \mapsto \psi(t, x)$ is increasing by Proposition \ref{Proposition:Continuity_of_Exp}, which also implies that $Z_{a, t} \le Z_{a, c'}.$ We can argue an upper bound similarly. Since the mapping $t \mapsto Z_{a, t}$ is continuous by dominating convergence theorem, the set $[c, d]$ is compact and $Z_{a, t} \geq Z_{a, d} > 0$ for all $t \in [c, d]$, the mapping $t \mapsto \log Z_{a, t}$ is uniformly continuous on $[c, d]$. Therefore, for all $\epsilon > 0$, we may choose $\delta > 0$ so that
    \begin{equation}\label{Equation:Tilted_LDP_0}
        e^{-\epsilon}f_{\mu_{a, d'}}(x) \leq f_{\mu_{a, t}}(x) \leq e^{\epsilon}f_{\mu_{a, c'}}(x)
    \end{equation}
    for all $d' - c' < \delta$, $t \in [c', d']$ and $x \in \mathbb{R}$.
    
    Now, we work with the lower bound. Consider first the product case $\mathfrak{U} = \mathfrak{U}^1 \times \mathfrak{U}^2$, where the sets $\mathfrak{U}^1 \subseteq ([0, \infty) \times \mathcal{P}_s(\mathbb{R}), \mathscr{T} \otimes \mathscr{W}_s)$ and $\mathfrak{U}^2 \subseteq ([c, d], \mathscr{T})$ are open. Fix $\epsilon > 0$. For all $t \in \mathfrak{U}^2$, there exists some $\delta > 0$ such that $(t - \delta, t) \subseteq \mathfrak{U}^2$ and \eqref{Equation:Tilted_LDP_0} holds. It follows that
    \begin{align}
        &\liminf_{N \to \infty}\frac{1}{N}\log \mathbb{P}((m_2(L_{a, T, N}), \nu_{a, T, N}, T) \in \mathfrak{U}^1 \times \mathfrak{U}^2) \notag \\
        &\geq -\epsilon + \liminf_{N \to \infty} \frac{1}{N}\log \mathbb{E}\Big[\mathbbm{1}_{\{T \in (t - \delta, t)\}}\int_{\mathbb{R}^N} \mathbbm{1}_{\{(m_2(L_{x, N}), \nu_{x, N}) \in \mathfrak{U}^1\}} \mu_{a, t}^{\otimes N}(dx) \Big] \label{Equation:Tilted_LDP_1} \\
        &= -\epsilon + \liminf_{N \to \infty}\frac{1}{N}\log\mathbb{P}((m_2(L_{a, t, N}), \nu_{a, t, N}) \in \mathfrak{U}^1) \label{Equation:Tilted_LDP_2} \\
        &= -\epsilon - \inf_{m_2(\mu) = 1, (r, \mu) \in \mathfrak{U}^1} J_{a, t}(r, \mu), \label{Equation:Tilted_LDP_3}
    \end{align}
    where \eqref{Equation:Tilted_LDP_1} used  \eqref{Equation:Tilted_LDP_0}, \eqref{Equation:Tilted_LDP_2} holds since $N^{-1} \log \mathbb{P}(T \in (t - \delta, t))$ vanishes as $N \to \infty$, and lastly \eqref{Equation:Tilted_LDP_3} follows by Lemma \ref{Lemma:LDP_Tilted_Pre}. By taking supremum over $t \in \mathfrak{U}^2$ on the right hand side of \eqref{Equation:Tilted_LDP_2} and utilizing the fact that $\epsilon > 0$ is arbitrary, we conclude
    \[\liminf_{N \to \infty}\frac{1}{N}\log \mathbb{P}((m_2(L_{a, T, N}), \nu_{a, T, N}, T) \in \mathfrak{U}^1 \times \mathfrak{U}^2) \geq -\inf_{m_2(\mu) = 1, (r, \mu, t) \in \mathfrak{U}^1 \times \mathfrak{U}^2} J_{a, t}(r, \mu)\]
    where by definition, $J_{a, t}(r, \mu) = J_a(r, \mu, t)$. For the general open set $\mathfrak{U}$, we see that for all $(r, \mu, t) \in \mathfrak{U}$, there exist open sets $\mathfrak{U}^1 \subseteq ([0, \infty) \times \mathcal{P}_s(\mathbb{R}), \mathscr{T} \otimes \mathscr{W}_s)$ and $\mathfrak{U}^2 \subseteq ([c, d], \mathscr{T})$ such that $(r, \mu, t) \in \mathfrak{U}^1 \times \mathfrak{U}^2 \subseteq \mathfrak{U}$, then,
    \begin{align*}
        &\liminf_{N \to \infty}\frac{1}{N}\log \mathbb{P}((m_2(L_{a, T, N}), \nu_{a, T, N}, T) \in \mathfrak{U}) \\
        &\geq \liminf_{N \to \infty}\frac{1}{N}\log \mathbb{P}((m_2(L_{a, T, N}), \nu_{a, T, N}, T) \in \mathfrak{U}^1 \times \mathfrak{U}^2) = -\inf_{m_2(\mu) = 1, (r, \mu, t) \in \mathfrak{U}^1 \times \mathfrak{U}^2} J_a(r, \mu, t)
    \end{align*}
    and we may obtain the lower bound by taking the supremum over $(r, \mu, t) \in \mathfrak{U}$.

    Next, we handle the upper bound. Consider first the product case $\mathfrak{C} = \mathfrak{C}^1 \times \mathfrak{C}^2$, where the sets $\mathfrak{C}^1 \subseteq ([0, \infty) \times \mathcal{P}_s(\mathbb{R}), \mathscr{T} \otimes \mathscr{W}_s)$ and $\mathfrak{C}^2 \subseteq ([c, d], \mathscr{T})$ are closed. Fix $\epsilon > 0$ and choose $\delta > 0$ according to \eqref{Equation:Tilted_LDP_0}. Since $\mathfrak{C}^2$ is compact, there exist $(t_j)_{1\leq j\leq m}\subset \mathfrak{C}^2$ such that $\mathfrak{C}^2 \subseteq \bigcup_{j = 1}^m[t_j, t_j + \delta]$. Observe that the probability
    \begin{align}
        &\limsup_{N \to \infty}\frac{1}{N}\log \mathbb{P}((m_2(L_{a, T, N}), \nu_{a, T, N}, T) \in \mathfrak{C}^1 \times \mathfrak{C}^2) \notag \\
        &\leq \epsilon + \limsup_{N \to \infty} \frac{1}{N}\log\Big\{\sum_{j = 1}^m\mathbb{E} \Big[\mathbbm{1}_{\{T \in [t_j, t_j + \delta]\}}\int_{\mathbb{R}^N} \mathbbm{1}_{\{(m_2(L_{x, N}), \nu_{x, N}) \in \mathfrak{C}^1\}} \mu_{a, t_j}^{\otimes N}(dx)\Big] \label{Equation:Tilted_LDP_4} \\
        &= \epsilon + \limsup_{N \to \infty} \frac{1}{N}\log\Big\{\sum_{j = 1}^m \mathbb{P}((m_2(L_{a, t_j, N}), \nu_{a, t_j, N}) \in \mathfrak{C}^1)\Big\} \label{Equation:Tilted_LDP_5} \\
        &= \epsilon + \max_{1 \leq j \leq m} \Big[-\inf_{(r, \mu) \in \mathfrak{C}^1, m_2(\mu) = 1} J_{a, t_j}(r, \mu)\Big] \label{Equation:Tilted_LDP_6} \\
        &\leq \epsilon - \inf_{(r, \mu, t) \in \mathfrak{C}^1 \times \mathfrak{C}^2, m_2(\mu) = 1} J_{a, t}(r, \mu),
    \end{align}
    where \eqref{Equation:Tilted_LDP_4} follows from \eqref{Equation:Tilted_LDP_0}, \eqref{Equation:Tilted_LDP_5} holds since $N^{-1} \log\mathbb{P}(T \in [t_j, t_j + \delta])$ vanishes as $N \to \infty$, and \eqref{Equation:Tilted_LDP_6} is valid by Lemma \ref{Lemma:LDP_Tilted_Pre}. As for the case of arbitrary compact set $\mathfrak{C}$, there exists for $1 \leq j \leq m$, $(r_j, \mu_j, t_j) \in \mathfrak{C}$ and closed sets $\mathfrak{C}_j^1 \subseteq ([0, \infty) \times \mathcal{P}_s(\mathbb{R}), \mathscr{T} \otimes \mathscr{W}_s)$, $\mathfrak{C}_j^2 \subseteq ([c, d], \mathscr{T})$ such that $\mathfrak{C} \subseteq \bigcup_{j = 1}^m \mathfrak{C}_j^1 \times \mathfrak{C}_j^2$ and hence
    \begin{align*}
        &\limsup_{N \to \infty}\frac{1}{N}\log \mathbb{P}((m_2(L_{a, T, N}), \nu_{a, T, N}, T) \in \mathfrak{C}) \\
        &\leq \limsup_{N \to \infty} \frac{1}{N}\log\Big\{\sum_{i = 1}^m \mathbb{P}((m_2(L_{a, T, N}), \nu_{a, T, N}, T) \in \mathfrak{C}_j^1 \times \mathfrak{C}_j^2\Big\} \\
        &\leq \max_{1 \leq j \leq m} \Big[-\inf_{(r, \mu, t) \in \mathfrak{C}_j^1 \times \mathfrak{C}_j^2, m_2(\mu) = 1} J_{a, t}(r, \mu)\Big] \\
        &= -\inf_{(r, \mu, t) \in \mathfrak{C}, m_2(\mu) = 1} J_{a, t}(r, \mu)
    \end{align*}
    and we have again the upper bound. Last, we consider the general case where $\mathfrak{C}$ is closed. Since the law of the random sequence $(m_2(L_{a, T, N}), \nu_{a, T, N}, T)$ is exponentially tight by Lemma \ref{Lemma:Exp_Tight}, we may choose for all $M > 0$ a compact set $\mathfrak{K}_M$ such that
    \[\frac{1}{N}\log \mathbb{P}((m_2(L_{a, T, N}), \nu_{a, T, N}, T) \not \in \mathfrak{K}_M) \leq -M.\]
    We see then
    \begin{align*}
        &\limsup_{N \to \infty}\frac{1}{N}\log \mathbb{P}((m_2(L_{a, T, N}), \nu_{a, T, N}, T) \in \mathfrak{C}) \\
        &\leq \Big[-\inf_{(r, \mu, t) \in \mathfrak{C} \cap \mathfrak{K}_M, m_2(\mu) = 1} J_{a, t}(r, \mu)\Big] \vee \Big[\limsup_{N \to \infty}\frac{1}{N}\log \mathbb{P}((m_2(L_{a, T, N}), \nu_{a, T, N}, T) \in \mathfrak{C} \setminus \mathfrak{K}_M)\Big]
    \end{align*}
    and we may obtain the desired upper bound by taking $M \to \infty$.
\end{proof}

\subsection{Proof of Theorem \ref{Theorem:Main_Theorem}}

First, we have the following lower bound.

\begin{proposition}\label{Proposition:Main_Theorem_Lower}
    If $V$ satisfies Assumption \ref{ass1}, then for all $u \geq 0$,
    \[\liminf_{N \to \infty}\dfrac{1}{N}\log \mathbb{E}[\text{Crt}_N((Nu, \infty))] \geq \sup\{\,\mathcal{I}(\mu) \mid \mu \in \mathcal{P}_{2q_2 - 2}(\mathbb{R}) \cap \mathfrak{D}_+(u)\},\]
    where
    \begin{equation}\label{Equation:D_+(u)}
        \mathfrak{D}_+(u) : 
        = \Big\{\mu \in \mathcal{P}(\mathbb{R}) \,\Big\vert\, \mathbb{E}_\mu\big[p^{-1}XV'(X) - V(X)\big] > u\Big\}.
    \end{equation}
\end{proposition}

\begin{proof}
    Fix $M, \epsilon, \delta > 0$. The following inequalities hold as long as $K$ is large enough,
    \begin{align}
        &\liminf_{N \to \infty} \dfrac{1}{N}\log \mathbb{E}[\text{Crt}_N((Nu, \infty))] \notag \\
        &= \frac{1}{2}\log\Bpar{\frac{p - 1}{2\pi}} + \liminf_{N \to \infty} \frac{1}{N}\log I_{\mathbb{E}|\det M_{N - 1}(\cdot)|}(\mathbb{R}^N) \label{Equation:Main_Lower_Bound_1} \\
       &\geq \frac{1}{2}\log\Bpar{\frac{p - 1}{2\pi}} + \liminf_{N \to \infty} \frac{1}{N}\log I_{\mathbb{E}|\det M_{N - 1}(\cdot)|}(\{\delta \leq \opnorm{\sigma}_2 \leq M\}) \label{Equation:Main_Lower_Bound_2} 
        \\
        &\geq -\epsilon + \frac{1}{2}\log\Bpar{\frac{p - 1}{2\pi}} + \liminf_{N \to \infty} \frac{1}{N}\log I_{\mathbb{E}|\det Q_N^K(\cdot)|}(\{\delta \leq \opnorm{\sigma}_2 \leq M\}) \label{Equation:Main_Lower_Bound_3} \\
        &= -\epsilon + \frac{1}{2}\log\Bpar{\frac{p - 1}{2\pi}} + \liminf_{N \to \infty}\dfrac{1}{N}\log I_{\exp s_N^K(\cdot)}(\{\delta \leq \opnorm{\sigma}_2 \leq M\}) \label{Equation:Main_Lower_Bound_4} \\
        \nonumber&= -\epsilon + \frac{1}{2}\log(p - 1) + \frac{1}{2} \\
        &+ \liminf_{N \to \infty} \frac{1}{N}\log\Big\{\int_\delta^M  \mathbb{E}\Big[t^{-p - 1}\frac{|\langle\sqrt{N}t\omega, v(\sqrt{N}t\omega) \rangle|}{N}e^{N\varphi_K(t, L_{\sqrt{N}\omega, N})}\mathbbm{1}_{\{(t, L_{\sqrt{N}\omega, N}) \in \mathfrak{F}(u)\}}\Big]dt\Big\} \label{Equation:Main_Lower_Bound_5},
    \end{align}
    where \eqref{Equation:Main_Lower_Bound_1} holds by \eqref{Kac-Rice-Final-Formula},  \eqref{Equation:Main_Lower_Bound_3} follows from Proposition \ref{Proposition:Main_Truncation}, \eqref{Equation:Main_Lower_Bound_4} uses Proposition \ref{Proposition:Application_of_Ben}, and finally \eqref{Equation:Main_Lower_Bound_5} is valid by change of variables $\sigma = \sqrt{N}t\omega$ where $\omega \in \mathbb{S}^{N - 1}$. Note that by condition \ref{Condition:V} and \eqref{Equation:Tech_2}, we have
    \[t^{-p - 1}\frac{|\langle\sqrt{N}t\omega, v(\sqrt{N}t\omega) \rangle|}{N} \geq \oldconstant{Constant:Tech_2}t^{q_1 - p - 1} \geq \oldconstant{Constant:Tech_2}\delta^{q_1 - p - 1} \geq \oldconstant{Constant:Tech_2}\min\{M^{q_1 - p - 1}, \delta^{q_1 - p - 1}\}\]
    and this term is therefore negligible in \eqref{Equation:Main_Lower_Bound_5}. Hence, after noting that $L_{\sqrt{N}\omega,N}$ is identically distributed as $\nu_{g,N}$, we have arrived at
    \begin{align}
        &\liminf_{N \to \infty} \dfrac{1}{N}\log \mathbb{E}[\text{Crt}_N((Nu, \infty))] \notag \\
        &\geq -\epsilon + \frac{1}{2}\log(p - 1) + \frac{1}{2} + \liminf_{N \to \infty} \frac{1}{N}\log \mathbb{E}\Big[e^{N\varphi_K(T, \nu_{g, N})} \mathbbm{1}_{\{(T, \nu_{g, N}) \in \mathfrak{F}(u)\}}\Big], \label{Equation:Main_Lower_Bound_6}
    \end{align}
    where $T$ is uniformly sampled from $[\delta, M]$. Set $q_2 < s < 2q_2 - 2$ and $\Phi \colon [\delta, M] \times \mathcal{P}_s(\mathbb{R}) \to \mathbb{R}$ by $\Phi = \phi_1 + \phi_{3, K}$, which is continuous on $([\delta, M] \times \mathcal{P}_s(\mathbb{R}), \mathscr{T} \otimes \mathscr{W}_s)$ by Proposition \ref{Proposition:Continuity_of_Exp}. Notice that 
    \begin{align}
        \Phi(\mu) - \phi_2(t, \mu) &= \phi_1(t, \mu) - \phi_2(t, \mu) + \phi_{3, K}(t, \mu) \notag \\
        &\leq -\oldconstant{Constant:Exp}t^{2 - 2p}\int_{\mathbb{R}}V'(tx)^2\mu(dx) + \phi_{3, K}(t, \mu) \label{Equation:Main_Upper_Bound_7} \\
        &\leq -\oldconstant{Constant:Bound}^{-2}\oldconstant{Constant:Exp}t^{2(q_2 - p)}m_{2q_2 - 2}(\mu) + (4\oldconstant{Constant:Free_Convolution_Infty} + 2\oldconstant{Constant:A_1}K + 4), \label{Equation:Main_Upper_Bound_8}
    \end{align}
    where \eqref{Equation:Main_Upper_Bound_7} holds by the Cauchy-Schwartz inequality and \eqref{Equation:Main_Upper_Bound_8} is true because of  \ref{Condition:V} and \eqref{Equation:Bounded_phi_3}. Therefore, $\Phi$ satisfies both conditions \eqref{Equation:Varadhan_Tilted_Condition_1} and \eqref{Equation:Varadhan_Tilted_Condition_2}. As a result, we can apply the lower bound in Theorem \ref{Theorem:Varadhan_Tilted} to the right hand side of \eqref{Equation:Main_Lower_Bound_6} with the open set 
    \[\mathfrak{U} = \mathfrak{F}_+(u) = \Big\{(t, \mu) \in [0, \infty) \times \mathcal{P}_s(\mathbb{R}) \,\Big\vert\, \mathbb{E}_\mu\big[p^{-1}tXV'(tX) - V(tX)\big] > u\Big\}\] 
    to conclude that
    \begin{equation}\label{Equation:Main_Lower_Bound_7}
        \liminf_{N \to \infty} \dfrac{1}{N}\log \mathbb{E}[\text{Crt}_N((Nu, \infty))] \geq -\epsilon + \sup_{\delta \leq t \leq M, m_2(\mu) = 1, (t, \mu) \in \mathfrak{F}_+(u)} \mathcal{I}_K(t, \mu).
    \end{equation}
    At last, by Proposition \ref{Proposition:Approximation_of_phi}, we can pass to the limits in \eqref{Equation:Main_Lower_Bound_7} in the order $K \uparrow \infty$ and then $\delta\downarrow 0,M\uparrow \infty$  by using the equivalence \eqref{Equation:Reduction_t} and noting that adding an additional constraint $m_{2q_2 - 2}(\mu) < \infty$ reduces a further lower bound.
\end{proof}

Next we treat the upper bound. As before, we need to avoid the singularity by performing a truncation near the origin. Unlike the trivial bound \eqref{Equation:Main_Lower_Bound_2} in our upper bound proof, this truncation requires more delicate justifications.
\begin{lemma}[Truncation in $\opnorm{\cdot}_2$]\label{Lemma:Upper_Truncation}\
    \begin{enumerate}[label = (\roman*)]
        \item\label{Item:Upper_Truncation_1} If $V$ satisfies \ref{Condition:V}, then for all $M \ge 1$ and $ u > 0$, there exists some $\delta = \delta(M, u)>0$ such that
        \[I_{\exp s_N^K(\cdot)}(\{\opnorm{\sigma}_2 \leq \delta, \opnorm{\sigma}_{2q_2 - 2} \leq M\}) = 0.\]
        \item\label{Item:Upper_Truncation_2} If $V$ satisfies Assumption \ref{ass1} and $u = 0$, then for all $M \geq 1$, $K <\infty$, and $\epsilon > 0$, there exists some $\delta = \delta(M, K, \epsilon) > 0$ such that
        \begin{align*}
            &\limsup_{N \to \infty}\frac{1}{N}\log I_{\exp s_N^K(\cdot)}(\{\opnorm{\sigma}_{2q_2 - 2} \leq M\}) \\
            &\leq \epsilon + \limsup_{N \to \infty}\frac{1}{N}\log I_{\exp s_N^K(\cdot)}(\{\delta \leq \opnorm{\sigma}_2, \opnorm{\sigma}_{2q_2 - 2} \leq M\}).
        \end{align*}
    \end{enumerate}
\end{lemma}

\begin{proof}
    For \ref{Item:Upper_Truncation_1}, recall the set
    \[\Omega(u) = \Big\{\sigma \in \mathbb{R}^N \, \Big\vert\,\frac{1}{N}\sum_{i = 1}^N p^{-1}\sigma_iV'(\sigma_i) - V(\sigma_i) \geq u\Big\}.\]
    If $\opnorm{\sigma}_{2q_2 - 2} \leq M$, we can choose $\delta > 0$ small enough so that if $\opnorm{\sigma}_2 \leq \delta$, then 
    \begin{align*}
        \frac{1}{N}\sum_{i = 1}^N p^{-1}\sigma_iV'(\sigma_i) - V(\sigma_i) &\leq \frac{1}{N}\sum_{i = 1}^N p^{-1}\oldconstant{Constant:Bound}(|\sigma_i|^{q_1} + |\sigma_i|^{q_2}) = p^{-1}\oldconstant{Constant:Bound}(\opnorm{\sigma}_{q_1}^{q_1} + \opnorm{\sigma}_{q_2}^{q_2}) \\
        &\leq p^{-1}\oldconstant{Constant:Bound}\opnorm{\sigma}_2(\opnorm{\sigma}_{2q_1 - 2}^{q_1 - 1} + \opnorm{\sigma}_{2q_2 - 2}^{q_2 - 1}) \leq \delta \cdot 2p^{-1}\oldconstant{Constant:Bound}M^{q_2 - 1} < u.
    \end{align*}
    Hence, $\sigma\notin\Omega(u)$ and \ref{Item:Upper_Truncation_1} holds. 
    
    For \ref{Item:Upper_Truncation_2}, note that as in Proposition \ref{Proposition:Continuity_of_Exp}, we have 
    \[\left|s_N^K(\sigma) - s_N^K(\underline{0})\right| \leq o_{\eta \to 0}(1) + \frac{2(\oldconstant{Constant:A_1}K + 2)}{\eta} \cdot d_{\text{L}}(\mu_{D_N^K(\sigma)}, \mu_{D_N^K(\underline{0})}),\]
    where the L\'{e}vy distance satisfies
    \begin{align*}
        d_{\text{L}}(\mu_{D_N^K(\sigma)}, \mu_{D_N^K(\underline{0})}) &\leq \inf\Big\{\epsilon > 0 \,\Bigl\vert\, \frac{1}{N}\sum_{i = 1}^N \mathbbm{1}_{\{\oldconstant{Constant:A_1}V''(\sigma_i)/\opnorm{\sigma}_2^{p - 2} > \epsilon\}} < \epsilon\Big\} \\
        &\leq \inf\Big\{\epsilon > 0 \,\Bigl\vert\, \frac{1}{N}\sum_{i = 1}^N \Big(\frac{\oldconstant{Constant:A_1}V''(\sigma_i)}{\epsilon\opnorm{\sigma}_2^{p - 2}}\Big)^{2/(q_2 - 2)} < \epsilon\Big\} \\
        &\leq \inf\Big\{\epsilon > 0 \,\Bigl\vert\, \Big(\frac{2^{q_2/2}\oldconstant{Constant:Bound}\oldconstant{Constant:A_1}}{\epsilon\opnorm{\sigma}_2^{p - 2}}\Big)^{2/(q_2 - 2)}\opnorm{\sigma}_2^2 < \epsilon\Big\} = 2(\oldconstant{Constant:Bound}\oldconstant{Constant:A_1})^{2/q_2}\opnorm{x}_2^{2(q_2 - p)/q_2}.
    \end{align*}
    Hence, we can choose $\eta > 0$ and $\delta = \delta(K) > 0$ such that $|s_N^K(\sigma) - s_N^K(\underline{0})| < \epsilon$ whenever $\opnorm{x}_2 < 2\delta$. Now, we apply \eqref{Equation:Tech} so that
    \begin{align*}
        &I_{\exp{s_N^K(\cdot)}}(\{\delta \leq \opnorm{\sigma}_2 \leq 2\delta, \opnorm{\sigma}_{2q_2 - 2} \leq M\}) \notag \\
        &\geq e^{N(s_N^K(0) - \epsilon)}I_1(\{\delta \leq \opnorm{\sigma}_2 \leq 2\delta, \opnorm{\sigma}_{2q_2 - 2} \leq M\}) \\
        &\geq \oldconstant{Constant:Tech_2}e^{N(s_N^K(0) - \epsilon)}N^{N/2}S_{N - 1}\int_{\delta^{q_1 - p}}^{(2\delta)^{q_1 - p}} e^{-2\oldconstant{Constant:Bound}^2\oldconstant{Constant:Exp}pN\big(t^2C^{2(2q_1 - 2)} + t^{2(q_2 - p)/(q_1 - p)}C^{2(2q_2 - 2)}\big)} \\
        &\hspace{1cm} \times \mathbb{P}\Big(\big\{t^{1/(q_1 - p)}\opnorm{\sqrt{N}|g|/\|g\|_2}_{2q_2 - 2} \leq M\big\} \cap \bigcap_{i = 1}^N \{C^{-1} \leq |g_i| \leq C\}\Big) dt.
    \end{align*}
    Observe that for $(t,g)$ satisfying that $\delta^{q_1-p}\leq t\leq (2\delta)^{q_1-p}$ and $C^{-1}\leq |g_i|\leq C$ for all $1\leq i\leq N,$ we have
    $$t^{1/(q_1 - p)}\opnorm{\sqrt{N}g/\|g\|_2}_{2q_2 - 2} \leq 2\delta C^2,$$
    which implies that the integral above is bounded from below by
    \[[(2\delta)^{q_1 - p} - \delta^{q_1 - p}]e^{-2\oldconstant{Constant:Bound}^2\oldconstant{Constant:Exp}pN\big((2\delta)^{2(q_1 - p)}C^{2(2q_1 - 2)} + (2\delta)^{2(q2 - p)}C^{2(2q_2 - 2)}\big)}\mathbbm{1}_{\{2\delta C^2 \leq M\}}\mathbb{P}(C^{-1} \leq |g_1| \leq C)^N.\]
    By choosing $C$ large and $\delta > 0$ small enough, it follows that
    \[I_{\exp{s_N^K(\cdot)}}(\{\delta \leq \opnorm{\sigma}_2 \leq 2\delta, \opnorm{\sigma}_{2q_2 - 2} \leq M\}) \geq e^{N(s_N^K(0) - \epsilon)}N^{N/2}S_{N - 1}(1 + \epsilon)^{-N/2}.\]
    Similarly, 
    \begin{align}
        &I_{\exp{s_N^K(\cdot)}}(\{\opnorm{\sigma}_2 \leq \delta, \opnorm{\sigma}_{2q_2 - 2} \leq M\}) \notag \\
        \nonumber&\leq e^{N(s_N^K(\underline{0}) + \epsilon)} \int_{\{\opnorm{\sigma}_2 \leq \delta\}} \frac{|\lan{\sigma, v(\sigma)}|}{N\opnorm{\sigma}_2^{N + p}} d\sigma \\
        \\
        &\leq \oldconstant{Constant:Bound}\oldconstant{Constant:A_1}e^{N(s_N^K(\underline{0}) + \epsilon)} \int_{\{\opnorm{\sigma}_2 \leq \delta\}} \frac{\opnorm{\sigma}_{q_1}^{q_1} + \opnorm{\sigma}_{q_2}^{q_2}}{\opnorm{\sigma}_2^{N + p}} d\sigma, \label{Equation:Large_Truncation_2} \\
        &= \oldconstant{Constant:Bound}\oldconstant{Constant:A_1}e^{N(s_N^K(\underline{0}) + \epsilon)} N^{N/2}S_{N - 1}\int_0^\delta \mathbb{E}\big[\opnorm{\sqrt{N}\omega}_{q_1}^{q_1}t^{q_1 - p - 1} + \opnorm{\sqrt{N}\omega}_{q_2}^{q_2}t^{q_2 - p - 1}\big] dt, \label{Equation:Large_Truncation_3}
    \end{align}
    where \eqref{Equation:Large_Truncation_2} holds by condition \ref{Condition:V} and \eqref{Equation:Large_Truncation_3} uses change of variables $\sigma = \sqrt{N}t\omega$ for $\omega \in \mathbb{S}^{N - 1}$. Since $q_2 - p > q_1 - p > 0$, we have that for $N$ large enough, 
    \[I_{\exp{s_N^K(\cdot)}}(\{\opnorm{\sigma}_2 \leq \delta, \opnorm{\sigma}_{2q_2 - 2} \leq M\}) \leq e^{N(s_N^K(\underline{0}) + \epsilon)} N^{N/2}S_{N - 1}(1 + \epsilon)^{N/2}.\]
    Combining the upper and lower bound yields that for large $N,$
    \[\frac{I_{\exp{s_N^K(\cdot)}}(\{\opnorm{\sigma}_2 \leq \delta, \opnorm{\sigma}_{2q_2 - 2} \leq M\})}{I_{\exp{s_N^K(\cdot)}}(\{\opnorm{\sigma}_2 \geq \delta, \opnorm{\sigma}_{2q_2 - 2} \leq M\})} \leq e^{N\epsilon} (1 + \epsilon)^N.\]
This completes the proof of our second assertion. 
\end{proof}

\begin{proposition}\label{Proposition:Main_Theorem_Upper}
    If $V$ satisfies Assumption \ref{ass1}, then for all $u \geq 0$, we have the upper bound
    \[\limsup_{N \to \infty}\dfrac{1}{N}\log \mathbb{E}[\text{Crt}_N((Nu, \infty))] \leq \sup\{\,\mathcal{I}(\mu) \mid \mu \in \mathcal{P}_{2q_2 - 2}(\mathbb{R}) \cap \mathfrak{D}(u)\}.\]
\end{proposition}

\begin{proof}
    Fix $\epsilon > 0$. First, observe that there exist constants $M = M(\epsilon)$, $K = K(M, \epsilon)$,  and $\delta = \delta(M, K, \epsilon)$ such that
    \begin{align}
        &\limsup_{N \to \infty} \frac{1}{N}\log \mathbb{E}[\text{Crt}_N((Nu, \infty))] \notag \\
        &= \frac{1}{2}\log\Big(\frac{p - 1}{2\pi}\Big) + \limsup_{N \to \infty} \frac{1}{N}\log I_{\mathbb{E}[|M_{N - 1}(\cdot)|]}(\mathbb{R}^N) \label{Equation:Main_Upper_Bound_1} \\
        &\leq \epsilon + \frac{1}{2}\log\Big(\frac{p - 1}{2\pi}\Big) + \limsup_{N \to \infty} \frac{1}{N}\log I_{\mathbb{E}[|\det Q_N^K(\cdot)|]}(\{\opnorm{\sigma}_{2q_2 - 2} \leq M\}) \label{Equation:Main_Upper_Bound_2} \\
        &= \epsilon + \frac{1}{2}\log\Big(\frac{p - 1}{2\pi}\Big) + \limsup_{N \to \infty} \frac{1}{N}\log I_{\exp{s_N^K(\cdot)}}(\{\opnorm{\sigma}_{2q_2 - 2} \leq M\}) \label{Equation:Main_Upper_Bound_3} \\
        &\leq 2\epsilon + \frac{1}{2}\log\Big(\frac{p - 1}{2\pi}\Big) + \limsup_{N \to \infty} \frac{1}{N}\log I_{\exp{s_N^K(\cdot)}}(\{\delta \leq \opnorm{\sigma}_2, \opnorm{\sigma}_{2q_2 - 2} \leq M\}) \label{Equation:Main_Upper_Bound_4} \\
        \nonumber&= 2\epsilon + \frac{1}{2}\log(p - 1) + \frac{1}{2} 
        \\
        &+ \limsup_{N \to \infty} \frac{1}{N}\log\Big\{\int_\delta^M t^{-p - 1} 
         \mathbb{E}\Big[\frac{|\langle\sqrt{N}t\omega, v(\sqrt{N}t\omega) \rangle|}{N}e^{N\varphi_K(t, L_{\sqrt{N}\omega, N})}\mathbbm{1}_{\{(t, L_{\sqrt{N}\omega, N}) \in \mathfrak{F}(u)\}}\Big]dt\Big\}, \label{Equation:Main_Upper_Bound_5}
    \end{align}
    where \eqref{Equation:Main_Upper_Bound_2} holds by using Propositions \ref{Proposition:Main_Truncation} and \ref{Proposition:q_Truncation}, \eqref{Equation:Main_Upper_Bound_3} uses Proposition \ref{Proposition:Application_of_Ben},  \eqref{Equation:Main_Upper_Bound_4} follows from Lemma \ref{Lemma:Upper_Truncation}, and \eqref{Equation:Main_Upper_Bound_5} is due to change of variables $\sigma = \sqrt{N}t\omega$ where $\omega \in \mathbb{S}^{N - 1}$. Observe that by condition \ref{Condition:V} and the definition of $v$ from \eqref{Equation:Vector},
    \begin{align*}
        t^{-p - 1}\frac{|\langle\sqrt{N}t\omega, v(\sqrt{N}t\omega) \rangle|}{N} &\leq \oldconstant{Constant:A_1}t^{-p-1} \opnorm{(\sqrt{N}t\omega)^2V''(\sqrt{N}t\omega)}_1 \\
        &\leq \oldconstant{Constant:Bound}\oldconstant{Constant:A_1}t^{-p-1}\big(\opnorm{\sqrt{N}t\omega}_{q_1}^{q_1} + \opnorm{\sqrt{N}t\omega}_{q_2}^{q_2}\big) \\
        &\leq \oldconstant{Constant:Bound}\oldconstant{Constant:A_1}\sum_{i = 1}^2\max\{M^{q_i - p - 1}, \delta^{q_i - p - 1}\}N^{q_i/2 - 1}.
    \end{align*}
    Hence, the left-hand side is negligible in \eqref{Equation:Main_Upper_Bound_5} so that after noting that $L_{\sqrt{N}\omega,N}$ is identically distributed as $\nu_{g,N},$
    \begin{align}
        &\limsup_{N \to \infty} \frac{1}{N}\log \mathbb{E}[\text{Crt}_N((Nu, \infty))] \notag \\
        &\leq 2\epsilon + \frac{1}{2}\log(p - 1) + \frac{1}{2} + \limsup_{N \to \infty} \frac{1}{N}\log\mathbb{E}\Big[e^{N\varphi_K(T, \nu_{g, N})}\mathbbm{1}_{\{(T, \nu_{g, N}) \in \mathfrak{F}(u)\}}\Big], \label{Equation:Main_Upper_Bound_6}
    \end{align}
    where $T$ is sampled uniformly from $[\delta, M]$. Set $q_2 < s < 2q_2 - 2$ and $\Phi \colon [\delta, M] \times \mathcal{P}(\mathbb{R}) \to \mathbb{R}$ by $\Phi = \phi_1 + \phi_{3, K}$. Recall that we have shown in Proposition \ref{Proposition:Main_Theorem_Lower} that $\Phi$ is a continuous function on $([\delta, M] \times \mathcal{P}_s(\mathbb{R}), \mathscr{T} \otimes \mathscr{W}_s)$ satisfying \eqref{Equation:Varadhan_Tilted_Condition_1} and \eqref{Equation:Varadhan_Tilted_Condition_2}. Therefore, we can apply the upper bound in Theorem \ref{Theorem:Varadhan_Tilted} to the right-hand side of \eqref{Equation:Main_Upper_Bound_6} with closed set $\mathfrak{C} = \mathfrak{F}(u)$, which leads to
    \begin{equation}\label{Equation:Main_Upper_Bound_9}
        \limsup_{N \to \infty} \frac{1}{N}\log \mathbb{E}[\text{Crt}_N((Nu, \infty))] \leq 2\epsilon + \sup_{\delta \leq t \leq M, m_2(\mu) = 1, (t, \mu) \in \mathfrak{F}(u)} \mathcal{I}_K(t, \mu).
    \end{equation}
    In the last step, by utilizing Proposition \ref{Proposition:Approximation_of_phi}, we can pass to the limit in \eqref{Equation:Main_Upper_Bound_9} in the order  $K \uparrow \infty$ and then $\delta \downarrow 0, M \uparrow \infty$ to deduce the asserted bound again relying on \eqref{Equation:Reduction_t}.
\end{proof}

\begin{proposition}\label{Proposition:Matching}
    If $V$ satisfies conditions \ref{Condition:V} and \ref{Condition:V''}, then for all $u \geq 0$,
    \[\sup\{\,\mathcal{I}(\mu) \mid \mu \in \mathcal{P}_{2q_2 - 2}(\mathbb{R}) \cap \mathfrak{D}_+(u)\} = \sup\{\,\mathcal{I}(\mu) \mid \mu \in \mathcal{P}_{2q_2 - 2}(\mathbb{R}) \cap \mathfrak{D}(u)\}.\]
\end{proposition}

\begin{proof}
    Suppose $\mu \in \mathfrak{D}(u) \setminus \mathfrak{D}_+(u)$ with $0 < m_{2q_2 - 2}(\mu) < \infty$. For $\epsilon > -1$, consider the measure $\mu_\epsilon = S_{1 + \epsilon}(\mu)$. Observe that from \ref{Condition:V} and \ref{Condition:V''}, the function in the set $\mathfrak{D}(u)$ is strictly increasing since by \eqref{Equation:V},
    \[\frac{d}{dx}\big[p^{-1}xV'(x) - V(x)\big] = p^{-1}\big[xV''(x) - (p - 1)V'(x)\big] \geq p^{-1}(q - p)V'(x) > 0\]
    for all $x > 0$. Therefore, $\mu_\epsilon \in \mathfrak{D}_+(u)$ for all $\epsilon > 0$. Now, we claim that the mapping $\epsilon \mapsto \mathcal{I}(\mu_\epsilon)$ is continuous at $\epsilon = 0$. Observe that for the first two terms in $\mathcal{I}(\mu_\epsilon)$, namely, 
    \[\epsilon \mapsto \frac{m_2(\mu)^{-p}}{2(1 + \epsilon)^{2p - 2}p^2}\Big[(p - 1)\mathbb{E}_{\mu}\big[XV'((1 + \epsilon)X)\big]^2 - pm_2(\mu)\mathbb{E}_{\mu}\big[V'((1 + \epsilon)X)^2\big]\Big]\]
    is continuous by dominate convergence theorem. For the third term, observe that
    \[\big(g_{m_2(\mu_\epsilon)^{1/2}}\big)_\ast\mu_\epsilon = \Big((1 + \epsilon)^{2 - p}m_2(\mu)^{1 - p/2}V''((1 + \epsilon)|\cdot|)\Big)_\ast \mu\]
    converge to $(g_{m_2(\mu)^{1/2}})_\ast\mu$ weakly by dominated convergence theorem. Therefore, the convolution $\nu_\epsilon = (g_{m_2(\mu_\epsilon)^{1/2}})_\ast\mu_\epsilon \boxplus \mu_{\text{sc}}$ also converge weakly by Theorem \ref{Theorem:Triangle_Equation_Free_Convolution}. Note that
    \begin{align*}
        \int_{\mathbb{R}} \log |\lambda| \nu_\epsilon(d\lambda) \leq \int_{\{|\lambda| < \delta\}} \log|\lambda| \nu_\epsilon(d\lambda) + \int_{\{\delta \leq |\lambda| \leq M\}} \log|\lambda| \nu_\epsilon(d\lambda) + \int_{\{|\lambda| > M\}} |\lambda| \nu_\epsilon(d\lambda).
    \end{align*}
    Observe that the first term vanishes as $\delta \to 0^+$ since $\nu_\epsilon$ has bounded density by Theorem \ref{Theorem:Free_Convolution_Semi_Circle}; the third term vanishes as $M \to \infty$ by uniform integrability (to be precise, $\nu_\epsilon$ has uniformly bounded $2q_2 - 2$ moment); and the second term vanishes as $\epsilon \to 0^+$ by weak convergence. Therefore, for all $\delta > 0$, we may choose $\epsilon > 0$ small enough so that
    \[\mathcal{I}(\mu_\epsilon) > \mathcal{I}(\mu) - \delta.\]
    We may then conclude the equality for the supremum.
\end{proof}

\begin{proof}[\bf Proof of Theorem \ref{Theorem:Main_Theorem}]
    Note that Theorem \ref{Theorem:Main_Theorem} follows directly from Proposition \ref{Proposition:Main_Theorem_Lower}, Proposition \ref{Proposition:Main_Theorem_Upper}, and Proposition \ref{Proposition:Matching}.
\end{proof}

\section{Analysis of the Variational Formula}\label{Section:Analysis}

\subsection{Proof of Proposition \ref{Proposition:Finite_of_I}}

Note that the supremum is never $-\infty$ since for all $u \geq 0$, we may always choose $\mu = S_t(\mu_{\text{Norm}})$ and $t > 0$ large enough so that $\mu \in \mathfrak{D}(u)$.

To show that the supremum is never $\infty$, it suffices to prove it is finite when $u = 0$. Observe that for all $\mu \in \mathcal{P}_{2q_2 - 2}(\mathbb{R})$, by setting $m_2(\mu) = t^2$,
\begin{align}
    \varphi(t, \mu) &\leq -\frac{t^{2-2p}}{2p^2}\mathbb{E}_\mu\big[V'(X)^2\big] + \int_{\mathbb{R}} \log|\lambda| \bigl((g_t)_\ast\mu \boxplus \mu_{\text{sc}}\bigr)(\text{d}\lambda) \notag \\
    &\leq -\frac{t^{2-2p}}{2p^2}\mathbb{E}_\mu\big[V'(X)^2\big] + \dfrac{1}{2}\log m_2 \bigl((g_t)_\ast\mu \boxplus \mu_{\text{sc}}\bigr) \label{Equation:Rough_Upper_Bound_Var_1} \\
    &\leq -\frac{t^{2-2p}}{2p^2} \mathbb{E}_\mu\big[V'(X)^2\big]  + \dfrac{1}{2} \log \bigl[4 (m_2((g_t)_\ast\mu) + 1)\bigr] \label{Equation:Rough_Upper_Bound_Var_2} \\
    &\leq -\frac{t^{2-2p}}{2c_1^2p^2} \big( m_{2q_1 -2}(\mu) + m_{2q_2 -2}(\mu) \big)  + \dfrac{1}{2} \log \Big[\frac{8\oldconstant{Constant:Bound}^2t^{4 - 2p}}{p(p - 1)} \big(m_{2q_1 - 4}(\mu) + m_{2q_2 - 4}(\mu)\big) + 4\Big], \label{Equation:Rough_Upper_Bound_Var_3}
\end{align}
where \eqref{Equation:Rough_Upper_Bound_Var_1} holds by Jensen's inequality; \eqref{Equation:Rough_Upper_Bound_Var_2} holds by \eqref{Equation:Free_Addition_m_2} with $s = 2$; \eqref{Equation:Rough_Upper_Bound_Var_3} holds by the definition of $g_t$, the condition \ref{Condition:V} and the elementary inequality $(a + b)^2 \leq 2(a^2 + b^2)$.

Let $X \sim \mu$ and $Y = X^2.$ By FKG inequality, $\mathbb{E}[Y^{q_i - 1}] \ge \mathbb{E}[Y^{q_i - 2}] \mathbb{E}[Y]$ for $i=1, 2,$ which yields
\[ t^2 m_{2q_i-4}(\mu) \le   m_{2q_i-2}(\mu), \quad i=1, 2.   \]
Therefore, \eqref{Equation:Rough_Upper_Bound_Var_3} can be bounded above by 
\begin{equation}\label{Equation:Rough_Upper_Bound_Var_40}
   -\frac{t^{2-2p}}{2\oldconstant{Constant:Bound}^2p^2} \big(m_{2q_1 -2}(\mu) + m_{2q_2 -2}(\mu) \big) + \dfrac{1}{2} \log \Big[\frac{8\oldconstant{Constant:Bound}^2t^{2 - 2p}}{p(p - 1)} \big(m_{2q_1 - 2}(\mu) + m_{2q_2 - 2}(\mu)\big) + 4\Big]
\end{equation}
If we define
\[w(x) = -\frac{x}{2\oldconstant{Constant:Bound}^2p^2} + \dfrac{1}{2}\log \Big(\frac{8 \oldconstant{Constant:Bound}^2}{p(p - 1)}x + 4\Big), \quad x \ge 0\]
then \eqref{Equation:Rough_Upper_Bound_Var_40} is same  as $w$ evaluated at $x= t^{2 - 2p} \big(m_{2q_1 - 2}(\mu) + m_{2q_2 - 2}(\mu) \big).$ Hence, we  have 
\[\varphi(t, \mu)  \le \sup_{x \geq 0} w(x) < \infty, \]
for an arbitrary choice of $\mu$, as desired.

\subsection{Proof of Proposition \ref{Proposition:Critical_Level}} 

First, we show \eqref{Equation:Critical_Level_2}. Observe that the probability
\begin{align*}
    \mathbb{P}(u_N > u_c + \epsilon) &\leq \mathbb{P}(\text{Crt}_N([N(u_c + \epsilon), \infty)) \geq 1) \\
    &\leq \mathbb{E}\bsq{\text{Crt}_N([N(u_c + \epsilon), \infty))} \\
    &= \exp\bsq{N\bpar{\Sigma(u_c + \epsilon) + o_{N \to \infty}(1)}}.
\end{align*}
Since the $\Sigma(u_c + \epsilon) < 0$ by definition of $u_c$, we see that $\mathbb{P}(u_N > u_c + \epsilon) \to 0$ as $N \to \infty$ and therefore \eqref{Equation:Critical_Level_2} holds.

Now, we turn to \eqref{Equation:Critical_Level_1}. For the lower bound, we see by plugging $\mu = S_t(\mu_{\text{Norm}})$, we have
\begin{align*}
    \mathcal{I}(S_t(\mu_{\text{Norm}})) &= \frac{1}{2}\log(p - 1) + \frac{1}{2} + \frac{t^{2-2p}}{2p^2}\Big[(p - 1)\mathbb{E}_{\mu_{\text{Norm}}}\big[XV'(tX)\big]^2 - p\mathbb{E}_{\mu_{\text{Norm}}}\big[V'(tX)^2\big]\Big] \\
    &+ \int_{\mathbb{R}} \log|\lambda| \big((g_t)_\ast S_t(\mu_{\text{Norm}}) \boxplus \mu_{\text{sc}}\big)(d\lambda) - \text{KL}(S_t(\mu_{\text{Norm}}) \,\|\, \mu_{\text{Norm}}) - \frac{1}{2}(1 - t^2 + 2\log t).
\end{align*}
Note that by the Cauchy-Schwarz inequality and condition \ref{Condition:V}, we have
\begin{align*}
    &\frac{t^{2-2p}}{2p^2}\Big[(p - 1)\mathbb{E}_{\mu_{\text{Norm}}}\big[XV'(tX)\big]^2 - p\mathbb{E}_{\mu_{\text{Norm}}}\big[V'(tX)^2\big]\Big] \\
    &\geq p^{-1}t^{2(q_1 - p)}(m_{2q_1 - 2}\big(\mu_{\text{Norm}}) + t^{2(q_2 - q_1)}m_{2q_2 - 2}(\mu_{\text{Norm}})\big) = o_{t \to 0^+}(1)
\end{align*}
Moreover, one can again argue as in Proposition \ref{Proposition:Matching} that by weak convergence and uniform integrability,
\[\lim_{t \to 0^+}\int_{\mathbb{R}} \log|\lambda| \big((g_t)_\ast S_t(\mu_{\text{Norm}}) \boxplus \mu_{\text{sc}}\big)(d\lambda) = \int_{\mathbb{R}} \log|\lambda| \mu_{\text{sc}}(d\lambda) = -\frac{1}{2}.\]
Last, note that the divergence term cancels out with the $\frac{1}{2}(1 - t^2 + 2\log t)$ and we may conclude
\[\lim_{t \to 0^+} \mathcal{I}(S_t(\mu_{\text{Norm}})) = \frac{1}{2}\log(p - 1) > 0\]
provided $p \geq 3$. This shows that there exist $t > 0$ such that $\mathcal{I}(S_t(\mu_{\text{Norm}})) > 0$ and therefore $u > 0$ so that $\mathbb{E}[\text{Crt}_N((-\infty, Nu])]$ has positive log-asymptote, giving $u_c > 0$ whenever $p \geq 3$.

Now, for the case $p = 2$, we see that by \eqref{Equation:Critical_Level_2}, it suffices to show that $u_N$ is lower bounded by some positive constant for large $N$. Observe that is we set $G_N$ to be a GOE, then
\[u_N = \sup_{\sigma \in \mathbb{R}^N} \frac{1}{N}\Bpar{\frac{1}{\sqrt{N}}\sum_{i, j = 1}^N g_{ij}\sigma_i\sigma_j - \sum_{i = 1}^N V(\sigma_i)} = \sup_{\sigma \in \mathbb{R}^N} \Big(\dfrac{1}{\sqrt{2}N}\sigma^\top G_N\sigma - \opnorm{V(\sigma)}_1\Big).\]
Set $\sigma_{\max}$ to be the unit eigenvector associated to the maximum eigenvalue of $G_N$ (which is uniformly distributed on $\mathbb{S}^{N - 1}$ \cite[Corollary 2.5.4]{anderson2010introduction}), then by plugging $\sigma = \sqrt{N}\epsilon\sigma_{\max}$, we have 
\[u_N \geq \frac{\epsilon^2}{\sqrt{2}}\lambda_N(G_N) - \opnorm{V(\sqrt{N}\epsilon\sigma_{\text{max}})}_1 \geq \frac{\epsilon^2}{\sqrt{2}}\lambda_N(G_N) - \oldconstant{Constant:Bound}\big(\opnorm{\sqrt{N}\epsilon\sigma_{\text{max}}}_{q_1}^{q_1} + \opnorm{\sqrt{N}\epsilon\sigma_{\text{max}}}_{q_2}^{q_2}\big).\]
For the first term, we have by \cite{geman1980limit} that $\lambda_N(G_N) \to 2$ almost surely. Moreover, by setting the standard normal vector $g \in \mathbb{R}^N$, we see that
\[\opnorm{\sqrt{N}\epsilon\sigma_{\text{max}}}_{q_1}^{q_1} + \opnorm{\sqrt{N}\epsilon\sigma_{\text{max}}}_{q_2}^{q_2} \overset{d}{=} \epsilon^{q_1} \dfrac{\opnorm{g}_{q_1}^{q_1}}{\opnorm{g}_2^{q_1}} + \epsilon^{q_2} \dfrac{\opnorm{g}_{q_2}^{q_2}}{\opnorm{g}_2^{q_2}} \to \epsilon^{q_1} m_{q_1}(\mu_{\text{Norm}}) + \epsilon^{q_2} m_{q_2}(\mu_{\text{Norm}})\]
almost surely by the strong law of large numbers. Therefore, we have
\[\limsup_{N \to \infty} u_N \geq \sqrt{2}\epsilon^2 - \big( \epsilon^{q_1}m_{q_1}(\mu_{\text{Norm}}) + \epsilon^{q_2 }m_{q_2}(\mu_{\text{Norm}})\big)\]
almost surely. Since $q_2 \ge q_1 > 2$, by picking $\epsilon > 0$ appropriately, we may obtain a positive uniform lower bound and thus $u_c > 0$ for $p = 2$.

For the upper bound, we observe that if $\sigma \in \Omega(u)$, then
\[u \leq \frac{1}{N}\sum_{i = 1}^N \big(p^{-1}\sigma_iV'(\sigma_i) - V(\sigma_i)\big) \leq \oldconstant{Constant:Bound} p^{-1}(\opnorm{\sigma}_{q_1}^{q_1} + \opnorm{\sigma}_{q_2}^{q_2}) \leq \oldconstant{Constant:Bound}p^{-1}\big(\opnorm{\sigma}_{2q_2 - 2}^{q_1} + \opnorm{\sigma}_{2q_2 - 2}^{q_2}\big).\]
Therefore, for all $M > 0$, there exist $u > 0$ so that $\Omega(u) \subseteq \{\opnorm{\sigma}_{2q_2 - 2} \geq M\}$. We may then apply \eqref{Equation:Rough_Upper_Bound} with $w_N(\sigma) = \mathbb{E}[|\det M_{N - 1}(\sigma)|]$ and see that there exist $u > 0$ (independent of $N$) so that $N^{-1}\log \mathbb{E}[\text{Crt}_N([Nu, \infty))] \leq -M$ for $N$ large enough. Hence, we have $u_c < \infty$. \hfill $\square$

\begin{appendices}

\appendix

\section{Computation of Covariance Structure}\label{Section:Covariance}

\begin{proof}[\bf Proof of Proposition \ref{Proposition:Cov_Structure}]    The computation for the means is straightforward; we omit the details. For $x, y \in \mathbb{R}^N$, define the covariance function of $H_N$ as 
\begin{align*}
    C(x, y) &= \mathbb{E}\left[\big(H_N(x) - \mathbb{E}\big[H_N(x)\big]\big)\big(H_N(y) - \mathbb{E}\big[H_N(y)\big]\big)\right]= N^{1 - p}\lan{x, y}^p.
\end{align*}
Recall a useful formula (see \cite[Eq (5.5.4)]{adler2009random}),
\begin{equation}\label{Equation:covariance_Formula}
    \text{Cov}\left(\frac{\partial^k H_N(x)}{\partial x_{i_1} \cdots \partial x_{i_k}}, \frac{\partial^\ell H_N(y)}{\partial y_{j_1} \cdots \partial y_{j_\ell}}\right) = \frac{\partial^{k + \ell} C(x, y)}{\partial x_{i_1} \cdots \partial x_{i_k} \partial y_{j_1} \cdots \partial y_{j_\ell}}.
\end{equation}
    For \eqref{Equation:Cov_H}, we have by \eqref{Equation:covariance_Formula} that the variance is
    \[\text{Var}(H_N(\sigma)) = C(\sigma, \sigma) = N^{1 - p}\|\sigma\|_2^{2p}.\]
    For \eqref{Equation:Cov_H_Grad}, we have again by \eqref{Equation:covariance_Formula} that
    \[\text{Cov}\bpar{H_N(\sigma), \nabla H_N(\sigma)} = \left.\nabla_{x} C(x, y)\right|_{x, y = \sigma} = \left.N^{1 - p} \cdot p\lan{x, y}^{p - 1}x \right|_{x, y = \sigma} = N^{1 - p}p\|\sigma\|_2^{2(p - 1)}\sigma.\]
    For \eqref{Equation:Cov_H_Hess}, we have by \eqref{Equation:covariance_Formula} that
    \begin{align*}
        &\text{Cov}\bpar{H_N(\sigma), \nabla^2H_N(\sigma)} \\
        &= \nabla_{y}^2C(x, y)\big|_{x, y = \sigma} = \Bpar{\partial_{y_i}\bpar{N^{1 - p}p\lan{x, y}^{p - 1}x_j} \big|_{x, y = \sigma}}_{1 \leq i, j \leq N} = N^{1 - p}p(p - 1)\|\sigma\|_2^{2(p - 2)}\sigma\sigma^\top.
    \end{align*}
    Now, for \eqref{Equation:Cov_Grad}, we have by \eqref{Equation:covariance_Formula} that
    \begin{align*}
        \text{Cov}\bpar{\partial_{\sigma_i}H_N(\sigma), \partial_{\sigma_j}H_N(\sigma)} &= \partial_{x_i}\partial_{y_j}C(x, y)\big|_{x, y = \sigma} \notag \\
        &= \partial_{x_i}\left(N^{1 - p}p\lan{x, y}^{p - 1}x_j\right) \big|_{x, y = \sigma} \notag \\
        &= N^{1 - p}p\|\sigma\|_2^{2(p - 2)}\left[(p - 1)\sigma_i\sigma_j + \|\sigma\|_2^2\delta_{ij}\right].
    \end{align*}
    For \eqref{Equation:Cov_Grad_Hess}, we have by \eqref{Equation:covariance_Formula} that 
    \begin{align*}
        \text{Cov}\bpar{\partial_{\sigma_i}H_N(\sigma), \partial_{\sigma_j\sigma_k}H_N(\sigma)} &= \left.\partial_{x_i}\partial_{y_jy_k}^2C(x, y)\right|_{x, y = \sigma} \notag \\
        &= \left.\partial_{y_k}\left(N^{1 - p}p\left[(p - 1)\lan{x, y}^{p - 2}x_jy_i + \lan{x, y}^{p - 1}\delta_{ij}\right]\right)\right|_{x, y = \sigma} \notag \\
        &= N^{1 - p}p(p - 1)\|\sigma\|_2^{2(p - 3)}\left[(p - 2)\sigma_i\sigma_j\sigma_k + \|\sigma\|_2^2(\sigma_j\delta_{ik} + \sigma_k\delta_{ij})\right]. 
    \end{align*}
    Last but not least, for \eqref{Equation:Cov_Hess}, we see by \eqref{Equation:covariance_Formula} that
    \begin{align*}
        &\text{Cov}\bpar{\partial_{\sigma_i\sigma_j}H_N(\sigma), \partial_{\sigma_k\sigma_\ell}H_N(\sigma)} \\ 
        &= \partial_{x_ix_j}^2\partial_{y_ky_\ell}^2C(x, y)\big|_{x, y = \sigma} \\
        &= \left.\partial_{x_i}\bpar{N^{1 - p}p(p - 1)\bsq{(p - 2)\lan{x, y}^{p - 3}x_kx_\ell y_j + \lan{x, y}^{p - 2}(x_\ell\delta_{jk} + x_k\delta_{j\ell})}}\right|_{x, y = \sigma} \\
        &= N^{1 - p}p(p - 1)\|\sigma\|_2^{2(p - 4)}\big[(p - 2)(p - 3)\sigma_i\sigma_j\sigma_k\sigma_\ell \\
        &+ \|\sigma\|_2^2(p - 2)\big(\delta_{ik}\sigma_j\sigma_\ell + \delta_{i\ell}\sigma_j\sigma_k + \delta_{jk}\sigma_i\sigma_\ell + \delta_{j\ell}\sigma_i\sigma_k\big) + \|\sigma\|_2^4(\delta_{ik}\delta_{j\ell} + \delta_{i\ell}\delta_{jk})\big].
    \end{align*}
    We have then shown our assertion in Proposition \ref{Proposition:Cov_Structure}.
\end{proof}

\begin{proof}[\bf Proof of Eq \eqref{Equation:Condition_Cov_Hess}] Note that by \eqref{Equation:Gaussian_Condition}, the covariance
    \begin{align*}
        & \text{Cov}\bpar{\partial_{\sigma_i\sigma_j}^2H_N(\sigma), \partial_{\sigma_k\sigma_\ell}^2H_N(\sigma) \,\big\vert\,\nabla H_N(\sigma)} \\
        &= \text{Cov}\bpar{\partial_{\sigma_i\sigma_j}^2H_N(\sigma), \partial_{\sigma_k\sigma_\ell}^2H_N(\sigma)} \\
        &\hspace{1cm} - \bsq{\text{Cov}(\nabla H_N(\sigma), \nabla^2H_N(\sigma)) \cdot \text{Cov}(\nabla H_N(\sigma))^{-1} \cdot \text{Cov}(\nabla^2H_N(\sigma), \nabla H_N(\sigma))}_{ij, k\ell} \\
        &= N^{1 - p}p(p - 1)\|\sigma\|_2^{2(p - 4)}\big[(p - 2)(p - 3)\sigma_i\sigma_j\sigma_k\sigma_\ell \\
        &\hspace{1cm} + \|\sigma\|_2^2(p - 2)\big(\delta_{ik}\sigma_j\sigma_\ell + \delta_{i\ell}\sigma_j\sigma_k + \delta_{jk}\sigma_i\sigma_\ell + \delta_{j\ell}\sigma_i\sigma_k\big) + \|\sigma\|_2^4(\delta_{ik}\delta_{j\ell} + \delta_{i\ell}\delta_{jk})\big] \\
        &\hspace{1cm} - \sum_{\alpha, \beta = 1}^N b\bpar{N^{1 - p}p(p - 1)\|\sigma\|_2^{2(p - 3)}\left[(p - 2)\sigma_i\sigma_j\sigma_\alpha + \|\sigma\|_2^2\big(\delta_{i\alpha}\sigma_j + \delta_{j\alpha }\sigma_i\big)\right]} \\
        &\hspace{1cm} \cdot \bsq{N^{p - 1}p^{-2}\|\sigma\|_2^{-2p}\bpar{(1 - p)\sigma_\alpha\sigma_\beta + p\|\sigma\|_2^2\delta_{\alpha\beta}}} \\
        &\hspace{1cm} \cdot \bpar{N^{1 - p}p(p - 1)\|\sigma\|_2^{2(p - 3)}\left[(p - 2)\sigma_k\sigma_\ell\sigma_\beta + \|\sigma\|_2^2\big(\delta_{\beta k}\sigma_\ell + \delta_{\beta\ell}\sigma_k\big)\right]} \\
        &= N^{1 - p}p(p - 1)\|\sigma\|_2^{2(p - 4)}\big[2\sigma_i\sigma_j\sigma_k\sigma_\ell \\
        &\hspace{1cm} - \|\sigma\|_2^2\big(\delta_{ik}\sigma_j\sigma_\ell + \delta_{i\ell}\sigma_j\sigma_k + \delta_{jk}\sigma_i\sigma_\ell + \delta_{j\ell}\sigma_i\sigma_k\big) + \|\sigma\|_2^4(\delta_{ik}\delta_{j\ell} + \delta_{i\ell}\delta_{jk})\big] \\
        &= N^{-1}p(p - 1)\opnorm{\sigma}_2^{2(p - 2)}\Bsq{\Bpar{\delta_{ik} - \dfrac{\sigma_i\sigma_k}{\|\sigma\|_2^2}}\Bpar{\delta_{j\ell} - \dfrac{\sigma_j\sigma_\ell}{\|\sigma\|_2^2}} + \Bpar{\delta_{i\ell} - \dfrac{\sigma_i\sigma_\ell}{\|\sigma\|_2^2}}\Bpar{\delta_{jk} - \dfrac{\sigma_j\sigma_k}{\|\sigma\|_2^2}}}
    \end{align*}
    and we may conclude \eqref{Equation:Condition_Cov_Hess}.
\end{proof}

\section{Free Convolution}\label{Section:Free_Convolution}

\noindent This section collects some key facts about the free convolution of probability measures with the semi-circle law. 

\begin{theorem}[{\cite[Corollaries 2, 4, 5]{biane1997free}}]\label{Theorem:Free_Convolution_Semi_Circle}
    For all $\mu \in \mathcal{P}(\mathbb{R})$, the probability measure $\mu \boxplus \mu_{\text{sc}}$ admits continuous density $f_{\mu \boxplus \mu_{\text{sc}}}$ with respect to the Lebesgue measure. Moreover, it is analytic on where it is positive, and there exists $\newconstant\label{Constant:Free_Convolution_Infty} > 0$ such that
    \[\|f_{\mu \boxplus \mu_{\text{sc}}}\|_\infty \leq \oldconstant{Constant:Free_Convolution_Infty}.\]
    Moreover, if $m_\infty(\mu) < \infty$, then
    \[m_\infty(\mu \boxplus \mu_{\text{sc}}) \leq m_\infty(\mu) + 2.\]
\end{theorem}

\begin{theorem}
    If $\mu \in \mathcal{P}_s(\mathbb{R})$, then for all $s > 0$,
    \begin{equation}\label{Equation:Free_Addition_m_2}
        m_s(\mu \boxplus \mu_{\text{sc}}) \leq 2^s(m_s(\mu) + m_s(\mu_{\text{sc}})).
    \end{equation}
\end{theorem}

\begin{proof}
    By \cite[Proposition 5.3.34]{anderson2010introduction}, we can always find a $W^\ast$-probability space $(\mathscr{A}, \tau)$ where $\tau$ is a normal faithful tracial state and $a, b$ are self-adjoint operators affiliated with $\mathscr{A}$ such that $a, b$ has law $\mu$, $\mu_{\text{sc}}$, respectively, and which are free. Note that $a \in \mathcal{L}^s(\mathscr{A}, \tau)$ and $b \in \mathcal{L}^\infty(\mathscr{A}, \tau) = \mathscr{A}$ (we adapt the non-commutative $\mathcal{L}^p$-space notation from \cite{pisier2003non}). Recall that $\|\cdot\|_s = \tau(|\cdot|^s)^{1/s}$ is a quasi-norm if $0 < s < 1$ and a norm if $s \geq 1$ on $\mathcal{L}^s(\mathscr{A}, \tau)$ and hence we always have
    \[m_s(\mu \boxplus \mu_{\text{sc}}) = \tau(|a + b|^s) \leq \begin{cases}
            \tau(|a|^s) + \tau(|b|^s) &, 0 < s < 1; \\
            (\tau(|a|^s)^{1/s} + \tau(|b|^s)^{1/s})^s &, s \geq 1,
        \end{cases}\]
    which both can be dominated by $2^s(\tau(|a|^s) + \tau(|b|^s)) = 2^s(m_s(\mu) + m_s(\mu_{\text{sc}}))$.
\end{proof}

\begin{theorem}[{\cite[Proposition 4.13]{bercovici1993free}}]\label{Theorem:Triangle_Equation_Free_Convolution}
    If $\mu_1$, $\mu_2$, $\nu_1$, $\nu_2$ are probability measures on $\mathbb{R}$, we have
    \[d_{\text{KS}}(\mu_1 \boxplus \nu_1, \mu_2 \boxplus \nu_2) \leq d_{\text{KS}}(\mu_1, \mu_2) + d_{\text{KS}}(\nu_1, \nu_2)\]
    and also
    \[d_{\text{L}}(\mu_1 \boxplus \nu_1, \mu_2 \boxplus \nu_2) \leq d_{\text{L}}(\mu_1, \mu_2) + d_{\text{L}}(\nu_1, \nu_2).\]
\end{theorem}

\section{Random Matrix Results}

\noindent In this section, we gather some important facts from the random matrix theory.

\subsection{Concentration Inequality}

\begin{lemma}[{\cite[Lemma 6.3]{arous2001aging}}]\label{Lemma:Tail_of_Norm}
    For $t \geq 8$ and all $N \ge 1$,
    \begin{equation}\label{Equation:Tail_of_GOE}
        \mathbb{P}(\|\text{GOE}_N\|_{\text{op}} \geq t) \leq e^{-Nt^2/9}.
    \end{equation}
\end{lemma}

\begin{corollary}
    There exist $\newconstant\label{Constant:UB_GOE} > 0$ such that for all $N \ge 1$,
    \begin{equation}\label{Equation:Moment_of_GOE}
        \mathbb{E}[\|\text{GOE}_N\|_{\text{op}}^N] \leq \oldconstant{Constant:UB_GOE}^N.
    \end{equation}
\end{corollary}

\begin{proof}
    This is a direct consequence of \eqref{Equation:Tail_of_GOE} and the expression
    \[\mathbb{E}\big[\|\text{GOE}_N\|_{\text{op}}^N\big] = \int_0^\infty t^N \cdot \mathbb{P}(\|\text{GOE}_N\|_{\text{op}} \geq t)\,dt.\]
\end{proof}

\begin{definition}[Concentration]
    A sequence of $N \times N$ random symmetric matrices $H_N$ is said to have \textbf{concentration over Lipschitz function} if there exist $\newconstant\label{Constant:General_Concentration} > 0$ such that for all Lipschitz function $f$,
    \begin{equation}\label{Equation:General_Concentration}
        \mathbb{P}\Big(\Big|\int_{\mathbb{R}} f(\lambda)(\mu_{H_N} - \mathbb{E}[\mu_{H_N}])(d\lambda)\Big| \geq t\Big) \leq \oldconstant{Constant:General_Concentration}\exp\Big(-\frac{N^2t^2}{\oldconstant{Constant:General_Concentration}\|f\|_{\text{Lip}}^2}\Big).
    \end{equation}
\end{definition}

\begin{theorem}[{\protect\cite[Theorem 1.1]{guionnet2000concentration}}]\label{Theorem:GZ_Concentration}
    If $H_N = (h_{ij})_{1 \leq i, j \leq N}$ is an $N \times N$ symmetric random matrix such that the family $\{h_{ij}\}_{1\le i\le j\le N}$ is independent and each $h_{ij}$ satisfies a logarithmic Sobolev inequality with uniform constant $C_{\text{LSI}} > 0$, then $H_N$ satisfies \eqref{Equation:General_Concentration} with $\oldconstant{Constant:General_Concentration} = 8C_{\text{LSI}}$.
\end{theorem}

\subsection{Strong Wegner Estimate}

\begin{definition}[Strong Wegner Estimation]\label{Definition:Wegner}
    A sequence of random symmetric matrices $H_N$ is said to have \textbf{strong Wegner estimation} if there exist $\newconstant\label{Constant:General_Wegner} > 0$ such that for any interval $I \subseteq \mathbb{R}$,
    \begin{equation}\label{Equation:General_Wegner}
        \mathbb{E}(|\text{Spec}(H_N) \cap I|) \leq \oldconstant{Constant:General_Wegner}N|I|.
    \end{equation}
\end{definition}

\begin{proposition}[Moment Estimate]
    If $H_N$ is a sequence of random symmetric matrix satisfying \eqref{Equation:General_Wegner}, then for all $\epsilon > 0$, there exist $\newconstant\label{Constant:Wegner_Corollary} > 0$ such that
    \begin{equation}\label{Equation:Wegner_Corollary}
        \int_{\{|\lambda| < \eta\}} |\lambda|^{-\epsilon}\,\mathbb{E}\left[\mu_{H_N}\right](d\lambda) \leq \oldconstant{Constant:Wegner_Corollary}\eta^{1 - \epsilon}
    \end{equation}
\end{proposition}

\begin{proof}
    Observe that
    \begin{align}
        \int_{\{|\lambda| < \eta\}} |\lambda|^{-\epsilon}\,\mathbb{E}[\mu_{H_N}](d\lambda) &= \sum_{n = 0}^\infty \int_{\{2^{-(n + 1)}\eta \leq |\lambda| < 2^{-n}\eta\}} |\lambda|^{-\epsilon}\,\mathbb{E}[\mu_{H_N}](d\lambda) \notag \\
        &\leq \sum_{n = 0}^\infty \Bpar{\frac{\eta}{2^{n + 1}}}^{-\epsilon} \mathbb{E}\Bsq{\mu_{H_N}\Bpar{\Bsq{-\frac{\eta}{2^n}, \frac{\eta}{2^n}}}} \notag \\
        &\leq \sum_{n = 0}^\infty \Bpar{\frac{\eta}{2^{n + 1}}}^{-\epsilon} \cdot \frac{\oldconstant{Constant:General_Wegner}N \cdot 2^{-n}(2\eta)}{N} = \eta^{1 - \epsilon} \cdot 2\oldconstant{Constant:General_Wegner}\sum_{n = 0}^\infty \big(2^{-(1 - \epsilon)}\big)^n, \label{Equation:Wegner_Corollary_1}
    \end{align}
    where \eqref{Equation:Wegner_Corollary_1} holds by (strong) Wegner's estimate \eqref{Equation:General_Wegner}. Note that the latter term in \eqref{Equation:Wegner_Corollary_1} converges to a constant and therefore the corollary holds.
\end{proof}

\begin{corollary}[Truncation of Logarithm]\label{corollary:C.8}
    Let $H_N$ be a sequence of random symmetric matrices satisfying \eqref{Equation:General_Wegner}. Fix $0 < \epsilon < 1$, there exist $0 < \eta < 1$ small enough such that
    \begin{equation}\label{Equation:Condition_Wegner}
        -\log|\lambda| \leq |\lambda|^{-\epsilon} \quad \text{for all} \quad |\lambda| < \eta^{1/2}.
    \end{equation}
    Moreover, there exist universal constant $\newconstant\label{Constant:Truncation_Log} > 0$ such that
    \begin{equation}\label{Equation:Truncation_Log}
        \int_{\mathbb{R}} \left(\log_{\eta}(\lambda) - \log|\lambda|\right)\,\mathbb{E}\left[\mu_{H_N}\right](d\lambda) \leq \oldconstant{Constant:Truncation_Log}\big(\log(1 + \eta) + \eta^{(1 - \epsilon)/2}\big).
    \end{equation}
\end{corollary}

\begin{proof}
    Observe that
    \begin{align*}
        \int_{\{|\lambda| \geq \eta^{1/2}\}} \left(\log_\eta(\lambda) - \log|\lambda|\right)\mathbb{E}\left[\mu_{H_N}\right](d\lambda) = \dfrac{1}{2}\int_{\{|\lambda| \geq \eta^{1/2}\}} \log\Bpar{1 + \dfrac{\eta^2}{\lambda^2}} \mathbb{E}\left[\mu_{H_N}\right](d\lambda) \leq \frac{1}{2}\log(1 + \eta).
    \end{align*}
    Moreover, since \eqref{Equation:Condition_Wegner} is satisfied, we have
    \begin{align*}
        &\int_{\{|\lambda| < \eta^{1/2}\}} \left(\log_\eta(\lambda) - \log|\lambda|\right)\mathbb{E}\left[\mu_{H_N}\right](d\lambda) \\
        &\leq \int_{\{|\lambda| < \eta^{1/2}\}} \log_\eta(\lambda) \,\mathbb{E}\left[\mu_{H_N}\right](d\lambda) + \int_{\{|\lambda| < \eta^{1/2}\}} \dfrac{1}{|\lambda|^\epsilon} \,\mathbb{E}\left[\mu_{H_N}\right](d\lambda) \leq \dfrac{1}{2}\log\left(\eta + \eta^2\right) + \oldconstant{Constant:Wegner_Corollary}\eta^{(1 - \epsilon)/2},
    \end{align*}
    where the second inequality holds by \eqref{Equation:Wegner_Corollary}.
\end{proof}

\begin{theorem}[{\cite[Theorem 1]{aizenman2017matrix}}]\label{Theorem:GOE_Wegner_Estimate}
    Let $G_N$ and $A_N$ be a $N \times N$ GOE and a deterministic symmetric matrix, respectively. Then, the matrix $H_N = A_N + G_N$ satisfies \eqref{Equation:General_Wegner}, where  the constant $\oldconstant{Constant:General_Wegner}$ is independent of $A_N$.
\end{theorem}

\subsection{Matrix Dyson Equation}

\begin{theorem}[\protect{\cite[Proposition 2.1]{ajanki2019stability}}]\label{Theorem:Support_MDE}
    If $\mathcal{S} \colon \mathbb{C}^{N \times N} \to \mathbb{C}^{N \times N}$ is self-adjoint and positive preserving, then the solution $M_N \colon \mathbb{H} \to \mathbb{C}^{N \times N}$ to matrix Dyson equation (MDE)
    \begin{equation}\label{Equation:MDE}
        -M_N(z)^{-1} = -z \cdot I_N - A_N + \mathcal{S}[M_N(z)] \quad \text{subject to} \quad \Im M_N(z) > 0
    \end{equation}
    admits a Stieltjes transform representation
    \[\dfrac{1}{N}\text{Tr}(M_N(z)) = \int_{\mathbb{R}}\dfrac{\mu(d\lambda)}{\lambda - z}\]
    for some probability measure $\mu$, where $\text{supp}(\mu) \subseteq [-\theta, \theta]$, $\theta = \|A_N\|_{\text{op}} + 2\|\mathcal{S}\|_{\text{op} \to \text{op}}^{1/2}$.
\end{theorem}

\begin{theorem}[\protect{\cite[Proposition 2.2]{ajanki2019stability}}]\label{Theorem:Density_MDE}
    Let $\mathcal{S} \colon \mathbb{C}^{N \times N} \to \mathbb{C}^{N \times N}$ is self-adjoint, positive preserving. Moreover, assume $\mathcal{S}$ satisfies the following property
    \begin{enumerate}[label = (F)]
        \item\label{Condition:Flat} (Flatness) there exist universal $\newconstant\label{Constant:Flat} > 0$ such that
        \[\oldconstant{Constant:Flat}^{-1}N^{-1}\text{Tr}(M_N) \leq \mathcal{S}[M_N] \leq \oldconstant{Constant:Flat}N^{-1}\text{Tr}(M_N).\]
    \end{enumerate}
    Then there exist a universal constant $\beta > 0$ such that the induced probability measure given by Theorem \ref{Theorem:Support_MDE} from \eqref{Equation:MDE} admits $\beta$-H\"{o}lder continuous density with respect to the Lebesgue measure.
\end{theorem}

\subsection{Miscellaneous}

\begin{proposition}[Cauchy Interlacing Inequality {\cite[Corollary 7.3.6]{horn2012matrix}}]
    If $A$ is a $N \times N$ symmetric matrix and $B$ is a $N \times r$ ($r \leq N$) semi-orthogonal matrix (meaning $B^\top B = I_r$), then for all $i \in [r]$,
    \begin{equation}\label{Equation:Cauchy_Interlacing}
        \lambda_i(A) \leq \lambda_i(B^\top AB) \leq \lambda_{N - r + i}(A).
    \end{equation}
\end{proposition}

\begin{proposition}[{\cite[Theorem 3.1]{shebrawi2013trace}, \cite[Eq (2.1)]{Rohde2011}}]
    For square matrices $A_1, \dots, A_m \in \mathbb{R}^{n \times n}$ and $s > 0$, we have
    \begin{equation}\label{Equation:Subadditive_Trace}
        \text{Tr}\Big(\Big|\sum_{i = 1}^m A_i\Big|^s\Big) \leq m^s\sum_{i = 1}^m \text{Tr}\big(|A_i|^s\big),
    \end{equation}
    where $|A| = \sqrt{A^\top A}$.
\end{proposition}

\begin{proposition}[Determinant Breaking]\
    \begin{enumerate}[label = (\roman*)]
        \item {\cite[Eq (0.8.5.1)]{horn2012matrix}} Let $A \in \mathbb{R}^{n \times n}$, $B, C^\top \in \mathbb{R}^{n \times m}$, $D \in \mathbb{R}^{m \times m}$ and $D$ is invertible, then the determinant
        \begin{equation}\label{Equation:Det_Block_Matrix}
            \det\begin{pmatrix}
                A & B \\
                C & D
            \end{pmatrix} = \det(D)\det(A - BD^{-1}C)
        \end{equation}
        \item {\cite[Eq (0.8.5.11)]{horn2012matrix}} Let $A \in \mathbb{R}^{n \times n}$ be invertible, $v, w \in \mathbb{R}^n$, then
        \begin{equation}\label{Equation:Det_Rank_1}
            \det(A - vw^\top) = \det(A)(1 - w^\top A^{-1}v).
        \end{equation}
    \end{enumerate}
\end{proposition}

\begin{lemma}[\protect{\cite[Theorem A.43]{bai2010spectral}}]\label{Lemma:Rank_Inequality}
    Let $A$, $B$ be two $N \times N$ symmetric matrices, then
    \[d_{\text{KS}}(\mu_{A}, \mu_{B}) \leq \dfrac{1}{N}\,\text{rank}(A - B).\]
\end{lemma}

\begin{lemma}
    If $\mu$ and $\nu$ are random measures on $\mathbb{R}$, then
    \begin{equation}\label{Equation:E[KS]}
        d_{\text{KS}}(\mathbb{E}[\mu], \mathbb{E}[\nu]) \leq \mathbb{E}[d_{\text{KS}}(\mu, \nu)].
    \end{equation}
\end{lemma}

\begin{proof}
    This follows directly from the definition, 
    \begin{align*}
        d_{\text{KS}}(\mathbb{E}[\mu], \mathbb{E}[\nu]) &= \sup_{x \in \mathbb{R}} \big|\mathbb{E}\bsq{\mu((-\infty, x]) - \nu((-\infty, x])}\big| \\
        &\leq \mathbb{E}\Bsq{\sup_{x \in \mathbb{R}}|\mu((-\infty, x]) - \nu((-\infty, x])|} = \mathbb{E}[d_{\text{KS}}(\mu, \nu)],
    \end{align*}
    and we are done.
\end{proof}

\section{Asymptote of \texorpdfstring{$\mathbb{E}|\det H_N|$}{TEXT}}\label{Section:Asymptote_of_H_N}

\subsection{General Theory}

\noindent The following is a modified version of \cite[Theorem 1.2]{ben2023exponential}. It is not a direct corollary, since in their case one needs to assume that $\mathbb{E}\left[\mu_{H_N}\right]$ admits a bounded density near the origin and that the measures are supported on a common compact set. Instead, we assume the Wegner condition \eqref{Equation:General_Wegner} to establish the result. The proof is almost identical to that in \cite{ben2023exponential}, so we include it in the appendix for completeness.

\begin{theorem}\label{Theorem:Main_Asymptote_Unbounded}
    Let $H_N$ be a sequence of symmetric random matrix satisfying \eqref{Equation:General_Concentration} and \eqref{Equation:General_Wegner}, then
    \[\lim_{N \to \infty}\Big(\dfrac{1}{N}\log\mathbb{E}|\det H_N| - \int_{\mathbb{R}} \log|\lambda|\,\mathbb{E}[\mu_{H_N}](d\lambda)\Big) = 0.\]
\end{theorem}

\begin{proof}
    This is a combination of Lemma \ref{Lemma:Unbounded_Upper_Bound} and Lemma \ref{Lemma:Unbounded_Lower_Bound}.
\end{proof}

\begin{lemma}\label{Lemma:Unbounded_Upper_Bound}
    Let $H_N$ be a sequence of symmetric random matrices satisfying \eqref{Equation:General_Concentration} and \eqref{Equation:General_Wegner}, then
    \begin{equation}\label{Equation:Unbounded_Upper_Bound}
        \limsup_{N \to \infty}\Big(\dfrac{1}{N}\log\mathbb{E}|\det H_N| - \int_{\mathbb{R}} \log|\lambda|\,\mathbb{E}[\mu_{H_N}](d\lambda)\Big) \leq 0.
    \end{equation}
\end{lemma}

\begin{proof}
    Observe that we have directly
    \begin{align}
        &\dfrac{1}{N}\log\mathbb{E}|\det H_N| - \int_{\mathbb{R}} \log|\lambda| \,\mathbb{E}[\mu_{H_N}](d\lambda) \notag \\
        &\leq \dfrac{1}{N}\log\mathbb{E}\bsq{e^{N\int_{\mathbb{R}} \log_\eta(\lambda)\,\mu_{H_N}(d\lambda)}} - \int_{\mathbb{R}} \log|\lambda| \,\mathbb{E}[\mu_{H_N}](d\lambda) \notag \\
        &= \dfrac{1}{N}\log\mathbb{E}\bsq{e^{N\int_{\mathbb{R}} \log_\eta(\lambda)(\mu_{H_N} - \mathbb{E}[\mu_{H_N}])(d\lambda)}} + \int_{\mathbb{R}} \bpar{\log_\eta(\lambda) - \log|\lambda|}\mathbb{E}[\mu_{H_N}](d\lambda). \label{Equation:Unbounded_Upper_Bound_1}
    \end{align}
    Fix $0 < \epsilon < \frac{1}{2}$ and set the parameter $\eta = N^{-\epsilon}$. For the first term in \eqref{Equation:Unbounded_Upper_Bound_1}, observe that
    \begin{align*}
        \mathbb{E}\Big[e^{N\int_{\mathbb{R}} \log_\eta(\lambda)(\mu_{H_N} - \mathbb{E}[\mu_{H_N}])(d\lambda)}\Big] &= \int_0^\infty \mathbb{P}\Big(e^{N\int_{\mathbb{R}} \log_\eta(\lambda)(\mu_{H_N} - \mathbb{E}[\mu_{H_N}])(d\lambda)} \geq t\Big)\,dt \\
        &\leq 1 + \int_1^\infty \mathbb{P}\Big(\Big|\int_{\mathbb{R}} \log_\eta(\lambda)(\mu_{H_N} - \mathbb{E}[\mu_{H_N}])(d\lambda)\Big| \geq \frac{\log t}{N}\Big) \,dt \\
        &\leq 1 + \oldconstant{Constant:General_Concentration}\int_1^\infty \exp\Big(-\frac{\eta^2(\log t)^2}{4\oldconstant{Constant:General_Concentration}}\Big) \,dt \leq 1 + 2\sqrt{\frac{4\pi}{\eta^2}} \cdot \exp\Big(\frac{1}{\eta^2}\Big),
    \end{align*}
    where the second inequality holds by \eqref{Equation:General_Concentration}. Taking logarithm, dividing by $N$, and replace $\eta$ by $N^{-\epsilon}$, we obtain the inequality
    \begin{equation}\label{Equation:Unbounded_Upper_Bound_2}
        \dfrac{1}{N}\log\mathbb{E}\bsq{e^{N\int_{\mathbb{R}} \log_\eta(\lambda)(\mu_{H_N} - \mathbb{E}[\mu_{H_N}])(d\lambda)}} \leq \dfrac{1}{N}\log\bpar{1 + 2\sqrt{4\pi N^{2\epsilon}} \cdot \exp(N^{2\epsilon})} = o_{N \to \infty}(1). 
    \end{equation}
    For the second term in \eqref{Equation:Unbounded_Upper_Bound_1}, note that we may choose $N$ large enough so that \eqref{Equation:Condition_Wegner} holds. Then, we have by \eqref{Equation:Truncation_Log} that
    \begin{equation}\label{Equation:Unbounded_Upper_Bound_3}
        \int_{\mathbb{R}} \bpar{\log_{\eta}(\lambda) - \log|\lambda|}\,\mathbb{E}[\mu_{H_N}](d\lambda) \leq \oldconstant{Constant:Truncation_Log}\bpar{\log(1 + \eta) + \eta^{(1 - \epsilon)/2}} = o_{N \to \infty}(1).
    \end{equation}
    Combining \eqref{Equation:Unbounded_Upper_Bound_2} and \eqref{Equation:Unbounded_Upper_Bound_3}, we may then deduce \eqref{Equation:Unbounded_Upper_Bound}.
\end{proof}

For $\epsilon > 0$, set $\varphi = \varphi_N \colon \mathbb{R} \to \mathbb{R}$ be an even, smooth function such that
\begin{enumerate}[label = (\roman*)]
    \item $0 \leq \varphi(\lambda) \leq 1$ for all $\lambda \in \mathbb{R}$,
    \item $\varphi(\lambda) = 1$ if $\lambda \in \bsq{-N^{-3\epsilon}, N^{-3\epsilon}}$,
    \item $\text{supp}(\varphi) \subseteq \bsq{-2N^{-3\epsilon}, 2N^{-3\epsilon}}$.
\end{enumerate}

\begin{lemma}\label{Lemma:Unbounded_Lower_Bound}
    Let $H_N$ be a sequence of symmetric random matrices satisfying \eqref{Equation:General_Concentration} and \eqref{Equation:General_Wegner}, then
    \begin{equation}\label{Equation:Unbounded_Lower_Bound}
        \liminf_{N \to \infty} \Big(\dfrac{1}{N}\log\mathbb{E}|\det H_N| - \int_{\mathbb{R}} \log|\lambda|\,\mathbb{E}[\mu_{H_N}](d\lambda)\Big) \geq 0.
    \end{equation}
\end{lemma}

\begin{proof}
    Fix $\epsilon > 0$ and set the parameter $\eta = N^{-4\epsilon}$. Define the events
    \begin{align*}
        \mathcal{E}_{\text{Gap}} &= \big\{\text{Spec}(H_N) \cap \big[-e^{-N^\epsilon}, e^{-N^\epsilon}\big] = \varnothing\big\}, \\
        \mathcal{E}_{\text{Lip}} &= \Big\{\Big|\int_{\mathbb{R}} \log_\eta(\lambda)\left(\mu_{H_N} - \mathbb{E}\left[\mu_{H_N}\right]\right)(d\lambda)\Big| < N^{-\epsilon}\Big\}, \\
        \mathcal{E}_{\varphi} &= \Big\{\int_{\mathbb{R}} \varphi(\lambda) \mu_{H_N}(d\lambda) < N^{-2\epsilon}\Big\}.
    \end{align*}
    Observe that we have directly
    \begin{align*}
        &\dfrac{1}{N}\log\mathbb{E} |\det H_N| - \int_{\mathbb{R}} \log|\lambda|\,\mathbb{E}[\mu_{H_N}](d\lambda) \\
        &= \dfrac{1}{N}\log \mathbb{E}\bsq{e^{N\int_{\mathbb{R}} (\log|\lambda| - \log_\eta(\lambda))\mu_{H_N}(d\lambda) + N\int_{\mathbb{R}} \log_\eta(\lambda) (\mu_{H_N} - \mathbb{E}[\mu_{H_N}])(d\lambda)}} \\
        &\hspace{1cm} + \int_{\mathbb{R}} \bpar{\log_\eta(\lambda) - \log|\lambda|}\,\mathbb{E}[\mu_{H_N}](d\lambda) \\
        &\geq \dfrac{1}{N}\log\mathbb{E}\bsq{e^{N\int_{\mathbb{R}} (\log|\lambda| - \log_\eta(\lambda))\mu_{H_N}(d\lambda)}\mathbbm{1}_{\mathcal{E}_{\text{Gap}} \cap \mathcal{E}_{\text{Lip}} \cap \mathcal{E}_{\varphi}}} - N^{-\epsilon}.
    \end{align*}
    since $\log_\eta(\lambda) - \log|\lambda| \geq 0$. Now, observe that if $|\lambda| \geq A$, then
    \begin{equation}\label{Equation:Unbounded_Lower_Bound_1}
        \log|\lambda| - \log_\eta(\lambda) = -\frac{1}{2}\log\Bpar{1 + \frac{\eta^2}{\lambda^2}} \geq -\frac{1}{2}\log\left(1 + A^{-2}N^{-8\epsilon}\right).
    \end{equation}
    We may then deduce
    \begin{align}
        &\mathbb{E}\bsq{e^{N\int_{\mathbb{R}} (\log|\lambda| - \log_\eta(\lambda))\mu_{H_N}(d\lambda)}\mathbbm{1}_{\mathcal{E}_{\text{Gap}} \cap \mathcal{E}_{\text{Lip}} \cap \mathcal{E}_{\varphi}}} \notag \\
        &= \mathbb{E}\bsq{e^{N\int_{\mathbb{R}} \varphi(\lambda)(\log|\lambda| - \log_\eta(\lambda))\mu_{H_N}(d\lambda)} \cdot e^{N\int_{\mathbb{R}} (1 - \varphi(\lambda))(\log|\lambda| - \log_\eta(\lambda))\mu_{H_N}(d\lambda)}\mathbbm{1}_{\mathcal{E}_{\text{Gap}} \cap \mathcal{E}_{\text{Lip}} \cap \mathcal{E}_{\varphi}}} \notag \\
        &\geq \mathbb{E}\bsq{e^{-\frac{N}{2}\log(1 + e^{2N^\epsilon}N^{-8\epsilon})\int_{\mathbb{R}} \varphi(\lambda) \mu_{H_N}(d\lambda)} \cdot e^{N\int_{\{|\lambda| \geq N^{-3\epsilon}\}} (\log|\lambda| - \log_\eta(\lambda))\mu_{H_N}(d\lambda)}\mathbbm{1}_{\mathcal{E}_{\text{Gap}} \cap \mathcal{E}_{\text{Lip}} \cap \mathcal{E}_{\varphi}}} \label{Equation:Unbounded_Lower_Bound_2} \\
        &\geq e^{-N^{1 - 2\epsilon}\log(1 + e^{2N^\epsilon}N^{-8\epsilon})/2} \cdot e^{-N\log(1 + N^{-2\epsilon})/2} \cdot \mathbb{P}\left(\mathcal{E}_{\text{Gap}} \cap \mathcal{E}_{\text{Lip}} \cap \mathcal{E}_{\varphi}\right), \label{Equation:Unbounded_Lower_Bound_3}
    \end{align}
    where \eqref{Equation:Unbounded_Lower_Bound_2} holds by setting $A = e^{-N^\epsilon}$ in \eqref{Equation:Unbounded_Lower_Bound_1} (which is valid on $\mathcal{E}_{\text{Gap}}$) and $1 - \varphi(\lambda) \leq \mathbbm{1}_{\left\{|\lambda| \geq N^{-3\epsilon}\right\}}$; \eqref{Equation:Unbounded_Lower_Bound_3} holds by setting $A = N^{-3\epsilon}$ in \eqref{Equation:Unbounded_Lower_Bound_1} and $\mathcal{E}_\varphi$. Taking logarithm and dividing by $N$, we obtain the following inequality
    \begin{align}
        &\dfrac{1}{N}\log\mathbb{E}\left|\det H_N\right| - \int_{\mathbb{R}} \log_{\eta}(\lambda)\,\mu_{H_N}(d\lambda) \notag\\
        &\hspace{1cm} \geq -\frac{N^{-2\epsilon}}{2}\log\left(1 + e^{2N^\epsilon}N^{-8\epsilon}\right) - \frac{1}{2}\log\left(1 + N^{-2\epsilon}\right) + \frac{1}{N}\log\mathbb{P}\left(\mathcal{E}_{\text{Gap}} \cap \mathcal{E}_{\text{Lip}} \cap \mathcal{E}_{\varphi}\right) - N^{-\epsilon} \notag \\
        &\hspace{1cm} \geq o_{N \to \infty}(1) + \frac{1}{N}\log\Big[1 - \mathbb{P}\big(\mathcal{E}_{\text{Gap}}^\complement\big) - \mathbb{P}\big(\mathcal{E}_{\text{Lip}}^\complement\big) - \mathbb{P}\big(\mathcal{E}_{\varphi}^\complement\big)\Big] \label{Equation:Unbounded_Lower_Bound_4}
    \end{align}
    Next, we show that the probability of these events converges to $0$. By (strong) Wegner's estimate \eqref{Equation:General_Wegner}, we have
    \begin{equation}\label{Equation:Unbounded_Lower_Bound_5}
        \mathbb{P}\big(\mathcal{E}_{\text{Gap}}^\complement\big) = \mathbb{P}\bpar{\text{Spec}(H_N) \cap \bsq{-e^{-N^\epsilon}, e^{-N^\epsilon}} \neq \varnothing} \leq \oldconstant{Constant:General_Wegner}N \cdot e^{-N^\epsilon} = o_{N \to \infty}(1).
    \end{equation}
    By concentration \eqref{Equation:General_Concentration}, we have
    \begin{equation}
        \mathbb{P}\big(\mathcal{E}_{\text{Lip}}^\complement\big) = \mathbb{P}\Big(\Big|\int_{\mathbb{R}} \log_\eta(\lambda)(\mu_{H_N} - \mathbb{E}[\mu_{H_N}])(d\lambda)\Big| \geq N^{-\epsilon} \Big) \leq \oldconstant{Constant:General_Concentration}e^{-N^{2 - 10\epsilon}/(4\oldconstant{Constant:General_Concentration})} = o_{N \to \infty}(1). \label{Equation:Unbounded_Lower_Bound_6}
    \end{equation}
    By Markov's inequality and (strong) Wegner's estimate \eqref{Equation:General_Wegner}, we have
    \begin{align}
        \mathbb{P}\big(\mathcal{E}_\varphi^\complement\big) &= \mathbb{P}\Big(\int_{\mathbb{R}} \varphi(\lambda) \mu_{H_N}(d\lambda) \geq N^{-2\epsilon}\Big) \notag \\
        &\leq N^{2\epsilon}\mathbb{E}\Big[\int_{\mathbb{R}} \varphi(\lambda) \mu_{H_N}(d\lambda)\Big] \notag \\
        &\leq N^{2\epsilon}\mathbb{E}\Big[\mu_{H_N}\left(\left[-2N^{-3\epsilon}, 2N^{-3\epsilon}\right]\right)\Big] \leq N^{2\epsilon} \cdot \frac{\oldconstant{Constant:General_Wegner}N \cdot 2N^{-3\epsilon}}{N} = 2\oldconstant{Constant:General_Wegner}N^{-\epsilon} = o_{N \to \infty}(1). \label{Equation:Unbounded_Lower_Bound_7}
    \end{align}
    By plugging \eqref{Equation:Unbounded_Lower_Bound_5}, \eqref{Equation:Unbounded_Lower_Bound_6}, and \eqref{Equation:Unbounded_Lower_Bound_7} through \eqref{Equation:Unbounded_Lower_Bound_4}, we may derive \eqref{Equation:Unbounded_Lower_Bound}.
\end{proof}

\begin{theorem}[\protect{\cite[Corollary 1.9A]{ben2023exponential}}]\label{Theorem:Main_Asymptote_Bounded}
    Let $H_N$ be a sequence of symmetric Gaussian matrices with independent entries up to symmetry. Set $A_N = \mathbb{E}[H_N]$ and $H_N = A_N + W_N$. Moreover, let $\mu_N$ be the solution of the MDE \eqref{Equation:MDE} with
    \[\mathcal{S}(M_N) = \mathbb{E}[W_NM_NW_N].\]
    Assume $\mathcal{S}$ satisfies \ref{Condition:Flat}. Then, for all $K > 0$, we have
    \begin{equation}\label{Equation:Main_Asymptote_Bounded}
        \lim_{N \to \infty} \sup_{\|A_N\|_{\text{op}} \leq K} \Big(\frac{1}{N}\log\mathbb{E}\left|\det H_N\right| - \int_{\mathbb{R}} \log|\lambda|\,\mu_N(d\lambda)\Big) = 0.
    \end{equation}
\end{theorem}

\begin{remark}
    In the original theorem of \cite{ben2023exponential}, the authors do not explicitly specify that the convergence is uniform over matrices $A_N$ with uniformly bounded operator norm. However, upon closer inspection of their proof, one finds that the convergence rate depends only on the supremum of the operator norm.
\end{remark}

\subsection{Asymptote for \texorpdfstring{$A_N + \text{GOE}_N$}{TEXT}}

\noindent Throughout this entire section, we fix $G_N$ to be a sequence of GOE.

\begin{theorem}\label{Theorem:Main_Asymptote_Unbounded_GOE}
    We have
    \[\lim_{N \to \infty}\sup_{A_N \in \mathbb{R}^{N \times N}, \text{ symmetric}}\Big(\dfrac{1}{N}\log\mathbb{E} |\det (A_N + G_N)| - \int_{\mathbb{R}} \log|\lambda|\,\mathbb{E}[\mu_{A_N + G_N}](d\lambda)\Big) = 0.\]
\end{theorem}

\begin{proof}
    It suffices to verify the conditions of Theorem \ref{Theorem:Main_Asymptote_Unbounded}. Note that the concentration inequality \eqref{Equation:General_Concentration} holds with by Theorem \ref{Theorem:GZ_Concentration} (the entries of $G_N$ satisfy the logarithmic Sobolev inequality with uniform constant $2/N$). Moreover, the Wegner estimate \eqref{Equation:General_Wegner} holds by Theorem \ref{Theorem:GOE_Wegner_Estimate}. Since these estimations do not depend on the mean $A_N$, the convergence is uniform.
\end{proof}

\vspace{0.3cm}

From now on, we assume the deterministic symmetric matrices $A_N$ have uniformly bounded operator norm
\[\sup_{N \in \mathbb{N}} \|A_N\|_{\text{op}} < \infty.\]

\begin{definition}[Solution of MDE]
    Let $\mu_N$ be the measure obtained from the MDE
    \[-M_N^{-1}(z) = zI_N - A_N +  \frac{M_N(z)^\top + (\text{Tr}\,M_N(z))I_N}{N} \quad \text{subject to} \quad \Im M_N(z) > 0\]
    and $\overline{\mu}_N$ be the measure obtained from the MDE
    \begin{equation}\label{Equation:Nice_MDE}
        -M_N^{-1}(z) = zI_N - A_N + \frac{\text{Tr}\,M_{N}(z)}{N}I_N \quad \text{subject to} \quad \Im M_N(z) > 0.
    \end{equation}
\end{definition}

\begin{proposition}[Stability of MDE]\label{Proposition:Stability_GOE}\
    \begin{enumerate}[label = (\roman*)]
        \item (\cite[Lemma 5.1]{ben2024landscape}) The measure 
        \[\overline{\mu}_N = \mu_{A_N} \boxplus \mu_{\text{sc}}.\]
        Moreover, it admits a bounded and compactly supported density.
        \item (\cite[Proof of Proposition 5.3]{ben2024landscape}) For all $K > 0$, there exist constant $\epsilon > 0$ such that the Wasserstein distance 
        \[\sup_{\|A_N\|_{\text{op}} \leq K} W_1\left(\mu_N, \overline{\mu}_N\right) < N^{-\epsilon}.\]
    \end{enumerate}
\end{proposition}

\begin{remark}
    Again, in the original theorem of \cite{ben2024landscape}, the authors do not explicitly specify that the convergence of the Wasserstein-$1$ distance is uniform over matrices $A_N$ with uniformly bounded operator norm. Similarly, upon closer inspection of their proof, one finds that the convergence rate depends only on the supremum of the operator norm.
\end{remark}

\begin{theorem}[Asymptote for Bounded Mean]\label{Theorem:Asymptote_Bounded}
    Set $H_N = A_N + G_N$. Then for all $K > 0$,
    \[\lim_{N \to \infty}\sup_{\|A_N\|_{\text{op}} \leq K} \Bpar{\dfrac{1}{N}\log\mathbb{E}\left|\det H_N\right| - \int_{\mathbb{R}} \log|\lambda|\,\left(\mu_{A_N} \boxplus \mu_{\text{sc}}\right)(d\lambda)} = 0.\]
\end{theorem}

\begin{proof}
    First, note that the operator
    \[\mathcal{S}(M_N) = \mathbb{E}[G_NM_NG_N] = \frac{M_N^\top + (\text{Tr}M_N)I_N}{N}\]
    satisfies the flatness condition \ref{Condition:Flat} and $\|A_N\|_{\text{op}}$ is uniformly bounded by assumption. Therefore, by Theorem \ref{Theorem:Main_Asymptote_Bounded}, we have
    \begin{equation}\label{Equation:Asymptote_Bounded_1}
        \lim_{N \to \infty}\sup_{\|A_N\|_{\text{op}} \leq K} \Bpar{\dfrac{1}{N}\log\mathbb{E}\left|\det H_N\right| - \int_{\mathbb{R}} \log|\lambda|\,\mu_N(d\lambda)} = 0.
    \end{equation}
    Now, note that the difference
    \begin{align}
        \left|\int_{\mathbb{R}} \log|\lambda| (\overline{\mu}_N - \mu_N)(d\lambda)\right| &\leq \int_{\mathbb{R}} (\log_\eta(\lambda) - \log|\lambda|) (\overline{\mu}_N + \mu_N)(d\lambda) + \left|\int_{\mathbb{R}} \log_\eta(\lambda) (\overline{\mu}_N - \mu_N)(d\lambda)\right| \notag \\
        &\leq \int_{\mathbb{R}} \dfrac{1}{2}\log\Bpar{1 + \frac{\eta^2}{\lambda^2}}(f_{\overline{\mu}_N}(\lambda) + f_{\mu_N}(\lambda))\,d\lambda + \dfrac{1}{2\eta}W_1(\overline{\mu}_N, \mu_N), \label{Equation:Asymptote_Bounded_2}
    \end{align}
    where $f_{\overline{\mu}_N}$ and $f_{\mu_N}$ are the densities of $\overline{\mu}_N$ and $\mu_N$, respectively. For the first term in \eqref{Equation:Asymptote_Bounded_1}, we ought to show that $\overline{\mu}_N$ and $\mu_N$ admit uniformly bounded density. By Theorem \ref{Theorem:Density_MDE}, $\overline{\mu}_N$ admits $\beta$-H\"{o}lder continuous density with respect to Lebesgue measure by for some universal constant $\beta > 0$. Moreover, by Theorem \ref{Theorem:Support_MDE}, we know that 
    \[m_\infty(\overline{\mu}_N) \leq \|A_N\|_{\text{op}} + 2\sqrt{2} \leq K + 2\sqrt{2}.\]
    Combining these results, we see that $\overline{\mu}_N$ admits a uniformly bounded density that is supported on a common compact set (where everything only depends on $K$). We have the same result for $\mu_N$ by Theorem \ref{Theorem:Free_Convolution_Semi_Circle}. Therefore, there exist $\newconstant\label{Constant:Alpha} = \oldconstant{Constant:Alpha}(K) > 0$ such that
    \[\int_{\mathbb{R}} \log\Big(1 + \frac{\eta^2}{\lambda^2}\Big)(f_{\nu_n}(\lambda) + f_{\nu_\infty}(\lambda))\,d\lambda \leq 2\oldconstant{Constant:Alpha}\int_{-\oldconstant{Constant:Alpha}}^{\oldconstant{Constant:Alpha}} \log\Big(1 + \frac{\eta^2}{\lambda^2}\Big)\,d\lambda,\]
    which converges to $0$ as $\eta \to 0$ by the dominated convergence theorem. Therefore, by Proposition \ref{Proposition:Stability_GOE},
    \[\sup_{\|A_N\|_{\text{op}} \leq K}\Big|\int_{\mathbb{R}} \log|\lambda| (\overline{\mu}_N - \mu_N)(d\lambda)\Big| \leq o_{\eta \to 0}(1) + \frac{o_{N \to \infty}(1)}{\eta}.\]
    By picking $\eta > 0$ small enough so the first term is less than $\epsilon/2$, then choosing $N \in \mathbb{N}$ large enough so that the second term is less than $\epsilon/2$, the whole thing is less than $\epsilon$. We see then
    \begin{equation}\label{Equation:Asymptote_Bounded_3}
        \lim_{N \to \infty} \sup_{\|A_N\|_{\text{op}} \leq K}\Big|\int_{\mathbb{R}} \log|\lambda| (\overline{\mu}_N - \mu_N)(d\lambda)\Big| = 0.
    \end{equation}
    Combining \eqref{Equation:Asymptote_Bounded_1} and \eqref{Equation:Asymptote_Bounded_3}, we see that our assertion holds.
\end{proof}

\section{Varadhan's Lemma}

\noindent In the following, we will denote $\mathcal{X}$ to be a regular topological space and $\mu_N$ be a sequence of probability measures on $\mathcal{X}$.

\begin{theorem}
    Suppose $\mu_N$ satisfies the LDP with a good rate function $J \colon \mathcal{X} \to \mathbb{R}$ and let $\Phi \colon \mathcal{X} \to \mathbb{R}$ be any continuous function, then
    \begin{enumerate}[label = (\roman*)]
        \item For all open sets $U \subseteq \mathcal{X}$, we have
        \begin{equation}\label{Equation:Varadhan_Original_Lower}
            \liminf_{N \to \infty} \frac{1}{N}\log \mathbb{E}_{\mu_N}\bigl[e^{N\Phi(X_N)}\mathbbm{1}_{\{X_N \in U\}}\bigr] \geq \sup_{x \in U} [\Phi(x) - J(x)].
        \end{equation}
        \item Denote $S = \bigcup_{N \in \mathbb{N}} \text{supp}(\mu_N) \subseteq \mathcal{X}$. For all closed sets $C \subseteq \mathcal{X}$, if we assume the additional condition that
        \begin{equation}\label{Equation:Varadhan_Original_Condition}
            M = \sup_{x \in S \cap C} \Phi(x) < \infty,
        \end{equation}
        then we have
        \begin{equation}\label{Equation:Varadhan_Original_Upper}
            \limsup_{N \to \infty} \frac{1}{N}\log \mathbb{E}_{\mu_N}\bigl[e^{N\Phi(X_N)}\mathbbm{1}_{\{X_N \in C\}}\bigr] \leq \sup_{x \in C} [\Phi(x) - J(x)].
        \end{equation}
    \end{enumerate}
\end{theorem}

\begin{proof}
    We prove the lower bound \eqref{Equation:Varadhan_Original_Lower} first. Note that for all $x \in U$ and $\delta > 0$, there exists a neighborhood $U_x \subseteq U$ of $x$ such that $\inf_{y \in U_x} \Phi(y) \geq \Phi(x) - \delta$ by continuity of $\Phi$. Then,
    \begin{align*}
        \liminf_{N \to \infty} \frac{1}{N}\log \mathbb{E}_{\mu_N}\bigl[e^{N\Phi(X_N)}\mathbbm{1}_{\{X_N \in U\}}\bigr] &\geq \liminf_{N \to \infty} \frac{1}{N}\log \mathbb{E}_{\mu_N}\bigl[e^{N\Phi(X_N)}\mathbbm{1}_{\{X_N \in U_x\}}\bigr] \\
        &\geq \Phi(x) - \delta + \liminf_{N \to \infty} \mu_N(U_x) \\
        &\geq \Phi(x) - \delta - \inf_{y \in U_x} J(y) \\
        &\geq \Phi(x)  - \delta - J(x),
    \end{align*}
    where the third inequality follows from LDP. \eqref{Equation:Varadhan_Original_Lower} then follows since $x \in U$ and $\delta > 0$ are arbitrary.
    
    For the upper bound \eqref{Equation:Varadhan_Original_Upper}, we define the continuous bounded function $\Phi_M(x) = \Phi(x) \wedge M$ so that from \eqref{Equation:Varadhan_Original_Condition},
    \[\mathbb{E}_{\mu_N}\bigl[e^{N\Phi(X_N)}\mathbbm{1}_{\{X_N \in C\}}\bigr] = \mathbb{E}_{\mu_N}\bigl[e^{N\Phi_M(X_N)}\mathbbm{1}_{\{X_N \in C\}}\bigr].\]
    By \cite[Exercise 4.3.11]{dembo2010large}, we may apply the Varadhan upper bound to the latter term to get
    \begin{align*}
        \limsup_{N \to \infty} \frac{1}{N}\log \mathbb{E}_{\mu_N}\bigl[e^{N\Phi(X_N)}\mathbbm{1}_{\{X_N \in C\}}\bigr] &= \limsup_{N \to \infty} \frac{1}{N}\log \mathbb{E}_{\mu_N}\bigl[e^{N\Phi_M(X_N)}\mathbbm{1}_{\{X_N \in C\}}\bigr] \\
        &\leq \sup_{x \in C} [\Phi_M(x) - J(x)] \\
        &\leq \sup_{x \in C} [\Phi(x) - J(x)].
    \end{align*}
    This completes our proof.
\end{proof}

\end{appendices}

\bibliographystyle{acm}

{\footnotesize\bibliography{reference}}

\end{document}